\pdfoutput=1
\documentclass[intlimits, sumlimits, titlepage]{amsart}
\usepackage[utf8]{inputenc}
\usepackage[T1]{fontenc}
\usepackage{amsaddr,lmodern, csquotes, graphicx}
\usepackage{amssymb, amsthm, physics, bm}
\usepackage[subsetnonamb, bbsets]{jkmath}

\usepackage[backend=biber,style=alphabetic,sorting=nyt,maxbibnames=4]{biblatex}
\newtheorem{theorem}{Theorem}[section]
\newtheorem{lemma}[theorem]{Lemma}
\newtheorem{proposition}[theorem]{Proposition}
\newtheorem{corollary}[theorem]{Corollary}
\theoremstyle{definition}
\newtheorem{definition}[theorem]{Definition}
\newtheorem{example}[theorem]{Example}
\theoremstyle{remark}
\newtheorem{notation}[theorem]{Notation}
\newtheorem{remark}[theorem]{Remark}

\DeclareMathOperator{\ct}{ct}
\DeclareMathOperator{\im}{Im}

\DeclareMathOperator{\lot}{l.o.t.}
\DeclareMathOperator{\st}{S}
\DeclareMathOperator{\ko}{K}

\begin{document}
\title{Non-Symmetric Askey--Wilson Shift Operators}
\date{\today}
\author[Van Horssen]{Max van Horssen}
\address{Department of Mathematics, KU Leuven, Celestijnenlaan 200B,\\ 3001 Leuven, Belgium}
\email{max.vanhorssen@kuleuven.be}
\author[Schlösser]{Philip Schlösser}
\address{IMAPP-Mathematics, Radboud University, Heyendaalseweg 135,\\ 6525 AJ Nijmegen, the Netherlands}
\email{philip.schloesser@ru.nl}
\subjclass[2020]{33D45, 20C08}
\keywords{Askey--Wilson polynomials, shift operators, double-affine Hecke algebra}

\begin{abstract}
    We classify the shift operators for the symmetric Askey--Wilson
    polynomials and construct shift operators for the non-symmetric
    Askey--Wilson polynomials using two decompositions of
    non-symmetric Askey--Wilson polynomials in terms of
    symmetric ones.
    These shift operators are difference-reflection operators, and we discuss the conditions under which they restrict to shift operators
    for the symmetric Askey--Wilson polynomials.
    We use them to compute the norms of the non-symmetric Askey--Wilson polynomials and compute their specialisations for $q\to1$. These
    turn out to be shift operators for the non-symmetric Heckman--Opdam polynomials of type $BC_1$ that have recently been found.
\end{abstract}

\maketitle

\section{Introduction}
Shift operators for families of orthogonal polynomials can be traced back to \cite{schroedinger,infeld} and the systematic treatment in
\cite{inui1,inui2}. When initially presenting their family of
orthogonal polynomials in \cite{AsWi85}, Askey and Wilson also examine possible shift operators and define the forward and
backward shift operators in \cite[Section~5]{AsWi85}, where they also derive the \enquote{second-order} difference eigenvalue equation. And while Askey and Wilson did not use the shift operators for systematic study of their polynomials, it was later shown
by Kalnins and Miller (in \cite{KaMi89}) that these shift operators together with two more shift operators discovered by Kalnins and Miller can be used to compute the norms
of the symmetric Askey--Wilson polynomials (in the following we will refer to them as \emph{AW-polynomials}).

Afterwards, shift operators have proven useful in determining norms.
In \cite{Opd89}, having previously found (\cite{OpdIII}) and proved the general existence (\cite{OpdIV}) of shift operators for the symmetric
Heckman--Opdam polynomials (henceforth referred to as
\emph{HO-polynomials}), Opdam uses these shift operators to compute
the norms of the HO-polynomials. The construction of the shift operators
(whose existence had previously been proven using asymptotic methods)
was shown to be possible using elementary methods in \cite{He91}. Combining this with the degenerate affine Hecke algebra approach pioneered in \cite{Ch89}, made it possible for Cherednik to adapt these
constructions to the non-degenerate case, and therefore the case of
Macdonald polynomials, in \cite{Ch95}.

To the extent of the authors' knowledge, the possibility of obtaining similar shift operators (now as
difference-reflection, or differential-reflection operators) for
the nonsymmetric versions (be they AW-polynomials, HO-polynomials, or
Macdonald polynomials) has not been considered so far in the literature. 
There are, however, two recent exceptions to this. During a talk in 2023, Opdam announced an upcoming publication together with Laredo on the existence of shift operators for the nonsymmetric HO-polynomials, see \cite[Section 8.4]{OpLa24} for more details. In \cite{vHvP24}, Van Pruijssen and one of the authors proved the existence and classification explicitly for the rank-one HO-polynomials.
In this article,
we will consider this question for the case of nonsymmetric AW-polynomials.

A note on terminology: throughout this article, we will refer to shift
operators for the symmetric AW-polynomials as \emph{symmetric shift operators} or just \emph{shift operators}, and to shift operators for
the non-symmetric AW-polynomials as \emph{non-symmetric shift operators}.

After discussing the setup, in Section~\ref{sec-symmetric} we define the notion of a shift operator, prove
some of their general properties, and classify them. Additionally, we introduce two notions of formal adjoints of shift operators. The classification relies crucially on the shift operators that were introduced in \cite{AsWi85} and \cite{KaMi89}.

In Section~\ref{sec:nonsym-shift-op} we then define the
notion of a non-symmetric shift operator in analogy with an
equivalent characterisation from Section~\ref{sec-symmetric}.
We choose to study these non-symmetric shift operators
by considering
$\mathcal{A}$, the vector space of Laurent polynomials in $z$, as 
a module over $\mathcal{A}_0$, its subring of invariants under the
reflection $s_1$ that maps $z^n$ to $z^{-n}$ (alternatively: polynomials
in $z+z^{-1}$). This is a free module, and we consider two different bases:
the bases $\{1,z\}$ (referred to as \emph{Steinberg basis} since it can
be obtained, up to choice of positive subsystem, from the general theory
explained in \cite{Ste75}), and the basis
\[
 1,\quad z^{-1}(1-az)(1-bz)
\]
(referred to as \emph{Koornwinder basis} since it was first considered by Koornwinder and Bouzeffour in \cite{KoBo11}).

With respect to these bases we write the non-symmetric AW-polynomials
as column vectors that we use to construct matrix polynomials whose entries turn out to be easily related to symmetric AW-polynomials.
Using the symmetric shift operators from Section~\ref{sec-symmetric}, 
we are then able to first construct matrix shift operators 
(for the matrix polynomials), and then non-symmetric AW-polynomials.
On the vector level we are also able to extract a Rodrigues-type recursion 
relation, in analogy with \cite[\S6.5.14]{Mac03}. Since the symmetric 
shift operators are classified, so are the matrices of symmetric shift 
operators. It has, however, proven difficult
to classify all matrices of symmetric shift operators that
(in either basis) give rise to (scalar) non-symmetric shift
operators. We therefore restrict ourselves to those that are diagonal.
This way we already obtain one non-zero operator for each possible
shift, which is more than sufficient for the norm computation in the
following section.

In Section~\ref{sec:norms} we use the non-symmetric shift operators
to compute the norms of the non-symmetric AW-polynomials in a way similar to \cite{KaMi89}. And finally, in Section~\ref{sec:specialisation}, we compute the $q\to1$ limits of our
non-symmetric shift operators and obtain four first-order differential-reflection operators and two (zeroth-order) constants. These operators agree with the (non-symmetric) shift operators for the non-symmetric HO-polynomials of type $BC_1$ that are described in \cite{vHvP24}.

While we make use of matrix-valued orthogonal polynomial (MVOP) methods that are decidedly single-variable to establish the relation between
the symmetric AW-polynomials and the matrix polynomials that arise from the non-symmetric AW-polynomials, similar methods might work in higher ranks if we use the partial ordering on monomials defined in \cite[\S2.7]{Mac03}. One condition our method does rely on, however, is that we can find
a basis, so that the corresponding matrix weight (see \cite[Lemma~6.1]{Sch23})
is similar to a diagonal matrix via a constant matrix. As is shown in \cite[Lemma~6.20]{Sch23}, this can fail to be the case even in rank 2.

\newpage

\section{Setup}
We adopt the conventions as in \cite{Mac03}. Let $S$ be an affine root system of type $(C_1^\vee, C_1)$, i.e.
\[
    S = \set{\pm\epsilon + \frac{r}{2}c,\pm 2\epsilon +rc\where r\in\Z},
\]
where $c$ is the constant affine function and $\epsilon$ a non-constant
linear function. We write $a_0 := \frac{c}{2}-\epsilon$ and $a_1 := \epsilon$. Let $\Lambda$ be the lattice spanned by $a_0,a_1$ and write $e^{a_0},e^{a_1}$ for the corresponding elements of the group algebra of $\Lambda$ over $\Z$ or any other commutative ring.

\begin{definition}\label{sec:def-nabla-delta}\leavevmode
\begin{enumerate}
    \item Let $\mathcal{R}$ be the ring of Laurent polynomials in
    $a,b,cq^{-\frac{1}{2}},dq^{-\frac{1}{2}}$, i.e.
    \[
        \mathcal{R} = \Q\qty[a^{\pm1},b^{\pm1},c^{\pm1}q^{\mp\frac{1}{2}},d^{\pm1}q^{\mp\frac{1}{2}}].
    \]
    \item Define the \emph{non-symmetric} weight $\Delta$ and the \emph{symmetric} weight $\nabla$ by
    \[
        \Delta := \prod_{x\in S^+}\frac{1-t_{2x}^{1/2}e^x}{1-t_xt_{2x}^{1/2}e^x},\qquad
        \nabla := \prod_{\substack{x\in S\\x(0)\ge0}}
        \frac{1-t_{2x}^{1/2}e^x}{1-t_xt_{2x}^{1/2}e^x},
    \]
    which are elements of the power series ring $\mathcal{R}[[e^{a_0},e^{a_1}]]$. We use
    \[
        t_{2x}^{1/2} = \begin{cases}
            -b & x\in \pm \epsilon + \Z\\
            -q^{-\frac{1}{2}}d & x\in \pm\epsilon + \frac{1}{2}+\Z\\
            1 & \text{otherwise}
        \end{cases}\quad\text{and}\quad 
        t_xt_{2x}^{1/2} = \begin{cases}
            a & x\in \pm\epsilon + \Z\\
            b^2 & x\in \pm2\epsilon + 2\Z\\
            q^{-\frac{1}{2}}c & x\in \pm\epsilon + \frac{1}{2}+\Z\\
            q^{-1}d^2 & x\in \pm2\epsilon + 1 + 2\Z
        \end{cases}.
    \]
\end{enumerate}
\end{definition}
\begin{remark}
    Note that for the two negative affine roots contained in the product
    defining $\nabla$, namely $x=-\epsilon, -2\epsilon$, we have
    \[
        \frac{1-t_x^{1/2}e^x}{1-t_xt_{2x}^{1/2}e^x}
        = -t_xt_{2x}^{1/2} \frac{e^{-x}-t_x^{1/2}}{1-t_x^{-1}t_{2x}^{-1/2}e^{-x}},
    \]
    which can be expanded as an $\mathcal{R}$-valued power series in
    $e^{a_0},e^{a_1}$.
\end{remark}

Write $q:=e^c$ and $z:=e^\epsilon$, so that $e^{n\epsilon + rc}=q^{r}z^n$ for any $n\in\Z, r\in\frac{\Z}{2}$. We can then view $\Delta,\nabla$ as formal distributions in
$z$ over the ring $\mathcal{R}((q^{1/2}))$ (formal Laurent series in $q^{1/2}$):
say
\[
    \Delta = \sum_{b\in\N_0a_0+\N_0a_1} u_b e^b,\qquad
    \nabla = \sum_{b\in\N_0a_0+\N_0a_1} v_b e^b,
\]
then we interpret $\Delta,\nabla$ as
\[
    \Delta = \sum_{n\in\Z} \qty(\sum_{r\in\frac{1}{2}\Z}
    u_{n\epsilon + rc} q^r) z^n,\qquad
    \nabla = \sum_{n\in\Z} \qty(\sum_{r\in\frac{1}{2}\Z}
    v_{n\epsilon + rc} q^r)z^n
    \in\mathcal{R}((q^{1/2}))[[z^{\pm1}]].
\]
\begin{lemma}
    The distributions $\Delta,\nabla$ are contained in
    $\mathcal{R}((q))[[z^{\pm1}]]$.
\end{lemma}
\begin{proof}
    We shall prove the claim for $\Delta$; the proof for
    $\nabla$ is analogous. We shall prove that in the
    power series representation of Definition~\ref{sec:def-nabla-delta}, odd powers of $e^{a_0}$ only occur paired
    with $q^{-\frac{1}{2}}$ and an expression in $a,b,c,d,q$, and even powers of $e^{a_0}$ only occur paired with an expression in $a,b,c,d,q$. It suffices to show this claim
    for every factor. Let $x\in S^+$. If $x=\pm\epsilon + rc$
    or $\pm2\epsilon + rc$ for some $r\in\Z$, the expression
    \[
        \frac{1-t_{2x}^{1/2}e^x}{1-t_xt_{2x}^{1/2}e^x}
    \]
    only contains even powers of $e^{a_0}$, and
    $t_{2x}^{1/2}=-b$ or $1$, and $t_xt_{2x}^{1/2}=a$ or
    $b^2$ or $q^{-1}d^2$. If $x=\pm\epsilon + rc$ for
    $r\in\frac{1}{2}\Z\setminus\Z$, we have
    \[
        \frac{1-t_{2x}^{1/2}e^x}{1-t_xt_{2x}^{1/2}e^x}
        = \frac{1+q^{-\frac{1}{2}}de^x}{1-q^{-\frac{1}{2}}ce^x}.
    \]
    Noting that $e^x$ is a product of positive powers of
    $e^{a_0},e^{a_1}$, we thus obtain a power series expansion
    with the claimed property.

    When reshaping this power series into a distribution, we thus note that any term of the shape
    $Cq^{-\frac{1}{2}}e^{ra_0+sa_1}$ for $r$ odd contributes
    exactly $Cq^{\frac{r-1}{2}}$, which only depends on $q$ and not on $q^{\frac{1}{2}}$ to the coefficient of $z^{s-r}$.
\end{proof}

\begin{definition}
    Let $R$ be a commutative ring. The \emph{constant term mapping}
    is the linear map
    \[
    \ct: R[[z^{\pm1}]]\to R,\qquad \sum_{n\in\Z} f_n z^n
    \mapsto f_0.
    \]
\end{definition}

\begin{definition}
    Let $K$ be the field $\Q(\sqrt{-ab}, b, \sqrt{-cd}, d, q^{1/2})$,
    equipped with the involution $*$ mapping all these generators to
    their multiplicative inverses. Let $\mathcal{A}:= K[z,z^{-1}]$ be the ring
    of Laurent polynomials over $K$ and let $\mathcal{A}_0$ the subring of polynomials invariant under $z\mapsto z^{-1}$.
    
    On $\mathcal{A}$ we define three involutions: 
    \begin{align*}
        \qty(\sum_{n\in\Z} g_n z^n)^\circ &:= \sum_{n\in\Z} g_n^* z^n\\
        \overline{\sum_{n\in\Z} g_n z^n} &:= \sum_{n\in\Z} g_n z^{-n}\\
        \qty(\sum_{n\in\Z}g_n z^n)^* &:= \sum_{n\in\Z} g_n^*z^{-n},
    \end{align*}
    from which we get three \enquote{inner products}:
    \begin{align*}
        (f,g) &:= \ct(fg^* \Delta)\\
        \langle f,g\rangle &:= \frac{1}{2} \ct(f\overline{g}\nabla)\\
        \langle f,g\rangle' &:= \ct(fg^*\nabla),
    \end{align*}
    which we interpret as mapping to $K\otimes_{\mathcal{R}[q^{\pm1/2}]}\mathcal{R}((q^{1/2}))$. All these maps can be normalised by dividing
    by the value at $(1,1)$. We then add the subscript 1:
    \[
        (f,g)_1 :=\frac{(f,g)}{(1,1)},\qquad
        \langle f,g\rangle_1 := \frac{\langle f,g\rangle}{\langle 1,1\rangle},\qquad
        \langle f,g\rangle'_1 := \frac{\langle f,g\rangle'}{\langle 1,1\rangle'}.
    \]
\end{definition}

\begin{proposition}\label{prop:general-facts-inner-products}\leavevmode
\begin{enumerate}
    \item $(\cdot,\cdot), \langle\cdot,\cdot\rangle$ are sesquilinear,
    meaning linear in the first argument and $*$-antilinear in the second
    and $\langle\cdot,\cdot\rangle$ is bilinear.
    \item The normalised inner products are well-defined and map to $K$.
    \item $(\cdot,\cdot)_1,\langle\cdot,\cdot\rangle'_1$ are
    Hermitian (with respect to $*$) and $\langle\cdot,\cdot\rangle$
    is symmetric.
    \item $(\cdot,\cdot)$ has no isotropic vectors.
    \item For $f,g\in \mathcal{A}_0$ we have
    \[
        (f,g) = \frac{1-ab}{2}\langle f,g\rangle'
        = (1-ab)\langle f,g^\circ\rangle,
    \]
    so that in particular $(f,g)_1 = \langle f,g\rangle'_1 = \langle f,g^\circ\rangle_1$, so that $\langle\cdot,\cdot\rangle$ has no
    isotropic vectors in $\mathcal{A}_0$, either.
\end{enumerate}
\end{proposition}
\begin{proof}\leavevmode
\begin{enumerate}
        \item This follows from the definitions.
        \item This is shown in \cite[\S5.1.10(a)]{Mac03}. Moreover, we shall
        later see, in Corollary~\ref{cor:norms}, that $(1,1)$ and $\langle 1,1\rangle$ are invertible in $\mathcal{R}((q^{1/2}))$.
        \item This follows from \cite[\S\S5.1.10(b),5.1.31]{Mac03} and the
        fact that (as is claimed ibid.) $\overline{\nabla}=\nabla$.
        \item This is the statement of \cite[\S5.1.20]{Mac03}.
        \item This follows from \cite[\S5.1.35]{Mac03} and the definition
        of $\langle\cdot,\cdot\rangle'$.\qedhere
    \end{enumerate}
\end{proof}

\begin{remark}\label{rem:non-symmetric-inner-product}
    Note that we avoided any claims of the non-normalised inner products $(\cdot,\cdot),\langle\cdot,\cdot\rangle,\langle\cdot,\cdot\rangle'$
    being Hermitian or symmetric. This stems from the fact that there is no
    way to define a $\star$ operation on $\mathcal{R}((q^{1/2}))$.

    A consequence of this is that taking the formal adjoint operations with
    respect to $(\cdot,\cdot)$ or $\langle\cdot,\cdot\rangle'$ will not be an involution.
\end{remark}

If we impose the following (total) monomial ordering on $A$:
\[
    1 < z < z^{-1} < z^2 < z^{-2} < \cdots,
\]
we can employ the Gram--Schmidt orthogonalisation procedure to
obtain the non-symmetric AW-polynomials:
\begin{proposition}
    There is a unique family $(E_n)_{n\in\mathbb{Z}}\subset\mathcal{A}$ of polynomials such that:
    \begin{enumerate}
        \item $E_n(z) = z^n + \lot$, where $\lot$ refers to lower-order terms
        with respect to the monomial ordering defined above, and
        \item $(E_n(z), z^m)=0$ for $z^m < z^n$.
    \end{enumerate}
    The polynomials $(E_n)_{n\in\Z}$ are the \emph{non-symmetric AW-polynomials}.
\end{proposition}

If we impose the following (total) monomial ordering on $\mathcal{A}_0$:
\[
1 < z+z^{-1} < z^2 + z^{-2} <\cdots,
\]
we can employ the Gram--Schmidt orthogonalisation procedure to
obtain the symmetric AW-polynomials:
\begin{proposition}
    There is a unique family $(P_n)_{n\in\mathbb{N}_0}\subset\mathcal{A}_0$ of
    symmetric polynomials such that
    \begin{enumerate}
        \item $P_n(z) = z^n + z^{-n} + \lot$, and
        \item $\langle P_n(z), z^m + z^{-m}\rangle'=0$
        for $z^m+z^{-m} < z^n + z^{-n}$,
    \end{enumerate}
    where in light of Proposition~\ref{prop:general-facts-inner-products}(v), instead of $\langle\cdot,\cdot\rangle'$ we can also use
    $\langle\cdot,\cdot\rangle$ and $(\cdot,\cdot)$. The polynomials
    $(P_n)_{n\in\N_0}$ are the
    \emph{symmetric AW-polynomials}.
\end{proposition}

\begin{theorem}
    $(E_n)_{n\in\mathbb{Z}}$ and $(P_m)_{m\in\N_0}$ are orthogonal bases
    (with respect to $(\cdot,\cdot)$) of $\mathcal{A}, \mathcal{A}_0$, respectively.
\end{theorem}
\begin{proof}
    They are bases because they were constructed out of monomial bases by
    means of Gram--Schmidt procedure. They are orthogonal because the monomial
    orders we chose are total. Note that for $(P_m)_{m\in\N_0}$ it does not
    matter which inner product we consider since by Proposition~\ref{prop:general-facts-inner-products}(v) all of these inner product are proportional to each other.
\end{proof}

\begin{notation}\label{not:powers-formal-stuff}
    Throughout, we will also consider cases where we replace
    $a,b,c,d$ by $q$-multiples of themselves, e.g. $(aq^{1/2},bq^{1/2},cq^{-1/2},dq^{-1/2})$. While this does not cause any
    problems in $K$, for $\mathcal{R}((q^{1/2}))$ it is not possible to do in general. Consider for example the element
    \[
        \sum_{n=0}^\infty a^n q^n.
    \]
    If we replace $a$ by $aq^{-1}$, we obtain
    \[
        \sum_{n=0}^\infty a^n,
    \]
    which is not an element of $\mathcal{R}((q^{1/2}))$. However, as is shown later in Corollary~\ref{cor:norms}, $(1,1)$ and $\langle 1,1\rangle$ are both of a form
    where such substitutions are valid. 
    
    We will write
    \[
        E_n(a',b',c',d';z), \quad P_m(a',b',c',d';z), \quad (\cdot,\cdot)_{a',b',c',d'}
    \]
    for 
    \[
    E_n(z), \quad P_m(z), \quad (\cdot,\cdot),
    \]
    where $(a,b,c,d)$ have been replaced by $(a', b', c', d')$, and similarly for the other \enquote{inner products}.

    In addition, as is touched upon at the beginning of \cite[\S5.1]{Mac03}, in connection with \cite[\S6.5.2]{Mac03}, it makes sense
    to view the parameters $a,b,c,d$ (which are here treated as formal
    variables) as formal powers of $q$ associated with a labelling $k$:
    \[
        (a,b,c,d) = \qty(q^{k_1}, -q^{k_2}, q^{\frac{1}{2} + k_3},
        -q^{\frac{1}{2} + k_4}),
    \]
    so that multiplication of any of the parameters $a,b,c,d$ by a power
    of $q$ can be more conveniently expressed by a change of labelling $k$ to
    $k+h$, which we will indicate by an additional subscript.

    Some more parameters we will use are $\tau_1,\widetilde{\tau}_1,\tau_0,\widetilde{\tau}_0$ (as defined in \cite[Definition~2.28]{Sch23} and \cite[\S5.1.6]{Mac03}), that are related to $a,b,c,d$
    as follows:
    \[
        a = \tau_1\widetilde{\tau}_1,\qquad
        b = -\tau_1\widetilde{\tau}_1^{-1},\qquad
        c = \tau_0\widetilde{\tau}_0q^{\frac{1}{2}},\qquad
        d = -\tau_0\widetilde{\tau}_0^{-1}q^{\frac{1}{2}}.
    \]
\end{notation}

\begin{remark}\label{rmk:permutation-invariance}
    It can be shown that $\nabla$ is invariant under any permutation of
    $(a,b,c,d)$, hence so are the $(P_m)_{m\in\N_0}$.

    Similarly, since $\Delta$ and $\nabla$ are related via a polynomial in $a,b$ that is
    symmetric, the symmetry in the parameters is only slightly broken when going to the
    non-symmetric AW-polynomials. In particular: $\Delta$ and hence the
    $(E_n)_{n\in\Z}$ are invariant under $a\leftrightarrow b$ and $c\leftrightarrow d$.
\end{remark}

\section{Symmetric Shift Operators}\label{sec-symmetric}
In \cite{AsWi85}, Askey and Wilson introduced the forward and backward shift operators (here called $G^q_+=S_{v_1}$ and $G^q_-=S_{-v_1}$). A contiguity shift operator (here called $E^q=E^q_{12}=S_{-v_2}$) was later found by Kalnins and Miller in \cite{KaMi89}. In this section, we classify all (symmetric) shift operators and show that these three shift operators and the ones obtained by permuting the parameters $a,b,c,d$, the so-called \emph{fundamental shift operators}, are in fact sufficient to describe all shift operators.

\subsection{General Facts}
\begin{definition}
    Let $T$ be the operator on $\mathcal{A}$ given by
    \[
        Tf(z) = f\qty(q^{\frac{1}{2}}z).
    \]
    An operator of the form
    \[
        S = \sum_{n\in\Z} c_n(z) T^n,
    \]
    with $c_n\in K(z)$ for $n\in\Z$ and $c_n=0$ for almost all $n\in\Z$,
    such that $S\mathcal{A}_0 \subset\mathcal{A}_0$, is called a
    \emph{symmetric difference operator}. A symmetric difference operator is called \emph{(symmetric) shift operator} with shift $h\in\Q^4$
    if
    \[
        \forall m\in\N_0: \exists C\in K: \quad SP_{m,k} = C P_{m+d,k+h},
    \]
    where $2d+h_1+h_2+h_3+h_4=0$. Write
    $\mathcal{S}(h)$ for the set of shift operators with shift $h$.
\end{definition}

\begin{notation}
    We will expand the shifts in terms of the following two bases:
    \begin{itemize}
        \item the standard basis $\epsilon_1,\dots,\epsilon_4$;
        \item and the basis
        \begin{align*}
            v_1 &:= \frac{1}{2}(1,1,1,1)\\
            v_2 &:= \frac{1}{2}(1,1,-1,-1)\\
            v_3 &:= \frac{1}{2}(1,-1,1,-1)\\
            v_4 &:= \frac{1}{2}(1,-1,-1,1).
        \end{align*}
    \end{itemize}
    We write $v\cdot w := \sum_{i=1}^4 v_iw_i$ for the standard
    inner product on $\Q^4$. This way we have
    \[
        h_i = h\cdot \epsilon_i,\qquad
        h'_i = h\cdot v_i
    \]
    for $i=1,\dots,4$, where $h'$ is the dual labelling as per
    \cite[\S1.5.1]{Mac03}. In particular, the dependence of $d$ on
    $h$ from the last definition reads $d=-h\cdot v_1$.
\end{notation}

\begin{remark}\leavevmode
\begin{enumerate}
    \item Note that the definition already restricts the subset of possible shifts to those $h$ satisfying $\sum_{i=1}^4 h_i\in 2\Z$. We will establish more
    restrictions in Lemma~\ref{lem-parity}.
    \item Each $\mathcal{S}(h)$ is a $K$-vector space.
\end{enumerate}
\end{remark}

\begin{definition}
    Define the \emph{graded $\Q$-algebra of shift operators} to be the
    $\Q^4$-graded $K$-vector space
    \[
        \mathcal{S}:=\bigoplus_{h\in\Q^4} \mathcal{S}(h)
    \]
    equipped with the following multiplication:
    \[
        \circ: \mathcal{S}(h)\times\mathcal{S}(h')\to\mathcal{S}(h+h'),\qquad
            (S_k,S'_k)\mapsto S_{k+h'}S'_k
    \]
    on homogeneous elements.
\end{definition}

\begin{remark}\label{rmk-about-S}\leavevmode
\begin{enumerate}
    \item Note that the multiplication is not $K$-bilinear,
    so $\mathcal{S}$ is not a $K$-algebra. However, $\mathcal{S}$ carries a left and right
    action by the algebra $\mathcal{S}(0)$, which contains $K$. This 
    turns out to be a more useful perspective.
    \item There is another way to approach shift operators.
    Let $\mathcal{S}$ be the category whose objects are labels
    of the shape $k+h$ for $h\in\Q^4$, and where
    $\mathcal{S}(k+h, k+h')$ consists of those symmetric difference operators $S$ such that
    \[
        \forall m\in\N_0: \exists C\in K:\quad
        SP_{m,k+h} = CP_{m-(h'-h)\cdot v_1, k+h'},
    \]
    with composition being the usual composition of symmetric
    difference operators. This category possesses a lot of symmetry because there are
    twisted $K$-linear isomorphisms between
    \[
        \mathcal{C}(k+h, k+h')
        \cong \mathcal{C}(k, k+h'-h)
    \]
    for all $h,h'$ that just replace $a,b,c,d$ by
    \[
        aq^{-h_1},\dots,dq^{-h_4}.
    \]
    The composition we defined earlier for the algebra of shift operators can be seen as applying one such isomorphism to make two shift operators into composable morphisms, and then composing them. Because of the redundancy afforded by the isomorphisms of the morphism spaces, we will stick to the algebra point of
    view. 
\end{enumerate}
\end{remark}

\begin{lemma}\label{lem-sdo-reflection-behaviour}
    Let
    \[
        S = \sum_{n\in\Z} c_n(z) T^n
    \]
    be a symmetric difference operator. Then
    $c_n(z) = c_{-n}\qty(z^{-1})$ for all $n\in\Z$.
\end{lemma}
\begin{proof}
    That $S$ is symmetric, means that $S$ commutes with $s_1$, hence
    that
    \[
        \sum_{n\in\Z} c_n\qty(z^{-1})T^{-n} s_1
        =s_1S = Ss_1 = \sum_{n\in\Z} c_n(z)T^n s_1.
    \]
    Since the $T^n$ are $K(z)$-linearly independent,
    we obtain the claim.
\end{proof}

\begin{proposition}\label{prop-transmutation-property}
    Write $L_k := Y_k + Y_k^{-1}$ for the symmetric Cherednik operator
    (e.g. \cite[\S6.5.3]{Mac03}) and let $S$ be a symmetric difference
    operator.
    \begin{enumerate}
        \item We have $S_kL_k = L_{k+h}S_k$ if and only if
        $S$ is a shift operator with shift $h$.
        \item $\mathcal{S}(0)$ is the centraliser of $L_k$ in
        the algebra of symmetric difference operators.
    \end{enumerate}
\end{proposition}
\begin{proof}
    Note that (ii) immediately follows from (i), so we shall just prove
    the first claim.\\
    \enquote{$\Rightarrow$}: We use some notation from \cite[\S1.5.1]{Mac03}, in particular $k'_1=k\cdot v_1$. With
    $d=-h\cdot v_1$, we then have $(k+h)'_1 = k'_1 - d$. Let $m\in\mathbb{N}$. By \cite[\S6.5.3]{Mac03} and the assumption we have
    \begin{align*}
        L_{k+h} S_k P_{m,k} &=
        S_k L_k P_{m,k}\\
        &= \qty(q^{m+k'_1}+q^{-m-k'_1}) S_kP_{m,k}\\
        &= \qty(q^{m+d+(k+h)'_1} + q^{-m-d-(k+h)'_1})
        S_kP_{m,k}.
    \end{align*}
    Since the eigenspaces of $L_{k+h}$ are one-dimensional, we can therefore
    conclude that $d\in\Z$ and that $S_kP_{m,k}$ is a multiple of $P_{m+d,k+h}$, as required.\\
    \enquote{$\Leftarrow$}: It suffices to show that for all $m\in\N_0$
    we have $S_kL_kP_{m,k}=L_{k+h}S_kP_{m,k}$ for all $m\in\N_0$, as
    $(P_{m,k})_{m\in\N_0}$ is a basis of $\mathcal{A}_0$. In particular,
    for $m\in\N_0$, there is $C\in K$ such that
    \begin{align*}
        S_k L_k P_{m,k} &=
        S_k \qty(q^{m+k'_1} + q^{-m-k'_1}) P_{m,k}\\
        &= C\qty(q^{m+k'_1} + q^{-m-k'_1}) P_{m+d,k+h}\\
        &= C\qty(q^{m+d+(k+h)'_1} + q^{-m-d-(k+h)'_1}) P_{m+d,k+h}\\
        &= C L_{k+h} P_{m+d,k+h}\\
        &= L_{k+h} S_k P_{m,k}.\qedhere
    \end{align*}
\end{proof}

\begin{lemma}\label{lem-decomposition-L}
    $L_k$ has the decomposition $f_{-2, k}(z)T^{-2} + f_0(z) + f_{2, k}(z)T^2$
    with
    \[
        f_{2, k}(z) = q^{-k'_1}\frac{(1-az)(1-bz)(1-cz)(1-dz)}{\qty(1-z^2)\qty(1-qz^2)}.
    \]
\end{lemma}
\begin{proof}
    From the proof of \cite[\S6.5.3]{Mac03} we have
    \[
        L_k = \bm{c}_1(z)\bm{c}_0\qty(q^{\frac{1}{2}}z)
        (T^2-1)
        + \bm{c}_1\qty(z^{-1})\bm{c}_0\qty(q^{\frac{1}{2}}z^{-1})(T^{-2}-1) + q^{k'_1} + q^{-k'_1},
    \]
    where
    \[
        \bm{c}_i(z) = \frac{\tau_iz - \tau_i^{-1}z^{-1} + \widetilde{\tau}_i - \widetilde{\tau}_i^{-1}}{z-z^{-1}},
    \]
    with the $\tau$ parameters from Notation~\ref{not:powers-formal-stuff}. Note that
    \begin{align*}
        \bm{c}_1(z)\bm{c}_0\qty(q^{\frac{1}{2}}z) &=
        \frac{\tau_1z - \tau_1^{-1}z^{-1} + \widetilde{\tau}_1 - \widetilde{\tau}_1^{-1}}{z-z^{-1}}
        \frac{\tau_0q^{\frac{1}{2}}z - \tau_0^{-1}q^{-\frac{1}{2}}z^{-1} + \widetilde{\tau}_0 - \widetilde{\tau}_0^{-1}}{q^{\frac{1}{2}}z-q^{-\frac{1}{2}}z^{-1}}\\
        &= q^{-k'_1}\frac{\qty(\tau_1^2z^2 - 1 + \qty(\tau_1\widetilde{\tau}_1-\tau_1\widetilde{\tau}_1^{-1})z)\qty(\tau_0^2qz^2 - 1 + \qty(\tau_0\widetilde{\tau}_0-\tau_0\widetilde{\tau}_0^{-1})q^{\frac{1}{2}}z)}{(1-z^2)\qty(1-qz^2)}\\
        &= q^{-k'_1}\frac{(1-az)(1-bz)(1-cz)(1-dz)}{\qty(1-z^2)\qty(1-qz^2)}.\qedhere
    \end{align*}
\end{proof}

\begin{proposition}\label{prop-sso-leading-term}
    If $S\in\mathcal{S}(h)$, say
    \[
        S = \sum_{n=-r}^r c_n(z) T^n,
    \]
    then there is $A\in K$ such that
    \begin{align}\label{eq-sso-leading-term}
        c_r(z) &= A z^{-h\cdot v_1}\cdot\\
        &\frac{
        \qty(a^{-1}q^{-h_1}z^{-1},\dots,d^{-1}q^{-h_4}z^{-1},
        q^{-\frac{r}{2}}z^{-1},-q^{-\frac{r}{2}}z^{-1},
        r^{-\frac{r+1}{2}}z^{-1},-q^{-\frac{r+1}{2}}z^{-1};q)_\infty}{\qty(a^{-1}q^{-\frac{r}{2}}z^{-1},\dots,d^{-1}q^{-\frac{r}{2}}z^{-1},z^{-1},-z^{-1},q^{-\frac{1}{2}}z^{-1},-q^{-\frac{1}{2}}z^{-1};q)_\infty}, \nonumber
    \end{align}
    written as a power series in $z^{-1}$.
\end{proposition}
\begin{proof}
    Recall from Lemma \ref{lem-decomposition-L}
    \[
        L_k = f_{-2,k}(z)T^{-2} + f_{0,k}(z) + f_{2,k}(z)T^2.
    \]
    Then by Proposition~\ref{prop-transmutation-property} and
    the $K((z^{-1}))$-linear independence of different powers
    of $T$ imply that
    \[
        c_r(z)T^r f_{2,k}(z)T^2 = f_{2,k+h}(z)T^2c_r(z)T^r,
    \]
    or equivalently
    \[
        c_r(z) f_{2,k}\qty(q^{\frac{r}{2}}z) = c_r(qz) f_{2,k+h}(z).
    \]
    Using the concrete expression from Lemma~\ref{lem-decomposition-L}, we obtain
    \[
        \frac{c_r(z)}{c_r(qz)}
        = q^{-h\cdot v_1}
        \frac{\qty(1-a^{-1}q^{-h_1}z^{-1})\cdots\qty(1-d^{-1}q^{-h_4}z^{-1})\qty(1-q^{-r}z^{-2})\qty(1-q^{-r-1}z^{-2})}{\qty(1-a^{-1}q^{-\frac{r}{2}}z^{-1})\cdots
        \qty(1-d^{-1}q^{-\frac{r}{r}}z^{-1})\qty(1-z^{-2})\qty(1-q^{-1}z^{-2})}.
    \]
    This $q$-difference equation has only one solution in
    $K((z^{-1}))$, namely the claimed function.
\end{proof}

\subsection{Zero Shift}
\begin{theorem}\label{thm-zero-shifts}
    We have
    \[
        \mathcal{S}(0)=K[L].
    \]
\end{theorem}
\begin{proof}
    Note that $L$ commutes with itself, as do all its powers,
    hence $K[L]$ is a subset of $\mathcal{S}(0)$. Conversely, assume
    that
    \[
        S = \sum_{n=-r}^r c_n(z) T^n,
    \]
    with $c_r(z)\ne0$, is an element of $\mathcal{S}(0)$. We claim that this is a polynomial in $L$, which we shall prove by induction in $r$. For $r=0$, there is nothing to show, as $S$ is a constant polynomial. Assume the claim is already proven for numbers smaller than $r$.

    By Proposition~\ref{prop-sso-leading-term}, $c_r(z)$ is
    given by \eqref{eq-sso-leading-term} for $h=0$. Note that we assume this to
    be a rational function in $z^{-1}$, so almost all factors in the infinite products
    have to cancel. Since $a,b,c,d$ are algebraically independent, this implies that $r$ is an even number.
    Furthermore, by inspection we find that
    \[
        c_r(z)T^r = A\qty(f_{k,2}(z) T^2)^{\frac{r}{2}},
    \]
    so that the shift operator with shift 0
    \[
        S - A L^{\frac{r}{2}}
    \]
    can be written in the form
    \[
        \sum_{n=1-r}^{r-1} d_n(z) T^n.
    \]
    Note that $S-AL^{\frac{r}{2}}$ is a symmetric difference operator,
    so that by Lemma~\ref{lem-sdo-reflection-behaviour}, the lower bound
    of this sum also rises to $1-r$. Consequently, $S-AL^{\frac{r}{2}}$ satisfies the induction hypothesis and is therefore a polynomial in $L$. Thus, $S\in K[L]$.
\end{proof}

\begin{corollary}
    $\mathcal{S}(0)$ is commutative, and its left and right actions
    on $\mathcal{S}$ coincide.
\end{corollary}
\begin{proof}
    Commutativity follows from Theorem~\ref{thm-zero-shifts}, and 
    the coincidence of the left and right actions is the transmutation property from Proposition~\ref{prop-transmutation-property}.
\end{proof}

\subsection{Non-Zero Shifts}
\begin{lemma}\label{lem-even-odd-part}
    Let
    \[
        S = \sum_{n\in\Z} c_n(z)T^n
    \]
    be a shift operator with shift $h$. Then
    \[
        S_{\mathrm{even}}:= \sum_{n\in2\Z} c_n(z)T^n,\qquad
        S_{\mathrm{odd}} := \sum_{n\in 1+2\Z} c_n(z)T^n
    \]
    are also shift operators with shift $h$.
\end{lemma}
\begin{proof}
    Recall from Lemma \ref{lem-decomposition-L}
    \[
        L_k = f_{-2,k}(z)T^{-2} + f_{0,k}(z) + f_{2,k}(z)T^2.
    \]
    By Proposition~\ref{prop-transmutation-property}, we have
    \[
        S_kL_k = L_{k+h}S_k,
    \]
    which results in the following equations
    \begin{align*}
        \sum_{m\in\Z} \sum_{n\in\Z} c_n(z) f_{m-n,k}\qty(q^{\frac{n}{2}}z)
        T^m
        &=\sum_{m,n\in\Z} c_m(z)T^m f_{n,k}(z)T^n\\
        &= \sum_{m,n\in\Z} f_{m,k+h}(z)T^m c_n(z)T^n\\
        &= \sum_{m\in\Z} \sum_{n\in\Z} c_{n,k+h}(z) c_{m-n}\qty(q^{\frac{n}{2}}z) T^m,
    \end{align*}
    which implies
    \begin{equation}\label{eq-transmutation-property-rational-functions}
        \sum_{n\in\Z} c_n(z) f_{m-n,k}\qty(q^{\frac{n}{2}}z) = \sum_{n\in\Z} f_{m-n,k+h}(z) c_n\qty(q^{\frac{m-n}{2}}z)
    \end{equation}
    for all $m\in\Z$. Since $L_k$ is written only in terms of even powers of $T$, the
    only terms in these sums that can be non-zero are those where
    $m$ and $n$ have the same parity. Consequently, taking $m$ to be an even integer in \eqref{eq-transmutation-property-rational-functions} proves that $S_{\text{even}}$ satisfies
    \[
    S_{\text{even},k}L_k = L_{k+h}S_{\text{even},k}.
    \]
    Similarly, taking $m$ to be an odd integer proves the same for $S_{\text{odd}}$. By Proposition~\ref{prop-transmutation-property}, we thus obtain that $S_{\text{even}}$ and $S_{\text{odd}}$ are shift operators with shift $h$.
\end{proof}

\begin{definition}
    With notation as in Lemma~\ref{lem-even-odd-part}, call $S_{\text{even}},S_{\text{odd}}$ the \emph{even/odd part} of $S$, respectively. A shift operator $S$ is called \emph{even} or
    \emph{odd} if it equals its even or odd part.
\end{definition}

\begin{lemma}\label{lem-parity}
    If $\mathcal{S}(h)\ne0$, then $h$ lies in the $\Z$-span of
    $v_1,v_2,v_3,v_4$. In particular, the elements of $\mathcal{S}(h)$ are even/odd if the entries of $2h$ are even/odd integers.
\end{lemma}
\begin{proof}
    Write
    \[
        S_k = \sum_{n=-r}^r c_n(z)T^n\in\mathcal{S}(h),
    \]
    with $c_r(z)\ne0$. Without loss of generality let $S_k$ be even/odd (which is also the parity of $r$). By Proposition~\ref{prop-sso-leading-term}, the leading
    coefficient $c_r(z)$ is given by \eqref{eq-sso-leading-term}. Since we assume this to be a rational function in $z^{-1}$, almost all factors of the infinite products have to cancel. Since $a,b,c,d$ are algebraically independent,
    we obtain that $h_i-\frac{r}{2}$ is an integer for $i=1,2,3,4$, say $h_i = n_i + \frac{r}{2}$. We then have
    \[
        2h\cdot v_1 = 2r + 2n\cdot v_1,\qquad
        2h\cdot v_i = 2n\cdot v_i \quad (i=2,3,4)
    \]
    where $n=(n_1,\dots,n_4)$. Consequently,
    $2h\cdot (v_i-v_1) = 2r + 2n\cdot (v_i-v_1)$. Note that
    $v_i-v_1$ has integer values as has $n$, so that $2h\cdot (v_i-v_1)$ is an even number. This shows that all the
    $2h\cdot v_i$ ($i=1,\dots,4$) are integers with the same
    parity. Since $S$ shifts degrees by $-h\cdot v_1$, we know
    that $h\cdot v_1$ must be an integer, hence so are
    $h\cdot v_i$ ($i=2,\dots,4$). As a consequence,
    \[
        h = (h\cdot v_1)v_1 + \cdots + 
        (h\cdot v_4)v_4 \in \Z v_1 \oplus\cdots\oplus\Z v_4.\qedhere
    \]
\end{proof}

\begin{example}\label{ex-three-shift-operators}
    Define the three operators
    \begin{align*}
        G^q_+ &:= \frac{1}{z-z^{-1}}\qty(T-T^{-1})\\
        G^q_- &:= \frac{1}{z-z^{-1}}\qty(A_1(z)T-A_1\qty(z^{-1})T^{-1})\\
        E^q &:= \frac{1}{z-z^{-1}}\qty(A_2(z)T-A_2\qty(z^{-1})T^{-1}),
    \end{align*}
    with
    \begin{align*}
        A_1(z) &:= q^{-\frac{1}{2}}z^{-2}\qty(1-aq^{-\frac{1}{2}}z)\cdots
        \qty(1-dq^{-\frac{1}{2}}z)\\
        A_2(z) &:= -z^{-1}\qty(1-aq^{-\frac{1}{2}}z)\qty(1-bq^{-\frac{1}{2}}z).
    \end{align*}
    $G^q_+$ and $G^q_-$ are called the \emph{forward} and \emph{backward shift operator}, and $E^q$ is called the \emph{contiguity shift operator}.

    Note that $G^q_+,G^q_-$ are symmetric in the parameters $a,b,c,d$
    (the former being completely independent of them). Since the
    symmetric AW-polynomials are symmetric as well, it makes sense to
    consider permutations of parameters as well. Hence, if we call
    $u_1:=a, u_2:=b, u_3:= c, u_4:= d$ and $i\ne j\in\set{1,2,3,4}$, define
    \[
        E^q_{ij}(a,b,c,d) := E^q(u_i, u_j).
    \]
    This operation takes precedence over shifting the labelling. For example, we have
    \[
        E^q_{34,k+h}(a,b,c,d) = E^q(cq^{h_3},dq^{h_4}).
    \]
\end{example}

\begin{remark}
    As mentioned in the introduction, $G^q_+,G^q_-$ were introduced in \cite{AsWi85} and $E^q$ in \cite{KaMi89}. In addition, these operators are also described in \cite{GTVZ20}, which uses a closely related normalisation for the AW-polynomials, so the formulae are easier to convert. In the notation of these papers, we have
    \begin{align*}
        G^q_+ = G^q_{+,k} &= \tau(a,b,c,d)\\
        G^q_{-,k+v_1} &= \tau^*(a,b,c,d)\\
        E^q = E^q_k &= \mu(a,b,c,d).
    \end{align*}
\end{remark}

\begin{proposition}\label{prop-shift-factors}
    The operators
    \[
        G^q_+,G^q_-,E^q_{12},E^q_{13},E^q_{14},E^q_{23},E^q_{24},E^q_{34}
    \]
    are shift operators with shifts
    \[
        v_1,-v_1,-v_2,-v_3,-v_4,v_4,v_3,v_2,
    \]
    respectively. In particular,
    \begin{align*}
        G^q_+ P_{m, k} &= (q^{\frac{m}{2}} - q^{-\frac{m}{2}}) P_{m - 1, k + v_1} \\
        G^q_- P_{m, k} &= q^{-\frac{1}{2}} \qty(\frac{abcd}{q^2} q^{\frac{m}{2}} - q^{-\frac{m}{2}}) P_{m + 1, k - v_1} \\
        E^q P_{m, k} &= -\qty(\frac{ab}{q} q^{\frac{m}{2}} - q^{-\frac{m}{2}})P_{m, k - v_2}.
    \end{align*}
\end{proposition}
\begin{proof}
See \cite[Equations~3.1,3.10]{GTVZ20} for the relations concerning $G^q_+$, $G^q_-$, and $E^q = E^q_{12}$. However, we should take into account that the authors of \cite{GTVZ20} use $q^2$ where we use $q$, and they use a different normalisation: Their AW-polynomials $(p_m)_{m\in\N_0}$ are related to the monic Askey--Wilson polynomials $(P_m)_{m\in\N_0}$ by
\[
    p_m = (abcdq^{m-1};q)_m P_m.
\]
This implies also that the $E^q_{ij}$ are shift operators with the indicated shifts by permuting the parameters.
\end{proof}

\begin{notation}\label{not-fundamental-shift-ops}
    For $\epsilon\in\set{\pm1},i\in\set{1,2,3,4}$ write
    $S_{\epsilon v_i} := S_{\epsilon v_i,k}$ for the operator among
    \[
        G^q_+,G^q_-,E^q_{12},E^q_{13},E^q_{14},E^q_{23},E^q_{24},E^q_{34}
    \]
    with shift $\epsilon v_i$. Together, they are referred to as the
    \emph{fundamental shift operators}.
\end{notation}

We thus have shift operators for every monoid generator of the lattice of possible shifts. We will later show that these suffice to
generate $\mathcal{S}$.

\begin{lemma}\label{lem-z-power-series}
    If $S\in\mathcal{S}(h)$, then there is a unique decomposition
    \[
        S = \sum_{n=0}^\infty z^{d-n} \alpha_n(T)
    \]
    where $\alpha_n\in K\qty[T,T^{-1}]$ for $n\in\N_0$, and with $S\ne0$ implying that $\alpha_0\ne0$. Here, we again take $2d+h_1+h_2+h_3+h_4=0$, i.e. $d=-h\cdot v_1$.
\end{lemma}
\begin{proof}
    Note that $K(z)$ can be seen as a subset of $K\qty(\qty(z^{-1}))$
    as the latter is a field that contains $z$. Consequently, every
    rational function $c_n$ in the expansion
    \[
        S = \sum_{n\in\Z} c_n(z) T^n
    \]
    can be expanded as a Laurent series in $z^{-1}$. Since only finitely
    many $c_n(z)$ are non-zero, the terms can be reorganised to give an
    expansion of the form
    \[
        S = \sum_{n=-N}^\infty z^{-n} \beta_n(T),
    \]
    with $\beta_{-N}\ne0$. Then we have
    \[
        S \qty(z^m + z^{-m})
        = \sum_{n=-N}^\infty \qty(\beta_n\qty(q^{\frac{m}{2}})z^{m-n}
        + \beta_n\qty(q^{-\frac{m}{2}})z^{-m-n}).
    \]
    For every $m\in\N_0$, the highest power in this expansion
    needs to be $z^{m+d}$ or lower. Thus, for $n<-d$ we have
    \[
        \beta_n\qty(q^{\frac{m}{2}})=0
    \]
    for all $m\in\N_0$. The polynomial $\beta_n$ thus has infinitely many distinct zeroes,
    hence $\beta_n=0$. This shows the series representation as in the claim. To show that $\beta_{-d}=\alpha_0=0$ implies that $S=0$,
    note that
    \[
        S_k P_{m,k} = \alpha_0\qty(q^{\frac{m}{2}})
        z^{m+d} + \lot
        = C z^{m+d} + \lot 
        = CP_{m+d,k+h},
    \]
    which shows that
    \[
        S_k P_{m,k} = \alpha_0\qty(q^{\frac{m}{2}})
        P_{m+d,k+h}.
    \]
    This determines the constants $C$ in the definition of shift
    operator. In particular, if $\alpha_0=0$, the shift operator
    $S_k$ maps $(P_{m,k})_{m\in\N_0}$ to 0. Since that is a basis of
    $\mathcal{A}_0$, $S_k$ is the zero operator, hence $S_k=0$.
\end{proof}

\begin{corollary}\label{cor-eta-shift-factor}
For $S\in\mathcal{S}(h)$, we have
\[
    \quad S_k P_{m,k} = \widetilde{\eta}_h(S)\qty(q^{\frac{m}{2}}) P_{m-h\cdot v_1,k+h}
\]
for all $m\in\N_0$. 
\end{corollary}

\begin{definition}
    Define the ring
    \[
        \mathcal{R} := 
        \bigoplus_{h\in\Q^4} \mathcal{R}(h) = \bigoplus_{h\in \Q^4} t^h K\qty[T,T^{-1}],
    \]
    with multiplication given by
    \[
        t^h f(a,b,c,d;T) t^{h'} f'(a,b,c,d;T)
        := t^{h+h'} f\qty(aq^{h_1},\dots,dq^{h_4};q^{-\frac{h\cdot v_1}{2}}T) f'(a,b,c,d;T).
    \]
    Note that this is an integral domain.
\end{definition}

Lemma~\ref{lem-z-power-series} allows us to define the following maps.

\begin{definition}
    For $h\in\Q^4$ define the map
    \[
        \widetilde{\eta}_h: \mathcal{S}(h)\to K\qty[T,T^{-1}],
    \]
    by mapping
    \[
        \sum_{n=0}^\infty z^{-h\cdot v_1-n}\alpha_{n}(T)
        \mapsto \alpha_{0}(T).
    \]
    Define furthermore $\eta: \mathcal{S}\to\mathcal{R}$ by
    mapping a homogeneous element $S\in\mathcal{S}(h)$ to
    $t^h \widetilde{\eta}_h(S)$.
\end{definition}

\begin{proposition}\label{prop-injective-ringhom}
    The map $\eta$ is an injective algebra homomorphism.
\end{proposition}
\begin{proof}
    Note that $\eta$ is evidently $\Q$-linear and injective by Lemma~\ref{lem-z-power-series}. It remains to show that it preserves multiplication. If
    $S,S'$ be shift operators with shifts $h,h'$, then we have
    \begin{align*}
        S_k &=  z^{-h\cdot v_1} \widetilde{\eta}_h(S)(a,b,c,d;T)
        + \lot\\\
        S'_k &= z^{-h'\cdot v_1} \widetilde{\eta}_{h'}(S')(a,b,c,d;T) + \lot,
    \end{align*}
    where $\lot$ refers to lower-order powers of $z$. The composition of $S$ and $S'$ is then
    \begin{align*}
        S_{k+h'} S'_k
        &= z^{-h\cdot v_1} \widetilde{\eta}_h(S)(aq^{h'_1},\dots,dq^{h'_4};T)
        z^{-h'\cdot v_1} \widetilde{\eta}_{h'}(S')(a,b,c,d;T)
        + \lot\\
        &= z^{-(h+h')\cdot v_1}
        \widetilde{\eta}_h(S)(aq^{h'_1},\dots,dq^{h'_4};q^{-\frac{h'\cdot v_1}{2}}T)
        \widetilde{\eta}_{h'}(S')(a,b,c,d;T) + \lot,
    \end{align*}
    so that $\eta(S\circ S')$ equals
    \begin{align*}
        \eta(S\circ S') &=
        t^{h+h'} \widetilde{\eta}_h(S)(aq^{h'_1},\dots,dq^{h'_4};q^{-\frac{h'\cdot v_1}{2}}T)
        \widetilde{\eta}_{h'}(S')(a,b,c,d;T)\\
        &= t^h \widetilde{\eta}_h(S) t^{h'} \widetilde{\eta}_{h'}(S')\\
        &= \eta(S)\eta(S').\qedhere
    \end{align*}
\end{proof}

We have thus reduced the characterisation of $\mathcal{S}$ to the
characterisation of $\im(\eta)$.

\begin{proposition}
    We have
    \[
        \eta(L) = \tau_0\tau_1 T^2 - \qty(\tau_0\tau_1 T^2)^{-1}.
    \]
\end{proposition}
\begin{proof}
    This follows from the proof of Lemma~\ref{lem-decomposition-L} and the fact that the $\bm{c}$ functions can be expanded as follows
    \begin{align*}
        \bm{c}_1(z) &= \tau_1 + \sum_{n=1}^\infty a_n z^{-n}\\
        \bm{c}_1 (z^{-1}) &= \tau_1^{-1} - \sum_{n=1}^\infty a_n z^{-n}\\
        \bm{c}_0\qty(q^{1/2}z^{-1}) &=
        \tau_0^{-1} - \sum_{n=1}^\infty b_n q^{n/2}z^{-n}\\
        \bm{c}_0\qty(q^{1/2}z)
        &= \tau_0 + \sum_{n=1}^\infty b_n q^{-n/2}z^{-n},
    \end{align*}
    where
    \[
    a_n = \begin{cases}
        \tau_1-\tau_1^{-1} & n\in 2\Z\\
        \widetilde{\tau}_1 - \widetilde{\tau}_1^{-1} & n\not\in 2\Z
    \end{cases},\qquad
    b_n = \begin{cases}
        \tau_0 - \tau_0^{-1} & n\in2\Z\\
        \widetilde{\tau}_0 - \widetilde{\tau}_0^{-1} & n\not\in2\Z
    \end{cases}.\qedhere
    \]
\end{proof}

\begin{corollary}
    The $h=0$-component of $\im(\eta)$ is
    \[
        K\qty[\frac{abcd}{q}T^2 + T^{-2}].
    \]
\end{corollary}
\begin{proof}
    Note that $\tau_0^2\tau_1^2 = \frac{abcd}{q}$, so multiplying
    $\eta(L)$ by $\tau_0\tau_1$ yields the claimed representation as
    a polynomial ring.
\end{proof}

\begin{proposition}\label{prop-eta-fundamental-symmetric-so}
    We have
    \begin{align*}
        \eta(L) &= q^{k\cdot v_1}T^2 + q^{-k\cdot v_1}T^{-2}\\
        \eta(G^q_+) &= t^{v_1}\qty(T-T^{-1})\\
        \eta(G^q_-) &= t^{-v_1}q^{-\frac{1}{2}}\qty(\frac{abcd}{q^2}T-T^{-1})\\
        \eta(E^q_{12}) &= -t^{-v_2}\qty(\frac{ab}{q}T-T^{-1})\\
        \eta(E^q_{13}) &= -t^{-v_3}\qty(\frac{ac}{q}T-T^{-1})\\
        \eta(E^q_{14}) &= -t^{-v_4}\qty(\frac{ad}{q}T-T^{-1})\\
        \eta(E^q_{23}) &= -t^{v_4}\qty(\frac{bc}{q}T-T^{-1})\\
        \eta(E^q_{24}) &= -t^{v_3}\qty(\frac{bd}{q}T-T^{-1})\\
        \eta(E^q_{34}) &= -t^{v_2}\qty(\frac{cd}{q}T-T^{-1}).
    \end{align*}
\end{proposition}
\begin{proof}
    In light of Corollary \ref{cor-eta-shift-factor}, we can read off $\eta$ for the fundamental shift operators from the concrete way they act on polynomials, see Proposition \ref{prop-shift-factors} for the cases of $G^q_+$, $G^q_-$, and $E^q = E^q_{12}$. The expressions for the other $E^q_{ij}$ can be obtained by permuting
    the parameters.
\end{proof}

\begin{corollary}\label{cor-HC-conjugation}
    For any $h\in\Q^4,\epsilon\in\set{\pm 1},i\in\set{1,\dots,4}$ we have
    \[
        t^{-h} \widetilde{\eta}_{\epsilon v_i}(S_{\epsilon v_i})(T) t^h
        = q^{h\cdot \frac{v_1-\epsilon v_i}{2}}
        \widetilde{\eta}_h(S_{\epsilon v_i})\qty(q^{-\epsilon \frac{h\cdot v_i}{2}} T)
    \]
    in $\mathcal{R}$.
\end{corollary}

\begin{lemma}\label{lem-commute-up-to-power}
    For $v,w\in\set{\pm v_i\where i=1,\dots,4}$ such that $v+w\ne0$, then $S_v\circ S_w = q^s S_w\circ S_v$, where
    \[
        s=\frac{v_1\cdot (w-v)}{2}.
    \]
\end{lemma}
\begin{proof}
    In light of Proposition~\ref{prop-injective-ringhom} it is
    sufficient to prove that $\eta(S_v)$ and $\eta(S_w)$ commute up to powers of $q^{\frac{1}{2}}$. Furthermore, if $v=w$, the claim is trivial. 
    
    From the definition of $\eta$ and Corollary~\ref{cor-HC-conjugation} we obtain
    \begin{align*}
        \eta(S_v)\eta(S_w) &= t^v \widetilde{\eta}_v(S_v)(T)
        t^w \widetilde{\eta}_w(S_w)(T)\\
        &= t^{v+w} q^{w\cdot \frac{v_1-v}{2}}
        \widetilde{\eta}_v(S_v)\qty(q^{-\frac{v\cdot w}{2}}T)\widetilde{\eta}_w(S_w)(T).
    \end{align*}
    Since we excluded the cases where $v=w$ or $v=-w$, we have
    $v\cdot w = 0$. Consequently,
    \[
        \eta(S_v)\eta(S_w)
        = t^{v+w} q^{\frac{v_1\cdot w}{2}}
        \widetilde{\eta}_v(S_v)(T)\widetilde{\eta}_w(S_w)(T)
        = q^{\frac{v_1\cdot (w-v)}{2}}\eta(S_w)\eta(S_v).\qedhere
    \]
\end{proof}

\begin{lemma}\label{lem-relatively-prime}
    The polynomials
    \[
        \widetilde{\eta}_{\epsilon v_i}(S_{\epsilon v_i})(q^{\frac{n}{2}}T),\qquad
        n\in\Z,\epsilon\in\set{\pm1},v_i=1,2,3,4
    \]
    are pairwise relatively prime.
\end{lemma}
\begin{proof}
    The roots of $\widetilde{\eta}_{\epsilon v_i}(S_{\epsilon v_i})\qty(q^{\frac{n}{2}}T)$ are as follows:
    \begin{align*}
        G^q_+:\qquad &\pm q^{-\frac{n}{2}}\\
        G^q_-:\qquad &\pm \frac{q^{1-\frac{n}{2}}}{\sqrt{abcd}}\\
        E^q_{jl}:\qquad &\pm \frac{q^{\frac{1-n}{2}}}{\sqrt{u_ju_l}},
    \end{align*}
    and, in particular, the polynomials split.
    None of them coincide for different choices of $n,\epsilon,i$.
    Consequently, the polynomials are relatively prime.
\end{proof}

\begin{lemma}\label{lem-multiple-property}
    Let $x\in\mathcal{R}$ and recall Notation~\ref{not-fundamental-shift-ops}. If there are $\epsilon_1,\dots,\epsilon_4\in\set{\pm1}$ and
    $n_1,\dots,n_4\in\N_0$ such that
    \[
        \eta(S_{\epsilon_1v_1})^{n_1}\cdots \eta(S_{\epsilon_4v_4})^{n_4}x\in
        \eta(\mathcal{S}(0)),
    \]
    then there is $y\in\im(\mathcal{S}(0))$ such that
    \[
        x=\eta(S_{-\epsilon_1v_1})^{n_1}\cdots
        \eta(S_{-\epsilon_4v_4})^{n_4}y.
    \]
\end{lemma}
\begin{proof}
    Write
    \begin{align*}
        g_{\pm v_i}(a,b,c,d;T) &= \widetilde{\eta}_{\pm v_i}(S_{\pm v_i})\\
        g_{\pm v_i,n}(a,b,c,d;T) &:= g_{\pm v_i}\qty(a,b,c,d;q^{\frac{n}{2}}T).
    \end{align*}
    Let $h=\epsilon_1n_1v_1 + \cdots + \epsilon_4n_4v_4$. Note that
    by Corollary~\ref{cor-HC-conjugation} there is $r\in\Z$ such that
    \[
        \eta(S_{\epsilon_1v_1})^{n_1}\cdots \eta(S_{\epsilon_4v_4})^{n_4}
        =  q^{\frac{r}{2}}\prod_{i=1}^4 \prod_{n=1}^{n_i}
        g_{\epsilon_i v_i,n} t^h.
    \]
    Since $\mathcal{R}$ is an integral domain, we need
    $x\in\mathcal{R}(-h)$, i.e. there is $f\in K\qty[T,T^{-1}]$ such
    that
    \[
        x = t^{-h} f,
    \]
    then
    \[
        \eta(S_{\epsilon_1v_1})^{n_1}\cdots \eta(S_{\epsilon_4v_4})^{n_4}x
        =  q^{\frac{r}{2}}\qty(\prod_{i=1}^4 \prod_{n=1}^{n_i}
        g_{\epsilon_i v_i,n}(T)) f(T)
        = g(T).
    \]
    By assumption, $g$ is a polynomial in $q^{k\cdot v_1}T^2 + q^{-k\cdot v_1}T^{-2}$, hence we have
    \[
        g(T) = g\qty(q^{-k\cdot v_1}T^{-1}).
    \]
    Denote the substitution automorphism mapping $T\mapsto q^{-k\cdot v_1}T^{-1}$ by $\phi$.
    
    Consequently, $\phi(g_{\epsilon_i v_i,n})$
    divides $g$ ($i=1,\dots,4$ and $1\le n\le n_i$) as well. Note
    that
    \[
        \phi(g_{\epsilon_i v_i,n})(T),\quad
        g_{\epsilon_i v_i,n}\qty(q^{-k\cdot v_i}T^{-1}),\quad
        g_{-\epsilon_i v_i,1-n}(T)
    \]
    are all non-zero multiples of each other.
    By Lemma~\ref{lem-relatively-prime}, these polynomials (for
    different $i, n$) are pairwise relatively prime, and also relatively
    prime with any $g_{\epsilon_i v_i,n}$ ($i=1,\dots,4$, $1\le n\le n_i$). Thus
    \[
        \prod_{i=1}^4\prod_{n=0}^{n_i-1}
        g_{-\epsilon_iv_i,-n}(T)
    \]
    divides $f(T)$, say
    \[
        f(T) = \qty(\prod_{i=1}^4\prod_{n=0}^{n_i-1}
        g_{-\epsilon_iv_i,-n}(T))y(T).
    \]
    Note that $g_{\epsilon_iv_i,n}g_{-\epsilon_i v_i,1-n}$ is
    invariant under $\phi$. Consequently, the invariance of $g$ under
    $\phi$ implies the invariance of $y$. Furthermore, all the
    $g_{\pm\epsilon_iv_i,n}$ are odd in $T$, so that $T\mapsto -T$ also
    leaves each factor $g_{\epsilon_iv_i,n}g_{-\epsilon_i v_i,1-n}$
    invariant, so that the invariance $g$ under that automorphism
    shows that $y$ also satisfies $y(-T)=y(T)$. Consequently,
    $y$ is a polynomial in $q^{k\cdot v_1}T^2+q^{-k\cdot v_1}T^{-2}$, i.e. $y\in\eta(\mathcal{S}(0))$. 
    
    Consequently,
    \[
        x = t^{-h}\qty(\prod_{i=1}^4\prod_{n=0}^{n_i-1}
        g_{-\epsilon_iv_i,-n}(T)) y(T)
        = q^{\frac{s}{2}}
        \eta(S_{-\epsilon_1v_1})^{n_1}
        \cdots
        \eta(S_{-\epsilon_4v_4})^{n_4}
        y(T)
    \]
    for an appropriate power $q^{\frac{s}{2}}$ of $q$.
\end{proof}

\begin{theorem}\label{thm-nonzero-shifts}
    Let $h\in\Q^4$. Then $\mathcal{S}(h)$ is a free module over
    $\mathcal{S}(0)$, of rank 1 if $h\cdot v_i\in\Z$ for $i=1,\dots,4$,
    and of rank $0$ otherwise.
\end{theorem}
\begin{proof}
    If $h$ does not lie in the integer span of $v_1,\dots,v_4$, then
    by Lemma~\ref{lem-parity} we have $\mathcal{S}(h)=0$, which is a free $\mathcal{S}(h)$-module. Otherwise, pick $\epsilon_1,\dots,\epsilon_4\in\set{\pm1},n_1,\dots,n_4\in\N_0$ such that
    \[
        h = \epsilon_1n_1v_1+ \cdots + \epsilon_4n_4v_4.
    \]
    If $S\in\mathcal{S}(h)$, then
    \[
        \eta(S_{-\epsilon_1v_1})^{n_1}\cdots
        \eta(S_{-\epsilon_4v_4})^{n_4} \eta(S)
    \]
    is an element of $\eta(\mathcal{S}(0))$ since the appropriate
    composition of shift operators is a shift operator with shift 0.
    By Lemma~\ref{lem-multiple-property}, we then have
    $y\in\eta(\mathcal{S}(0))$, say $y=y'(\eta(L))$, such that
    \[
        \eta(S) = \eta(S_{\epsilon_1v_1})^{n_1}\cdots
        \eta(S_{\epsilon_4v_4})^{n_4}y'(\eta(L)).
    \]
    Since $\eta$ is an injective algebra homomorphism, we thus
    have
    \[
        S = S_{\epsilon_1v_1}^{\circ n_1}\circ\cdots\circ
            S_{\epsilon_4v_4}^{\circ n_4}\circ y'(L).
    \]
    This shows that $S_{\epsilon_1v_1}^{\circ n_1}\circ\cdots\circ S_{\epsilon_4v_4}^{\circ n_4}$ is a generating system of
    $\mathcal{S}(h)$ over $\mathcal{S}(0)$. For linear independence
    we recall that $\eta$ is an injective algebra homomorphism to an
    integral domain, so there is no element of $\mathcal{S}(0)$ that annihilates $S_{\epsilon_1v_1}^{\circ n_1}\circ\cdots\circ S_{\epsilon_4v_4}^{\circ n_4}$.
\end{proof}

\begin{corollary}
    The algebra $\mathcal{S}$ is generated over $K$ by
    \[
        G^q_+,G^q_-,E^q_{ij}\qquad 1\le i<j\le 4.
    \]
\end{corollary}
\begin{proof}
    Follows by Theorems~\ref{thm-zero-shifts}, \ref{thm-nonzero-shifts}, and the fact that
    \[
        E^q_{ij}\circ E^q_{kl} = \sqrt{\frac{abcd}{q}}L
        - \frac{u_ku_l}{q} - u_i u_j.\qedhere
    \]
\end{proof}

\subsection{Formal adjoints}
\begin{definition}\label{def:adjoints_sym_shiftops}
    Let $S,S'$ be shift operators with shifts $h, -h$.
    \begin{enumerate}
        \item $S'$ is said to be the \emph{formal $\dagger$-adjoint} of $S$, denoted by $S^\dagger := S'$, if for all $f,g\in\mathcal{A}_0$ we have
        \[
            \langle S_k f, g\rangle_{k+h} = \langle f, S'_{k+h}g\rangle_{k}.
        \]
        \item $S'$ is said to be the \emph{formal $*$-adjoint}
        of $S$, denoted by $S^* := S'$, if for all $f,g\in\mathcal{A}_0$
        we have
        \[
            \langle S_k f,g\rangle'_{k+h} = \langle f, S'_{k+h}g\rangle'_k.
        \]
    \end{enumerate}
\end{definition}

\begin{remark}\leavevmode
\begin{enumerate}
    \item The way we defined the adjoint relation can be explained as follows:
    we can view $S$ as a family of operators going between different
    inner product spaces:
    \[
        S_k : (\mathcal{A}_0,\langle\cdot,\cdot\rangle_k)
        \to (\mathcal{A}_0,\langle\cdot,\cdot\rangle_{k+h})
    \]
    for all possible choices $k$ of parameters. Then, of the family
    $S'$, the operator that maps back in the other direction is $S'_{k+h}$:
    \[
        S'_{k+h}: (\mathcal{A}_0,\langle\cdot,\cdot\rangle_{k+h})
        \to (\mathcal{A}_0,\langle\cdot,\cdot\rangle_k).
    \]
    So this is the context (similarly for $\langle\cdot,\cdot\rangle'$) in which the notion of adjoint is most natural to define.
\item The fact that $\langle\cdot,\cdot\rangle$ and
    $\langle\cdot,\cdot\rangle'$ are $\Q$-bilinear shows
    that $\dagger$ and $*$ are both $\Q$-linear. Furthermore,
    that these inner products are $K$-bilinear and
    $K$-sesquilinear, respectively, shows that
    the formal $\dagger$- and $*$-adjoints of $c\in K$, a $K$-multiple
    of the identity (and hence a shift operator with shift $0$)
    are $c$ and $c^*$, respectively.

    Since $\dagger$ and $*$ are anti-multiplicative with
    respect to the composition operation, we thus obtain
    \begin{align*}
        (cS)^\dagger_k &= (S\circ c)^\dagger_k
        = \qty(c^\dagger\circ S^\dagger)_k\\
        &= \qty(c\circ S^\dagger)_k
        = c_{k-h} S^\dagger_k,\\
        (cS)^*_k &= (S\circ c)^*_k
        = \qty(c^*\circ S^*)_k\\
        &= c_{k-h}^* S^*_k,
    \end{align*}
    so up to a twist of all parameters by $-h$, the
    $\dagger$-operation is $K$-linear and the
    $*$-operation is $K$-antilinear.

\item Note furthermore, as already mentioned in Remark \ref{rem:non-symmetric-inner-product}, that $\langle\cdot,\cdot\rangle'$ is not Hermitian, so the adjoint operation $*$ is not involutive. However, by Proposition~\ref{prop:general-facts-inner-products}(iii), $\langle\cdot,\cdot\rangle$ is symmetric, so that $\dagger$ is an involution.
\item The operation $\dagger$ turns $\mathcal{S}$ into a
$*$-algebra over $\Q$. If we adopt the categorical picture
of Remark~\ref{rmk-about-S}(ii), the operation $\dagger$ (and a properly normed operation $*$ that is defined with respect to $(\cdot,\cdot)_1$ instead of $(\cdot,\cdot)$ and hence involutive) turns the category of shift operators into a dagger
category with $K$-linear (resp. anti-linear) involution.
With respect to this picture, $\dagger$ in the algebra
$\mathcal{S}$ can be seen as the corresponding operation in
the category of shift operators, followed by an isomorphism
of morphism spaces that sets the starting labelling to be $k$
again.
\end{enumerate}
\end{remark}

\begin{lemma}\label{lem-fundamental-dagger-adjoints}
    The fundamental shift operators have a formal $\dagger$-adjoint. In
    particular, $(S_{\epsilon v_i})^\dagger = S_{-\epsilon v_i}$ for
    $\epsilon\in\set{\pm1},i\in\set{1,2,3,4}$. Furthermore, $L$ is formally self-adjoint.
\end{lemma}
\begin{proof}
    By \cite[Section~2]{KaMi89}, we have $G^q_-=\qty(G^q_+)^\dagger$
    and $E^q_{34} = \qty(E^q_{12})^\dagger$. By permuting parameters, we
    furthermore obtain $(E^q_{ij})^\dagger = E^q_{kl}$ where
    $\set{i,j,k,l}=\set{1,2,3,4}$. Furthermore, since $\dagger$ is an
    involution, we obtain $(G^q_-)^\dagger = G^q_+$.

    Lastly, for $m,n\in\N_0$ we have
    \begin{align*}
        \langle L_k P_{m,k},P_{n,k}\rangle_k
        &= \qty(q^{k\cdot v_1 + m} + q^{-k\cdot v_1-m})
        \delta_{m,n} \langle P_{m,k},P_{m,k}\rangle_k\\
        &= \qty(q^{k\cdot v_1 + n} + q^{-k\cdot v_1-n})
        \delta_{m,n} \langle P_{n,k},P_{n,k}\rangle_k\\
        &= \langle P_{m,k}, L_k P_{n,k}\rangle_k.
    \end{align*}
    Hence by bilinearity of $\langle L_k\cdot,\cdot\rangle_k$ and
    $\langle \cdot,L_k\cdot\rangle_k$ that $L_k$ is formally self-adjoint with
    respect to $\langle\cdot,\cdot\rangle_k$.
\end{proof}

\begin{lemma}\label{lem-fundamental-adjoints}
    The fundamental shift operators and $L$ have formal $*$-adjoints:
    \begin{align*}
        L^* &= L\\
        \qty(G^q_+)^* &= -\frac{q^3}{abcd}G^q_-\\
        \qty(G^q_-)^* &= -G^q_+\\
        \qty(E^q_{ij})^* &= -\frac{q}{u_ku_l} E^q_{kl}
    \end{align*}
    for $\set{i,j,k,l}=\set{1,2,3,4}$.
\end{lemma}
\begin{proof}
    Let $S$ be a shift operator with formal $\dagger$-adjoint. If $f,g\in\mathcal{A}_0$, then
    \[
        \langle S_k f,g\rangle'_{k+h}
        = \langle S_k f,g^\circ\rangle_{k+h}
        = \langle f, (S^\dagger)_{k+h} g^\circ\rangle_k
        = \langle f, \qty((S^\dagger)_{k+h} g^\circ)^\circ\rangle'_k.
    \]
    Thus, the formal $*$-adjoint of $S$ exists and is given by
    \[
        S^*_{k+h} g = \qty(S^\dagger_{k+h} g^\circ)^\circ.
    \]
    Note that
    \begin{align*}
        \qty(T^n(z^m)^\circ)^\circ &=
        \qty(T^nz^m)^\circ=
        \qty(q^{\frac{mn}{2}}z^m)\\
        &= q^{-\frac{mn}{2}}z^m = T^{-n}z^m,
    \end{align*}
    so that for
    \[
        S^\dagger_k = \sum_{n\in\Z} c_n(z) T^n
    \]
    we have
    \[
        S^*_k = \sum_{n\in\Z} c_{-n}^\circ(z) T^n.
    \]
    Consequently, any shift operator that has a formal $\dagger$-adjoint,
    also has a formal $*$-adjoint. In particular, considering the case
    $c_1(z)=(z-z^{-1})^{-1}$, we obtain $(G^q_-)^*=-G^q_+$. Similarly, for
    $G^q_+$ we have
    \[
        (z-z^{-1})c_1(z) = q^{-\frac{1}{2}}z^{-2}\qty(1-aq^{-\frac{1}{2}}z)\cdots
        \qty(1-dq^{-\frac{1}{2}}z),
    \]
    so that
    \[
        (z-z^{-1})c_1^\circ(z) = -q^{\frac{1}{2}}z^{-2}\qty(1-a^{-1}q^{\frac{1}{2}}z)\cdots
        \qty(1-d^{-1}q^{\frac{1}{2}}z) = -\frac{q^3}{abcd}c_{-1}\qty(z),
    \]
    so that $(G^q_+)^*=-\frac{q^3}{abcd}G^q_-$. Lastly, for
    \[
        c_1(z) = -z^{-1}\qty(1-u_iq^{-\frac{1}{2}}z)\qty(1-u_jq^{-\frac{1}{2}}z)
    \]
    we have
    \[
        c_1^\circ(z) = -z^{-1}\qty(1-u_i^{-1}q^{\frac{1}{2}}z)
        \qty(1-u_j^{-1}q^{\frac{1}{2}}z)
        = \frac{q}{u_iu_j} c_{-1}(z),
    \]
    which shows that
    \[
        (E^q_{ij})^* = -\frac{q}{u_k u_l} E^q_{kl}
    \]
    for $\set{i,j,k,l}=\set{1,2,3,4}$. Lastly, note that by
    \cite[\S5.3.2]{Mac03} we have $P_{m,k}^\circ = P_{m,k}$, so that
    \begin{align*}
        \qty(L_k P_{m,k}^\circ)^\circ &=
        \qty(L_k P_{m,k})^\circ
        = \qty(q^{k\cdot v_1 + m}+q^{-k\cdot v_1 - m})^\circ
        P_{m,k}\\
        &= \qty(q^{k\cdot v_1 + m}+q^{-k\cdot v_1 - m})
        P_{m,k}
        = L_k P_{m,k},
    \end{align*}
    so that $L_k$ is also formally self-adjoint with respect to $*$.
\end{proof}

\begin{corollary}\label{cor-adjoints-exist}
    All shift operators have formal adjoints.
\end{corollary}
\begin{proof}
    This follows from Theorems~\ref{thm-zero-shifts} and
    \ref{thm-nonzero-shifts} as well as Lemmas~\ref{lem-fundamental-adjoints} and \ref{lem-fundamental-dagger-adjoints}.
\end{proof}

\section{Non-Symmetric Shift Operators}\label{sec:nonsym-shift-op}
\begin{definition}\leavevmode
\begin{enumerate}
    \item A difference-reflection operator is a formal expression
    \[
        S = \sum_{n\in\Z} c_n(z) T^n + \sum_{n\in\Z} d_n(z) T^n s_1,
    \]
    where $c_n,d_n\in K(z)$ for $n\in\Z$ and $c_n=d_n=0$ for
    all but finitely many $n\in\Z$ satisfying
    $S\mathcal{A}\subset\mathcal{A}$. Recall that $s_1$ is the
    reflection acting as $s_1z^n = z^{-n}$.
    \item Let $Y_k$ be the non-symmetric Cherednik operator (e.g.
    \cite[\S6.6.9]{Mac03}) for the labelling $k$.
    A difference-reflection operator $S$ is said to be a
    \emph{non-symmetric shift operator} with shift $h$ if
    \[
        SY_k = Y_{k+h}S.
    \]
    Write $\mathcal{S}_{\mathrm{ns}}(h)$ for the $K$-vector space
    of non-symmetric shift operators with shift $h$.
\end{enumerate}
\end{definition}

\begin{proposition}
    Let $S$ be a difference-reflection operator. Then $S\in\mathcal{S}_{\mathrm{ns}}(h)$ if and only if for every $n\in\Z$ there exists $C\in K$ such that
    \[
        SE_{n,k} = \begin{cases}
            C E_{n-h\cdot v_1, k+h} & n>0, n-h\cdot v_1>0\\
            C E_{n+h\cdot v_1, k+h} & n\le0, n+h\cdot v_1 \le 0\\
            0 & \text{otherwise}
        \end{cases}.
    \]
\end{proposition}
\begin{proof}
    Note that
    \[
        Y_k E_{n,k} = \begin{cases}
            q^{-n-k\cdot v_1} E_{n,k} & n>0\\
            q^{-n+k\cdot v_1} E_{n,k} & n\le 0
        \end{cases},
    \]
    which shows that $Y_k$ and hence also
    $Y_{k+h}$ has one-dimensional eigenspaces.
    
    \enquote{$\Rightarrow$} It is therefore
    sufficient to show that $SE_{n,k}$ are eigenvectors. We have
    \[
        Y_{k+h}SE_{n,k} = SY_kE_{n,k}
        = \begin{cases}
            q^{-n-k\cdot v_1} SE_{n,k} & n>0\\
            q^{-n+k\cdot v_1} SE_{n,k} & n\le0
        \end{cases}.
    \]
    This shows that $SE_{n,k}$ lies in the eigenspace of
    $Y_{k+h}$ with eigenvalue
    \[
        q^{-n\mp k\cdot v_1}
        = q^{-(n\mp h\cdot v_1) \mp (k+h)\cdot v_1}.
    \]
    This eigenvalue does not exist for $n>0$ and $n-h\cdot v_1\le 0$ or $n\le 0$ and $n+h\cdot v_1>0$, so in these
    cases we must have $SE_{n,k}=0$. In all other cases,
    we must have
    \[
        SE_{n,k} \in \begin{cases}
            K E_{n-k\cdot v_1,k+h} & n>0\\
            K E_{n+k\cdot v_1,k+h} & n\le 0
        \end{cases}.
    \]
    \enquote{$\Leftarrow$}: It suffices to show that
    $Y_{k+h}S$ and $SY_k$ coincide when applied to $E_{n,k}$
    for all $n\in\Z$. Let $n\in\Z$. If $n>0$ and
    $n-h\cdot v_1>0$ or
    $n\le 0$ and $n+h\cdot v_1\le0$, there is $C\in K$ such that
    \begin{align*}
        Y_{k+h}SE_{n,k} &=
        CY_{k+h}E_{n\mp h\cdot v_1,k+h}\\ 
        &= C\qty(q^{n\mp h\cdot v_1 \pm (k+h)\cdot v_1}
        + q^{-n\pm h\cdot v_1 \mp (k+h)\cdot v_1}) E_{n\mp h\cdot v_1,k+h}\\
        &= C\qty(q^{n\pm k\cdot v_1} + q^{-n\mp k\cdot v_1})E_{n\mp h\cdot v_1,k+h}\\
        &= \qty(q^{n\mp k\cdot v_1} + q^{-n\mp k\cdot v_1})
        SE_{n,k}\\
        &= SY_k E_{n,k}.
    \end{align*}
    In case $n>0$ and $n-h\cdot v_1\le0$ or
    $n\le0$ and $n+h\cdot v_1>0$, both sides equal zero.
\end{proof}

\begin{definition}
    Let $S\in\mathcal{S}_{\mathrm{ns}}(h),S'\in\mathcal{S}_{\mathrm{ns}}(-h)$. $S'$ is said to be formally adjoint to $S$ if for all $f,g\in\mathcal{A}_0$
    \[
        (S_kf, g)_{k+h} = (f, S'_{k+h}g)_k.
    \]
    In that case we write $S^* := S'$.
\end{definition}

In order to describe these non-symmetric shift operators and
relate them to their symmetric counterparts from the
previous section, we consider the two examples for vector-valued interpretations (the first of which slightly modified)
of the non-symmetric AW-polynomials that are considered in \cite[Section~6.1]{Sch23}. 

We note that $\mathcal{A}$ is a free $\mathcal{A}_0$-module.
In particular, we consider the basis $\set{1,z}$, henceforth referred to as the \emph{Steinberg basis}, and the basis
$\set{1,z^{-1}(1-az)(1-bz)}$, henceforth referred to as the \emph{Koornwinder basis}. Both are considered (the Steinberg basis up to inverses) in \cite[Section~6.1]{Sch23}, and the Koornwinder basis was studied in \cite[Section~6.1]{KoMa18} and was first considered by Koornwinder and Bouzeffour in \cite{KoBo11}.

It is shown in \cite[Section~6.1]{Sch23} that in both bases the
inner product $(\cdot,\cdot)_k$ can be expressed in terms of
a matrix weight function that is similar to a diagonal matrix
of $\nabla$ (the weight for the inner product $\langle\cdot,\cdot\rangle$ that is used to define the symmetric AW-polynomials) for different labellings. The first goal of this
section is therefore to establish a connection between the vector representations of the non-symmetric AW-polynomials $(E_n)_{n\in\Z}$ and the symmetric AW-polynomials $(P_m)_{m\in\N_0}$. This is achieved using the
general theory of matrix-valued orthogonal polynomials.

\begin{notation}
    We write
    \begin{align*}
        \mathcal{B}_{\st} : \mathcal{A}_0^2 \to \mathcal{A},\qquad
        \mqty(f_1(z)\\f_2(z))&\mapsto f_1(z) + zf_2(z)\\
        \mathcal{B}_{\ko,k}: \mathcal{A}_0^2\to
        \mathcal{A},\qquad
        \mqty(f_1(z)\\f_2(z))&\mapsto f_1(z) + z^{-1}(1-az)(1-bz) f_2(z),
    \end{align*}
    or as matrices:
    \[
        \mathcal{B}_{\st} = \mqty(1 & z),\qquad
        \mathcal{B}_{\ko,k} = \mqty(1 & z^{-1}(1-az)(1-bz)).
    \]
\end{notation}

\begin{proposition}\label{prop-inverse-basis-transforms}
    The inverses of $\mathcal{B}_{\st},\mathcal{B}_{\ko}$ are given by
    \begin{align*}
        \mathcal{B}_{\st}^{-1}f(z) &= \frac{1}{z-z^{-1}}\mqty(zf\qty(z^{-1})- z^{-1}f(z)\\f(z) - f\qty(z^{-1}))\\
        \mathcal{B}_{\ko,k}^{-1}f(z) &=
        \frac{1}{ab-1}\frac{1}{z-z^{-1}}
        \mqty(z^{-1}(1-az)(1-bz)f\qty(z^{-1}) - z\qty(1-az^{-1})\qty(1-bz^{-1})f(z)\\
        f(z) - f\qty(z^{-1})),
    \end{align*}
    or as matrices
    \begin{align*}
        \mathcal{B}_{\st}^{-1} &=
        \frac{1}{z-z^{-1}}\mqty(zs_1 - z^{-1}\\1-s_1)\\
        \mathcal{B}_{\ko,k}^{-1} &=
        \frac{1}{(ab-1)\qty(z-z^{-1})}
        \mqty(z^{-1}(1-az)(1-bz)s_1 - z\qty(1-az^{-1})\qty(1-bz^{-1})\\1-s_1).
    \end{align*}
\end{proposition}
\begin{proof}
    Over the fraction field of $\mathcal{A}$ it can be verified in a quick computation that
    the inverses of $\mathcal{B}_{\st},\mathcal{B}_{\ko,k}$ are indeed given by the expressions above.
    Note that $\mathcal{B}_{\st}^{-1},\mathcal{B}_{\ko,k}^{-1}$ as given above indeed map
    $\mathcal{A}$ to $\mathcal{A}_0^2$: the polynomials
    \begin{align*}
        (zs_1 &- z^{-1})f(z)\\
        (1&-s_1)f(z)\\
        \Big(z^{-1}(1-az)(1-bz)s_1 &- z\qty(1-az^{-1})\qty(1-bz^{-1})\Big)f(z)
    \end{align*}
    are all anti-symmetric under $s_1$ and hence elements of $\qty(z-z^{-1}) K\qty[z+z^{-1}]$.
\end{proof}

\subsection{Inner Products}
Using these isomorphisms we will now transfer various structures
from $\mathcal{A}$ to $\mathcal{A}_0^2$. We shall first
consider the inner products.

\begin{definition}
    Define the inner products on $\mathcal{A}_0^2$:
    \[H_{\st,k}(\cdot,\cdot) := (\mathcal{B}_{\st}\cdot,
        \mathcal{B}_{\st}\cdot)_k,\qquad
        H_{\ko,k}(\cdot,\cdot) :=
        (\mathcal{B}_{\ko,k}\cdot,\mathcal{B}_{\ko,k}\cdot)_k.
    \]
\end{definition}

\begin{proposition}\label{prop-diagonal-matrix-weights}
    Let $s\in\set{\st,\ko}$. For all $f,g\in \mathcal{A}_0^2$ we have
    \[
        H_{s,k}(f,g) = \ct(f^T \mathcal{W}_{s,k} g^*),
    \]
    where $g^*$ denotes the component-wise application of $*$ to $g$ for the following matrix weights:
    \begin{align*}
        \mathcal{W}_{\st,k} &= \frac{1}{2}
        \mqty(1-ab & a + b - ab\qty(z + z^{-1})\\
        z+z^{-1} - a - b & 1-ab)\nabla_k\\
        \mathcal{W}_{\ko,k} &=
        -\frac{ab-1}{2ab}\mqty(ab\nabla_{k} & 0\\0 & -\nabla_{k+v_1+v_2}),
    \end{align*}
    whereby in particular
    \[
        V_k^T\mathcal{W}_{\st,k} V_k^*
        = -\frac{a-b}{2ab}\mqty(-b\nabla_{k+\epsilon_1} & 0\\0 & a\nabla_{k+\epsilon_2}),
    \]
    with
    \[
        V_k = \mqty(-a^{-1} & -b^{-1}\\1 & 1).
    \]
\end{proposition}
\begin{proof}
    The shape of the Steinberg matrix weight follows from
    \cite[Lemma~6.5]{Sch23}, where we exchange the off-diagonal
    terms to account for taking the inverted bases. The
    shape of the Koornwinder weight follows from \cite[Lemma~6.10]{Sch23}.

    For the similarity transformation, note that
    \[
        V_k^T\mathcal{W}_{\st,k}V_k^*
        = \frac{a-b}{2}\mqty(
        - a^{-1}\qty(1-az)\qty(1-az^{-1}) & 0\\0 &
        b^{-1}\qty(1-bz)\qty(1-bz^{-1}))\nabla_k.
    \]
    Note that as power series in $e^{a_0},e^{a_1}$,
    $\nabla_k$ is invertible and that we have
    \[
        \frac{\nabla_{k+\epsilon_1}}{\nabla_k}
        = (1-az)(1-az^{-1}),\qquad
        \frac{\nabla_{k+\epsilon_2}}{\nabla_k} = (1-bz)(1-bz^{-1}).
    \]
    In particular, both $\nabla_{k+\epsilon_1},\nabla_{k+\epsilon_2}$ are well-defined
    elements of $\mathcal{R}[[e^{a_0},e^{a_1}]]$ and hence also
    of $\mathcal{R}((q^{1/2}))[[z^{\pm1}]]$. Moreover, we see that
    \[
         V^T \mathcal{W} V^*
         = \frac{a-b}{2}\mqty(-a^{-1} \nabla_{k+\epsilon_1} & 0\\0
        & b^{-1}\nabla_{k+\epsilon_2}),
    \]
    which yields the claim.
\end{proof}

\begin{definition}
    For $s\in\set{\st,\ko}$ define the following
    matrix inner products on $\mathcal{A}_0^{2\times 2}$:
    \[
        \widehat{H}_{s,k}(F,G) := \ct(F^T \mathcal{W}_{s,k} G^*),
    \]
    where, again, $G^*$ denotes the entry-wise application of $*$ to $G$. With respect to $A,B\in K^{2\times 2}$ it satisfies the following sesquilinearity condition:
    \[
        \widehat{H}_{s,k}(FA,GB) = A^T \widehat{H}_{s,k}(F,G) B^*.
    \]
\end{definition}

\begin{remark}\label{rmk-matrix-inner-product}
    If $F=\mqty(f_1 & f_2)$ and $G=\mqty(g_1 & g_2)$ and $s\in\set{\st,\ko}$, we have
    \[
        \widehat{H}_{s,k}(F,G) = \mqty(
        H_{s,k}(f_1,g_1) & H_{s,k}(f_1,g_2)\\
        H_{s,k}(f_2,g_1) & H_{s,k}(f_2,g_2)).
    \]
\end{remark}

\begin{definition}
    For $m\in\N_0$ define
    \begin{align*}
        \mathbb{E}_{\st,m,k} &:= \mqty(\mathcal{B}_{\st}^{-1}E_{-m,k} & \mathcal{B}_{\st}^{-1}E_{m+1,k})\\
        \mathbb{E}_{\ko,m,k} &:= \mqty(\mathcal{B}_{\ko,k}^{-1}E_{-m,k} & \mathcal{B}_{\ko,k}^{-1} E_{m,k}).
    \end{align*}
    We note that the occurrence of $E_{m+1,k}$ and $E_{m,k}$ is not a typo. These
    are the \emph{non-symmetric matrix AW-polynomials}.
\end{definition}

\begin{lemma}\label{lem-es-orthogonal}
    For $s\in\set{\st,\ko}$ and $m,n\in\N_0$ we have
    \[
        \widehat{H}_{s,k}(\mathbb{E}_{s,m,k},\mathbb{E}_{s,n,k}) = \delta_{m,n} H_{s,n,k},
    \]
    where $H_{s,n,k}$ is an invertible diagonal matrix.
\end{lemma}
\begin{proof}
    We have
    \begin{align*}
        \widehat{H}_{\st,k}(\mathbb{E}_{s,m,k},\mathbb{E}_{s,n,k})
        &= \mqty((E_{-m}, E_{-n}) & (E_{-m}, E_{n+1})\\
        (E_{m+1}, E_{-n}) & (E_{m+1}, E_{n+1}))\\
        &= \mqty((E_{-m},E_{-n}) & 0\\0 & (E_{m+1}, E_{n+1})),
    \end{align*}
    where we omitted the labelling $k$. This matrix is diagonal,
    and equals $0$ if $m\ne n$, otherwise it is invertible.

    Similarly, we obtain
    \[
        \widehat{H}_{\ko,k}(\mathbb{E}_{\ko,m,k},\mathbb{E}_{\ko,n,k})
        = \mqty((E_{-m},E_{-n}) & 0\\0 & (E_m, E_n)),
    \]
    where, again, we omitted the labelling $k$. This matrix is also
    diagonal, and equals $0$ if $m\ne n$, otherwise it is invertible.
\end{proof}

\subsection{Matrix Shift Operators}
The similarity of the Steinberg and Koornwinder weight functions $\mathcal{W}_{\st},\mathcal{W}_{\ko}$ to diagonal matrices of $\nabla$ (which is the weight function for $\langle\cdot,\cdot\rangle$ that is used to define the symmetric AW-polynomials) for different parameters suggests a connection between the non-symmetric matrix AW-polynomials and matrices of symmetric AW-polynomials.

To prepare the discussion of non-symmetric shift operators, we shall therefore discuss appropriate matrices of shift operators.

\begin{definition}
    We define the \emph{symmetric matrix AW-polynomials} for the Steinberg and Koornwinder basis as follows:
    \[
        \mathbb{P}_{\st,m,k} :=
        \mqty(P_{m,k+\epsilon_1} & 0\\0 & P_{m,k+\epsilon_2}),\qquad
        \mathbb{P}_{\ko,m,k} :=
        \mqty(P_{m,k} & 0\\0 & P_{m-1,k+v_1+v_2}).
    \]

    A matrix $X$ of symmetric shift operators is said to be a \emph{$s$-matrix shift operator} with shift $h$ if $s\in\set{\st,\ko}, h\in\Q^4$ and for every $m\in\N_0$ there is a matrix $C\in K^{2\times 2}$ (depending on $s,m,k$) such that
    \[
        X\mathbb{P}_{s,m,k} = \mathbb{P}_{s,m-h\cdot v_1,k+h}C,
    \]
    where we take $\mathbb{P}_{s,m,k}=0$ for $m<0$. The set of $s$-matrix shift operators with shift $h$ is denoted by $\widehat{\mathcal{S}}_s(h)$, which is a $K$-vector space.
\end{definition}

\begin{proposition}\label{prop-structure-mso}
    Let $X$ be a matrix of symmetric shift operators. For
    $s\in\set{\st,\ko}$ and $h\in\Q^4$ we have that
    $X\in\widehat{\mathcal{S}}_s(h)$ if and only if
    \begin{description}
        \item[Steinberg] there are symmetric shift operators
        $S_1,S_2,S_3,S_4$, where $S_1,S_4$ have shift $h$,
        $S_2$ has shift $h+v_3+v_4$, and $S_3$ has shift $h-v_3-v_4$ such that
        \[
            X = \mqty(S_{1,k+\epsilon_1} & S_{2,k+\epsilon_2}\\
            S_{3,k+\epsilon_1} & S_{4,k+\epsilon_2});
        \]
        \item[Koornwinder] there are symmetric shift operators
        $S_1,S_2,S_3,S_4$, where $S_1$, $S_4$ have shift $h$,
        $S_2$ has shift $h-v_1-v_2$, and $S_3$ has shift
        $h+v_1+v_2$ such that
        \[
            X = \mqty(S_{1,k} & S_{2,k+v_1+v_2}\\
            S_{3,k} & S_{4,k+v_1+v_2}).
        \]
    \end{description}
\end{proposition}
\begin{proof}
    We prove the Steinberg case; the Koornwinder case follows analogously.
    Writing $C=\mqty(C_1 & C_2\\C_3 & C_4)$ and $X=\mqty(X_1 & X_2\\X_3 & X_4)$, we have that the equation $X\mathbb{P}_{\st,m,k} = \mathbb{P}_{\st,m-h\cdot v_1,k+h}C$
    reads as follows:
    \begin{align*}
        \mqty(X_1 & X_2\\X_3 & X_4)
        \mqty(P_{m,k+\epsilon_1} & 0\\0 & P_{m,k+\epsilon_2})
        &= \mqty(X_1P_{m,k+\epsilon_1} & X_2P_{m,k+\epsilon_2}\\
        X_3P_{m,k+\epsilon_1} & X_4 P_{m,k+\epsilon_2})\\
        &= \mqty(C_1P_{m,k+\epsilon_1} & C_2P_{m,k+\epsilon_1}\\
        C_3P_{m,k+\epsilon_2} & C_4 P_{m,k+\epsilon_2}),
    \end{align*}
    which splits up into four equations that (up to shifted
    parameters) are exactly the defining relations of shift
    operators with the claimed shifts.
\end{proof}

\begin{corollary}\label{cor-mso-shifts}
    Let $s\in\set{\st,\ko}$. If $\widehat{\mathcal{S}}_s(h)\ne0$, then $h$ is contained in the $\Z$-span of $v_1,\dots,v_4$.
\end{corollary}
\begin{proof}
    This follows from Proposition~\ref{prop-structure-mso} as
    well as Lemma~\ref{lem-parity}.
\end{proof}

\begin{definition}
    For $s\in\set{\st,\ko}$ and $h\in\Q^4$ define
    \[
        \widetilde{\eta}_{s,h}: \widehat{\mathcal{S}}_s(h)
        \to K^{2\times 2}\qty[T,T^{-1}]
    \]
    as follows, say $X=\mqty(X_1 & X_2\\X_3 & X_4)\in\widehat{\mathcal{S}}_s(h)$ with 
    \[
        X_i = \sum_{n=0}^\infty z^{d_i-n} f_{i,n}(T),
    \]
    then
    \begin{align*}
        \widetilde{\eta}_{\st,h}(X) &:= 
        \mqty(f_{1,0}(T) & f_{2,0}(T)\\f_{3,0}(T)&f_{4,0}(T))\\
        \widetilde{\eta}_{\ko,h}(X) &:=
        \mqty(f_{1,0}(T) & f_{2,0}\qty(q^{-\frac{1}{2}}T) \\ f_{3,0}(T)
        &f_{4,0}\qty(q^{-\frac{1}{2}}T)).
    \end{align*}
\end{definition}

\begin{corollary}
    Let $s\in\set{\st,\ko}$, $h\in\Q^4$, and
    $X\in\widehat{\mathcal{S}}_s(h)$. For $m\in\N_0$ we then
    have
    \[
        X\mathbb{P}_{s,m,k} = \mathbb{P}_{s,m-h\cdot v_1,k+h}
        \tilde{\eta}_{s,h}(X)\qty(q^{\frac{m}{2}}).
    \]
\end{corollary}
\begin{proof}
    This follows from Proposition~\ref{prop-structure-mso}.
    In particular, if we decompose
    \[
        X = \mqty(S_{1,k+\epsilon_1} & S_{2,k+\epsilon_2}\\
        S_{3,k+\epsilon_1} & S_{4,k+\epsilon_2})\qquad\text{or}\qquad
        \mqty(S_{1,k} & S_{2,k+v_1+v_2}\\S_{3,k} & S_{4,k+v_1+v_2})
    \]
    for appropriate shift operators $S_1,\dots,S_4$, then
    $\tilde{\eta}_{s,h}$ can be derived from the
    $\tilde{\eta}_{h_i}(S_i)$, where $h_i$ is the shift of $S_i$.
\end{proof}

\begin{definition}
    Let $s\in\set{\st,\ko}$. If $X\in\widehat{\mathcal{S}}_s(h)$ and
    $X'\in\widehat{\mathcal{S}}_s(h')$, then define
    \[
        X\circ X' := X_{k+h'} X_k,
    \]
    as the matrix product of symmetric difference operators,
    where we shifted the labelling of those of $X$ by $h$. This
    is $\Q$-bilinear and defines a $\Q$-algebra structure on
    \[
        \widehat{\mathcal{S}}_s := \bigoplus_h \widehat{\mathcal{S}}_s(h).
    \]
\end{definition}

\subsection{Descending to the Vector Level}
\begin{proposition}\label{prop-relations-E-P}
    For $s\in\set{\st,\ko}$, there is
    $C_{s,k}\in K^{2\times 2}(T)$ such that
    \begin{align*}
        \mathbb{E}_{\st,m,k} &= V_k
        \mathbb{P}_{\st,m,k} V_k^{-1} C_{\st,k}\qty(q^{\frac{m}{2}})\\
        \mathbb{E}_{\ko,m,k} &=
        \mathbb{P}_{\ko,m,k} C_{\ko,k}\qty(q^{\frac{m}{2}})
    \end{align*}
    for $m\in\N_0$ in the first case and $m\in\N$ in the second. Furthermore,
    \[
    \mathbb{E}_{\ko,0,k} = \mathbb{P}_{\ko,0,k} \mqty(1 & 1\\x & y)
    \]
    for all $x,y\in K$.
\end{proposition}
\begin{proof}
    See Appendix~\ref{sec-proof-E-P}. In particular, we have
    \begin{align*}
        V_k^{-1}C_{\st,k}(T) &= \frac{ab}{a-b}
        \mqty(1 & -b^{-1}\frac{\qty(bcT-T^{-1})\qty(bdT-T^{-1})}{abcdT^2 - T^{-2}}\\
        -1 & a^{-1}\frac{\qty(acT-T^{-1})\qty(adT-T^{-1})}{abcdT^2-T^{-2}})\\
        C_{\ko,k}(T) &= \frac{1}{ab-1}
        \mqty(\frac{\qty(abT-T^{-1})\qty(\frac{abcd}{q}T-T^{-1})}{\frac{abcd}{q}T^2-T^{-2}} & -1\\
        -ab\frac{\qty(\frac{cd}{q}T-T^{-1})\qty(T-T^{-1})}{\frac{abcd}{q}T^2-T^{-2}} & 1).\qedhere
    \end{align*}
\end{proof}

\begin{proposition}\label{prop-nsso-mso}
    Let $S$ be a difference-reflection operator and let $h\in\Q^4$. The following are equivalent:
    \begin{enumerate}
        \item $S \in\mathcal{S}_{\mathrm{ns}}(h)$;
        \item The matrix of symmetric difference operators
        \[
            \widehat{S}:=V_{k+h}^{-1}\mathcal{B}_{\st}^{-1}S\mathcal{B}_{\st}V_k
        \]
        is an element of $\widehat{\mathcal{S}}_{\st}(h)$ and the matrix of Laurent polynomials
        \[
            C_{\st,k+h}\qty(q^{-\frac{h\cdot v_1}{2}}T)^{-1}V_{k+h} \widetilde{\eta}_{\st,h}(\widehat{S})(T)V_k^{-1} C_{\st,k}(T)
        \]
        is diagonal;
        \item The matrix of symmetric difference operators
        \[
            \widetilde{S}:= \mathcal{B}_{\ko,k+h}^{-1}
            S\mathcal{B}_{\ko,k}
        \] is contained in $\widehat{\mathcal{S}}_{\ko}(h)$,
        the matrix of Laurent polynomials
        \[
            C_{\ko,k+h}\qty(q^{-\frac{h\cdot v_1}{2}}T)^{-1}\widetilde{\eta}_{\ko,h}(\widehat{S})(T)C_{\ko,k}(T)
        \]
        is diagonal, and if $h\cdot v_1<0$, $\mqty(1\\0)$ is an eigenvector of
        \[
            C_{\ko,k+h}\qty(q^{-\frac{h\cdot v_1}{2}})^{-1}\widetilde{\eta}_{\ko,h}(\widetilde{S})(1).
        \]
    \end{enumerate}
\end{proposition}
\begin{proof}
    \enquote{(i)$\Leftrightarrow$(ii)}: Assume $(i)$ holds. If $m\in\N_0$, then
    \begin{align*}
        \mathcal{B}_{\st}^{-1}S\mathcal{B}_{\st}
        \mathbb{E}_{\st,m,k}
        &= \mathcal{B}_{\st}^{-1}S\mathcal{B}_{\st}
        \mqty(\mathcal{B}_{\st}^{-1}(E_{-m,k})
        & \mathcal{B}_{\st}^{-1}(E_{m+1,k}))\\
        &= \mqty(\mathcal{B}_{\st}^{-1}(SE_{-m,k})
        & \mathcal{B}_{\st}^{-1}(SE_{m+1,k}))\\
        &= \mqty(C_1\mathcal{B}_{\st}^{-1}(E_{-m+h\cdot v_1,k+h}) & C_2\mathcal{B}_{\st}^{-1}(E_{m+1-h\cdot v_1,k+h}))\\
        &= \mathbb{E}_{\st,m-h\cdot v_1,k+h}
        \mqty(C_1 & 0\\0 & C_2),
    \end{align*}
    where
    \begin{align*}
        SE_{-m,k} &= C_1 E_{-m+h\cdot v_1,k+h}\\
        SE_{m+1,k} &= C_2 E_{m+1-h\cdot v_1,k+h},
    \end{align*}
    with $C_1=C_2=0$ if $m<h\cdot v_1$.
    This shows that
    \begin{align*}
        \widehat{S}\mathbb{P}_{\st,m,k} 
        &= V_{k+h}^{-1}\mathcal{B}_{\st}^{-1}S\mathcal{B}_{\st}V_k \mathbb{P}_{\st,m,k}\\
        &= V_{k+h}^{-1}\mathcal{B}_{\st}^{-1}
        S\mathcal{B}_{\st}
        \mathbb{E}_{\st,m,k} C_{\st,k}\qty(q^{\frac{m}{2}})^{-1}
        V_k\\
        &= V_{k+h}^{-1}\mathbb{E}_{\st,m-h\cdot v_1,k+h}
        \mqty(C_1 & 0\\0 & C_2)C_{\st,k}\qty(q^{\frac{m}{2}})^{-1}
        V_k\\
        &= \mathbb{P}_{\st,m-h\cdot v_1,k+h}
        V_{k+h}^{-1}C_{\st,k+h}\qty(q^{\frac{m-h\cdot v_1}{2}})
        \mqty(C_1 & 0\\0 & C_2)C_{\st,k}\qty(q^{\frac{m}{2}})^{-1}V_k,
    \end{align*}
    which, in particular, equals $0$ if $m<h\cdot v_1$. Moreover, we
    obtain that
    \[
        C_{\st,k+h}\qty(q^{\frac{m-h\cdot v_1}{2}})^{-1}
            V_{k+h} \widetilde{\eta}_{\st,h}(\widehat{S})\qty(q^{\frac{m}{2}})
            V_k^{-1} C_{\st,k}\qty(q^{\frac{m}{2}})
            = \mqty(C_1 & 0\\0 & C_2),
    \]
    which is diagonal. Consequently, the off-diagonal terms of
    \[
    C_{\st,k+h}\qty(q^{-\frac{h\cdot v_1}{2}}T)^{-1}
            V_{k+h} \widetilde{\eta}_{\st,h}(\widehat{S})(T)
            V_k^{-1} C_{\st,k}(T)
    \]
    have infinitely many roots, hence they have to be identically $0$. All steps are reversible, hence the
    converse also holds.
    
    \enquote{(i)$\Leftrightarrow$(iii)}: $S$ is a non-symmetric shift operator if for every $m\in\N_0$ there is a matrix $\Lambda$ such that
    \[
        \widetilde{S}\mathbb{E}_{\ko,m,k} = \mathbb{E}_{\ko,m-h\cdot v_1,k+h}\Lambda,
    \]
    where $\Lambda$ is diagonal if $m\ne0$ or a scalar multiple of $\mqty(1 & 1\\0 & 0)$ if $m=0$.
    
    For the cases where neither $m$ nor $m-h\cdot v_1$ equals $0$, this
    is equivalent to the diagonality condition for $\widetilde{\eta}_{\ko,h}(\widetilde{S})(T)$ using Proposition~\ref{prop-relations-E-P}.

    For the remaining cases: If $h\cdot v_1=0$ and $m=0$, any K-matrix shift operator
    \[
        \widetilde{S} = \mqty(S_{1,k} & S_{2,k+v_1+v_2}\\
S_{3,k} & S_{4,k+v_1+v_2})
    \]
    with shift $h$ satisfies
    \[
        \widetilde{S}\mathbb{E}_{\ko,0,k} =\mqty(S_{1,k}1 & S_{1,k}1\\0 & 0) = \mathbb{E}_{\ko,0,k+h}\mqty(\widetilde{\eta}(S_1)(1) & \widetilde{\eta}(S_1)(1)\\0 & 0),
    \]
    which is of the shape $\mathbb{E}_{\ko,0,k+h}\Lambda$ for $\Lambda$
    a scalar multiple of $\mqty(1 & 1\\0 & 0)$.

    If $h\cdot v_1>0$ and $m=h\cdot v_1$, we have
    \begin{align*}
        \widetilde{S}\mathbb{E}_{\ko,h\cdot v_1,k}
        &= \widetilde{S}\mathbb{P}_{\ko,h\cdot v_1,k}
        C_{\ko,k}\qty(q^{\frac{h\cdot v_1}{2}})\\
        &= \mathbb{P}_{\ko,0,k+h}
        \qty(\widetilde{\eta}_{\ko,h}(\widetilde{S})
        C_{\ko,k})\qty(q^{\frac{h\cdot v_1}{2}})\\
        &= \mqty(1 & 0\\0 & 0)\mqty(\ast & \ast\\\ast & \ast)\\
        &= \mqty(\ast & \ast\\0 & 0).
    \end{align*}
    Any matrix of the above shape can be written as
    $\mathbb{E}_{\ko,0,k+h}\Lambda$ for a diagonal $\Lambda$.

    If $h\cdot v_1 < 0$ and $m=0$, we have
    \begin{align*}
        \widetilde{S}\mathbb{E}_{\ko,0,k} &= \mqty(S_{1,k}1 & S_{1,k}1\\
        S_{3,k}1 & S_{3,k}1)
        = \mathbb{P}_{\ko,-h\cdot v_1,k+h}\mqty(\widetilde{\eta}_h(S_1)(1) & \widetilde{\eta}_h(S_1)(1)\\
        \widetilde{\eta}_{h+v_1+v_2}(S_3)(1) & \widetilde{\eta}_{h+v_1+v_2}(S_3)(1))\\
        &= \mathbb{P}_{\ko,-h\cdot v_1,k+h}
        \widetilde{\eta}_{\ko,h}(\widetilde{S})(1)\mqty(1 & 1 \\0 & 0)\\
        &= \mathbb{E}_{\ko,-h\cdot v_1,k+h}
        C_{\ko,k+h}\qty(q^{-\frac{h\cdot v_1}{2}})^{-1}
        \widetilde{\eta}_{\ko,h}(\widetilde{S})(1)\mqty(1 & 1 \\0 & 0).
    \end{align*}
    This last matrix
    \[
        C_{\ko,k+h}\qty(q^{-\frac{h\cdot v_1}{2}})^{-1}
        \widetilde{\eta}_{\ko,h}(\widetilde{S})(1)\mqty(1 & 1 \\0 & 0)
    \]
    is a scalar multiple of $\mqty(1 & 1\\0 & 0)$ if and only if
    $\mqty(1\\0)$ is an eigenvector of
    \[
        C_{\ko,k+h}\qty(q^{-\frac{h\cdot v_1}{2}})^{-1}
        \widetilde{\eta}_{\ko,h}(\widetilde{S})(1).
    \]
\end{proof}

\begin{corollary}\label{cor-allowed-nsso-shifts}
    $\mathcal{S}_{\mathrm{ns}}(h)\ne0$ is only possible if $h$ lies
    in the $\Z$-span of $v_1,\dots,v_4$.
\end{corollary}
\begin{proof}
    This follows from Proposition~\ref{prop-nsso-mso} and Corollary~\ref{cor-mso-shifts}.
\end{proof}

\begin{definition}
    An S-matrix shift operator $\widehat{S}$ with shift $h$
    is said to \emph{descend to the vector level} if
    it satisfies condition (ii) from Proposition~\ref{prop-nsso-mso}.

    A K-matrix shift operator $\widetilde{S}$ with shift
    $h$ is said to \emph{descend to the vector level} if it
    satisfies condition (iii) from Proposition~\ref{prop-nsso-mso}.
\end{definition}

\begin{proposition}\label{prop-matrix-Y}
    We have
    \begin{align*}
        V_k^{-1}\mathcal{B}_{\st}^{-1}Y_k\mathcal{B}_{\st}V_k
        &= \frac{q^{-\frac{1}{2}}}{a-b}
        \mqty(a L_{k+\epsilon_1} & 0\\0 & -bL_{k+\epsilon_2})
        - \sqrt{\frac{ab}{cdq}}\frac{c+d}{a-b}\mqty(1 & 0\\0 & -1)\\
        &\quad + \frac{1}{\sqrt{abcdq}(a-b)}\mqty(0 & a(E^q_{23}\circ E^q_{24})_{k+\epsilon_2}\\-b(E^q_{14}\circ E^q_{13})_{k+\epsilon_1} & 0)\\
        \mathcal{B}_{\ko,k}^{-1}Y_k\mathcal{B}_{\ko,k}
        &= \frac{1}{ab-1} \mqty(ab L_k & 0\\0 & -L_{k+v_1+v_2})
        - \sqrt{\frac{abq}{cd}}\frac{\frac{cd}{q}+1}{ab-1}
        \mqty(1 & 0\\0 & -1)\\
        &\quad+ \sqrt{\frac{q}{abcd}}\frac{1}{ab-1}
        \mqty(0 & -(E^q_{12}\circ G^q_-)_{k+v_1+v_2}\\ab(G^q_+\circ E^q_{34})_k & 0).
    \end{align*}
\end{proposition}
\begin{proof}
    Call the matrix shift operators on the right-hand side
    $\widehat{Y}_k$ and $\widetilde{Y}_k$, respectively.

    By Proposition~\ref{prop-nsso-mso} it suffices to show that
    \[
        C_{\st,k}\qty(T)^{-1}V_k
        \widetilde{\eta}_{\st,0}(\widehat{Y})(T)
        V_k^{-1}C_{\st,k}(T) = \mqty(\sqrt{\frac{abcd}{q}}T^2 & 0\\0 & \sqrt{\frac{q}{abcd}}T^{-2}),
    \]
    which, when evaluated in $q^{\frac{m}{2}}$, shows that the actions of $V_k\widehat{Y}V_k^{-1}$ and $\mathcal{B}_{\st}^{-1}Y_k\mathcal{B}_{\st}$ agree on the $(\mathbb{E}_{\st,m,k})_{m\in\N_0}$, whence they have
    to equal each other. Recall from Proposition~\ref{prop-relations-E-P} that
    \[
        V_k^{-1}C_{\st,k}(T) = \frac{ab}{a-b}
        \mqty(1 & -b^{-1}\frac{\qty(bcT-T^{-1})\qty(bdT-T^{-1})}{abcdT^2 - T^{-2}}\\
        -1 & a^{-1}\frac{\qty(acT-T^{-1})\qty(adT-T^{-1})}{abcdT^2-T^{-2}}).
    \]
    Note that
    \begin{align*}
        &\frac{\sqrt{abcdq}(a-b)}{ab}\widetilde{\eta}_{\st,0}(\widehat{Y})\\
        = &\mqty(b^{-1}(abcdT^2+T^{-2}) - c-d & b^{-1}(bcT-T^{-1})(bdT-T^{-1})\\-a^{-1}(acT-T^{-1})(adT-T^{-1}) & 
        - a^{-1}(abcdT^2+T^{-2}) + c + d),
    \end{align*}
    so that we indeed obtain
    \[
        \qty(V_k^{-1}C_{\st,k}(T))^{-1}
        \widetilde{\eta}_{\st,0}(\widehat{Y})
        V_k^{-1}C_{\st,k}(T)
        =\mqty(abcdT^2 & 0\\0 & T^{-2}).
    \]
    For the second claim, we proceed analogously, note that
    \begin{align*}
        &\sqrt{\frac{abcd}{q}}\qty(ab-1)\widetilde{\eta}_{\ko,0}(\widetilde{Y})\\
        =&
        \mqty(ab\qty(\frac{abcd}{q}T^2+T^{-2})
        - ab\qty(\frac{cd}{q}+1) & \qty(abT-T^{-1})\qty(\frac{abcd}{q}T-T^{-1})\\
        -ab\qty(T-T^{-1})\qty(\frac{cd}{q}T-T^{-1}) & - \qty(\frac{abcd}{q}T^2+T^{-2}) + ab\qty(\frac{cd}{q}+1)),
    \end{align*}
    so that
    \[
        C_{\ko,k}(T) = \frac{1}{ab-1}
        \mqty(\frac{\qty(abT-T^{-1})\qty(\frac{abcd}{q}T-T^{-1})}{\frac{abcd}{q}T^2-T^{-2}} & -1\\
        -ab\frac{\qty(\frac{cd}{q}T-T^{-1})\qty(T-T^{-1})}{\frac{abcd}{q}T^2-T^{-2}} & 1),
    \]
    and we again obtain the claimed equality.
\end{proof}

\begin{proposition}\label{prop-descend-to-vector-necessary}
    Let $s\in\set{\st,\ko}$ and $h\in\Q^4$. If $X\in\widehat{\mathcal{S}}_s(h)$ descend to the vector level, then
    \[
        \mqty(1 & 1)\widetilde{\eta}_{s,h}(X)\mqty(1\\-1) = 0.
    \]
\end{proposition}
\begin{proof}
    In the Steinberg case,
    \[
        C_{\st,k+h}\qty(q^{-\frac{h\cdot v_1}{2}}T)^{-1}
            V_{k+h} \widetilde{\eta}_{\st,h}(X)(T)
            V_k^{-1} C_{\st,k}(T)
    \]
    has to be diagonal. Note that the first column of
    $V_k^{-1} C_{\st,k}(T)$ is proportional to $\mqty(1\\-1)$ and the
    second row of $(V_k^{-1} C_{\st,k}(T))^{-1}$ (and all of its shifted versions) is proportional to
    $\mqty(1&1)$. Consequently, the expression
    \[
    \mqty(1 & 1)\widetilde{\eta}_{\st,h}(X)\mqty(1\\-1)
    \]
    is the bottom-left entry of a diagonal matrix, hence it equals $0$.

    In the Koornwinder case,
    \[
        C_{\ko,k+h}\qty(q^{-\frac{h\cdot v_1}{2}}T)^{-1}
            \widetilde{\eta}_{\ko,h}(\widehat{X})(T)
            C_{\ko,k}(T)
    \]
    is diagonal. Note that the second column of
    $C_{\ko,k}(T)$ is proportional to $\mqty(1\\-1)$ and the
    first row of $C_{\ko,k}(T)^{-1}$ (and all its shifted versions) is proportional to $\mqty(1 & 1)$, which shows that
    \[
        \mqty(1 & 1)\widetilde{\eta}_{\ko,h}(X)\mqty(1\\-1)
    \]
    is the top-right entry of a diagonal matrix, hence it equals 0.
\end{proof}

\begin{theorem}
    We have
    \[
        \mathcal{S}_{\mathrm{ns}}(0) = K\qty[Y,Y^{-1}].
    \]
\end{theorem}
\begin{proof}
    Evidently, the right-hand side is contained in the left-hand side.

    For the remaining inclusion, it is sufficient to show that every
    S-matrix shift operator with shift $0$ that descends to the vector level is of the
    shape
    \begin{equation}\label{eq-0-component-shape}
        \mqty(f(L_{k+\epsilon_1}) & 0\\0 & f(L_{k+\epsilon_2}))
        + \widehat{Y}_k \mqty(g(L_{k+\epsilon_1}) & 0\\0 & g(L_{k+\epsilon_2}))
    \end{equation}
    for polynomials $f,g$. Write
    \[
        \widehat{L}_k:=\mqty(L_{k+\epsilon_1} & 0\\0 & L_{k+\epsilon_2}).
    \]
    It is easy to verify that $\widehat{L}_k = q^{\frac{1}{2}}\widehat{Y}_k + q^{-\frac{1}{2}}\widehat{Y}_k^{-1}$, which is why it is sufficient to show the decomposition \eqref{eq-0-component-shape}.
    
    If $X\in\widehat{\mathcal{S}}_{\st}(0)$, then by Proposition~\ref{prop-structure-mso} and Theorem~\ref{thm-nonzero-shifts}, there are polynomials $f_1,f_2,f_3,f_4$ such that
    \[
        X = \mqty(f_1(L_{k+\epsilon_1}) & (E^q_{23}\circ E^q_{24})_{k+\epsilon_2} f_2(L_{k+\epsilon_2})\\
        (E^q_{14}\circ E^q_{13})_{k+\epsilon_1}f_3(L_{k+\epsilon_1})
        & f_4(L_{k+\epsilon_2})).
    \]
    Write
    \[
        \omega := \widetilde{\eta}_0(L_{k+\epsilon_1})
        = \widetilde{\eta}_0(L_{k+\epsilon_2})
        = \sqrt{abcd}T^2 + \sqrt{abcd}^{-1}T^{-2}.
    \]
    Then
    \[
        \widetilde{\eta}_{\st,0}(X)
        = \mqty(f_1(\omega) & (bcT-T^{-1})(bdT-T^{-1})f_2(\omega)\\
        (acT-T^{-1})(adT-T^{-1})f_3(\omega) & f_4(\omega)).
    \]
    By Proposition~\ref{prop-descend-to-vector-necessary}, the Laurent polynomial
    \begin{equation}\label{eq-0-component-bl}
        f_1(\omega)  - f_4(\omega)
        - (bcT-T^{-1})(bdT-T^{-1})f_2(\omega)
        + (acT-T^{-1})(adT-T^{-1})f_3(\omega)
    \end{equation}
    is zero. Note that $K\qty[T,T^{-1}]$ is a free $K[\omega]$-module of rank 4, where
    a basis is given by $1,T,T^2,T^3$. In terms of this basis, \eqref{eq-0-component-bl} reads
    \begin{align}
        0 &= f_1(\omega) + (bc+bd-\sqrt{abcd}\omega)f_2(\omega)
        - (ac+ad-\sqrt{abcd}\omega)f_3(\omega) - f_4(\omega)\label{eq-0-component-1st-eq}\\
        &-T^2\qty((b^2cd - abcd)f_2(\omega) - (a^2cd-abcd)f_3(\omega)).
        \label{eq-0-component-2st-eq}
    \end{align}
    $K[\omega]$-Linear independence implies that the terms from
    \eqref{eq-0-component-1st-eq},\eqref{eq-0-component-2st-eq} are
    $0$ independently. From \eqref{eq-0-component-2st-eq} we obtain
    \[
        0 = bf_2(\omega) + af_3(\omega),
    \]
    i.e. $f_3(\omega) = -\frac{b}{a}f_2(\omega)$, which we can insert
    into \eqref{eq-0-component-1st-eq} to obtain
    \[
        0 = f_1(\omega) - f_4(\omega) + \qty(2b(c+d)
        - \frac{a+b}{a}\sqrt{abcd}\omega)f_2(\omega).
    \]
    This shows that
    \begin{align*}
        f_1(\omega) &= \frac{f_1+f_4}{2}(\omega)
        + \qty(\frac{ab}{2a}\sqrt{abcd}\omega - b(c+d))
        f_2(\omega)\\
        f_4(\omega) &= \frac{f_1+f_4}{2}(\omega)
        - \qty(\frac{ab}{2a}\sqrt{abcd}\omega - b(c+d))f_2(\omega),
    \end{align*}
    which implies that
    \[
        X = - \frac{\sqrt{abcdq}(a-b)}{a}
        \widehat{Y}_k f_2(\widehat{L}_k)
        + \frac{f_1+f_4}{2}(\widehat{L}_k) + \frac{b-a}{2a}\sqrt{abcd}\widehat{L}_k.\qedhere
    \]
\end{proof}

\begin{lemma}
    A matrix $X$ of symmetric difference operators is
    \begin{enumerate}
        \item an S-matrix shift operator with shift $h$ that descends to the vector level if and only if $\widehat{Y}_{k+h}X=X\widehat{Y}_k$, or
        \item a K-matrix shift operator with shift $h$ that descends to
        the vector level if and only if $\widetilde{Y}_{k+h}X=X\widetilde{Y}_k$.
    \end{enumerate}
\end{lemma}
\begin{proof}
    If $X$ is a matrix of symmetric difference operators, then both
    \[
    V_{k+h}\mathcal{B}_{\st}X\mathcal{B}_{\st}^{-1}V_k^{-1}\quad\text{and}\quad
    \mathcal{B}_{\ko,k+h}X\mathcal{B}_{\ko,k}
    \]
    are endomorphisms of $\mathcal{A}$ that are defined in terms of rational functions in $z$, powers of $T$, and reflections. In other words: they are difference-reflection operators that map $\mathcal{A}$ to itself. Consequently, the two claims are just reformulations of the definition of non-symmetric shift operators in light of the statement of Proposition~\ref{prop-nsso-mso}.
\end{proof}

\begin{proposition}\label{prop-nonsymmetric-adjoints}
    All non-symmetric shift operators have formal adjoints.
\end{proposition}
\begin{proof}
    Let $S\in\mathcal{S}_{\mathrm{ns}}(h)$ and write as in
    Proposition~\ref{prop-nsso-mso}
    $\widehat{S}\in\widehat{S}_{\st}(h),\widetilde{S}\in\widehat{S}_{\ko}(h)$ for the matrix shift operators given by
    \begin{align*}
        \widehat{S} &= V_{k+h}^{-1}\mathcal{B}_{\st}^{-1}S\mathcal{B}_{\st}V_k\\
        \widetilde{S} &= \mathcal{B}_{\ko,k+h}^{-1} S\mathcal{B}_{\ko,k}.
    \end{align*}
    As in Proposition~\ref{prop-structure-mso}, let $S_1,\dots,S_4,S_1',\dots,S_4'$ be
    shift operators such that
    \[
        \widehat{S} = \mqty(S_{1,k+\epsilon_1} & S_{2,k+\epsilon_2}\\
        S_{3,k+\epsilon_1} & S_{4,k+\epsilon_2}),\qquad
        \widetilde{S} = \mqty(S'_{1,k} & S'_{2,k+v_1+v_2}\\
        S'_{3,k} & S'_{4,k+v_1+v_2}).
    \]
    We now show that $V_{k+h}\widehat{S}V_k^{-1}$ and $\widetilde{S}$ have formal adjoints
    with respect to $H_{\st},H_{\ko}$ or equivalently with respect to $\widehat{H}_{\st},\widehat{H}_{\ko}$, which will show the claim.

    We claim that the formal adjoint of $V_{k+h}\widehat{S}V_k^{-1}$ with respect to $\widehat{H}_{\st}$ is given by
    \[
        V_{k-h}\widehat{S}_{k-h}V_k^{-1} := V_{k-h} \frac{a-b}{aq^{-h_1}-bq^{-h_2}}
        \mqty(q^{-h_2} (S_1^*)_{k+\epsilon_1} & -\frac{b}{a}q^{-h_2}
        (S_3^*)_{k+\epsilon_2}\\
        -\frac{a}{b}q^{-h_1}(S_2^*)_{k+\epsilon_1} & q^{-h_1}(S_4^*)_{k+\epsilon_2})V_k^{-1}.
    \]
    Let $F,G\in K^{2\times 2}[z+z^{-1}]$ be two multiples of the identity
    matrix, say $F(z) = \mqty(1 & 0\\0 & 1)f(z)$ and similarly for $G,g$. We now show that
    \begin{align*}
        \widehat{H}_{\st,k+h}\qty(V_{k+h}\widehat{S}_kF,
        V_{k+h}G)
        &=\ct\qty(\qty(\widehat{S}_k F)^T
        V_{k+h}^T\mathcal{W}_{k+h} V_{k+h}^* G^*)\\
        &= \ct\qty(F^T V_k^T\mathcal{W}_{\st,k} V_k^* \qty(\widehat{S}'_k G)^*)\\
        &= \widehat{H}_{\st}\qty(V_kF, V_k\widehat{S}'_k
        G).
    \end{align*}
    For this we apply Proposition~\ref{prop-diagonal-matrix-weights} to see
    \begin{align*}
        &\ct\qty(\qty(\widehat{S}\mqty(f & 0\\0 & f))^T
        V_{k+h}^T\mathcal{W}_{k+h} V_{k+h}^* \mqty(g^* & 0\\0 & g^*)) = \frac{aq^{h_1} - bq^{h_2}}{2}\cdot\\
        &\ct\qty(\mqty(S_{1,k+\epsilon_1}f & S_{3,k+\epsilon_1}f\\S_{2,k+\epsilon_2}f & S_{4,k+\epsilon_2}f)
        \mqty(-a^{-1} q^{-h_1}\nabla_{k+h+\epsilon_1} & 0\\0 & b^{-1}q^{-h_2}
        \nabla_{k+h+\epsilon_2})g^*).
    \end{align*}
    Using the $K$-linearity of $\ct$ and the definition
    of $\langle\cdot,\cdot\rangle'$, we can continue
    the equation chain as
    \[
        = \frac{aq^{h_1} - bq^{h_2}}{2}
        \mqty(-a^{-1}q^{-h_1} \langle S_{1,k+\epsilon_1}f,g\rangle'_{k+h+\epsilon_1} &
        b^{-1}q^{-h_2}\langle S_{3,k+\epsilon_1}f,g\rangle'_{k+h+\epsilon_2}\\
        -a^{-1}q^{-h_1} \langle S_{2,k+\epsilon_2}f,g\rangle'_{k+h+\epsilon_1}
        & b^{-1}q^{-h_2}\langle S_{4,k+\epsilon_2}f,g\rangle'_{k+h+\epsilon_2}).
    \]
    We now apply Definition~\ref{def:adjoints_sym_shiftops} and Corollary~\ref{cor-adjoints-exist} to obtain
    \[
    = \frac{aq^{h_1}-bq^{h_2}}{2}
        \mqty(-a^{-1}q^{-h_1} \langle f,(S_1^*)_{k+h+\epsilon_1}g\rangle'_{k+\epsilon_1} 
        & b^{-1}q^{-h_2}\langle f, (S_3^*)_{k+h+\epsilon_2}g\rangle'_{k+\epsilon_1}\\
        -a^{-1}q^{-h_1} \langle f, (S_2^*)_{k+h+\epsilon_1}g\rangle'_{k+\epsilon_2}
        & b^{-1}q^{-h_2}\langle f, (S_4^*)_{k+h+\epsilon_2}g\rangle'_{k+\epsilon_2}).
    \]
    Next, we rewrite this expression as
    \begin{align*}
        =& \frac{a-b}{2}
        \qty(q^{-h_1-h_2}\frac{aq^{h_1}-bq^{h_2}}{a-b})^*\cdot\\
        &\mqty(-a^{-1}\langle f, q^{h_1} (S_1^*)_{k+h+\epsilon_1}g\rangle'_{k+\epsilon_1}
        & -a^{-1}\left\langle f, -\frac{bq^{h_2}}{a}(S_3^*)_{k+h+\epsilon_2}g\right\rangle'_{k+\epsilon_1}\\
        b^{-1}\left\langle f, - \frac{aq^{h_1}}{b} (S_2^*)_{k+h+\epsilon_1}g\right\rangle'_{k+\epsilon_2}
        & b^{-1}\langle f, q^{h_2} (S_4^*)_{k+h+\epsilon_2}g\rangle'_{k+\epsilon_2}).
    \end{align*}
    Similarly to previous steps, this now equals
    \[
        =\ct\qty(\mqty(f & 0\\0 & f)V_k^T\mathcal{W}_kV_k^*
        \qty(\widehat{S}'\mqty(g & 0\\0 & g))^*)
    \]
    as claimed.
    This shows that
    \[
        \widehat{H}_{\st,k+h}(V_{k+h}\widehat{S}_k F, V_{k+h}G)
        = \widehat{H}_{\st,k}(V_k F, V_k\widehat{S}'_k G)
    \]
    for all symmetric matrix polynomials $F,G$ that are scalar multiples of the identity matrix. Due to the
    sesquilinearity of $\widehat{H}_{\st}$, this is actually true for
    all symmetric matrix polynomials $F,G$. In particular, substituting
    $F\to V_k^{-1}F$ and $G\to V_{k+h}^{-1}G$ we have
    \[
        \widehat{H}_{\st,k+h}(V_{k+h}\widehat{S}_k V_k^{-1} F, G)
        = \widehat{H}_{\st,k}(F, V_k\widehat{S}'_k V_{k+h}^{-1} G),
    \]
    which is claimed.

    A similar proof shows that the formal adjoint of $\widetilde{S}$ in the Koornwinder setting is given by
    \[
        \mqty(\frac{1-a^{-1}b^{-1}}{1-a^{-1}b^{-1}q^{h_1+h_2}} (S_1^{\prime *})_k & \frac{1-ab}{1-a^{-1}b^{-1}q^{h_1+h_2}}
        (S_3^{\prime *})_{k+v_1+v_2}\\
        \frac{1-a^{-1}b^{-1}}{1-abq^{-h_1-h_2}}(S_2^{\prime *})_k
        & \frac{1-ab}{1-abq^{-h_1-h_2}}(S_4^{\prime *})_{k+v_1+v_2}).
        \qedhere
    \]
\end{proof}

\begin{proposition}\label{prop-generators-nonsymmetric}
    Let $S\in\mathcal{S}(h)$.
    \begin{enumerate}
        \item Let $h\cdot v_3=h\cdot v_4=0$ and assume that
        $\widetilde{\eta}_h(S)$ only depends on $ab$ and $cd$, which is
        for example the case for the basis elements of Theorem~\ref{thm-nonzero-shifts}. Then
        \[
            \widehat{S} = \mqty(S_{k+\epsilon_1} & 0\\0 & S_{k+\epsilon_2})
        \]
        is an S-matrix shift operator with shift $h$ that descends to the vector level and thus gives rise to a
        non-symmetric shift operator of shift $h$.
        \item Let $h\cdot v_1=h\cdot v_2=0$ and assume that
        $\widetilde{\eta}_h(S)$ is a basis element from Theorem~\ref{thm-nonzero-shifts}. Then
        \[
            \widetilde{S} = \mqty(S_k & 0 \\0 & q^{-\frac{\abs{h\cdot v_3} + \abs{h\cdot v_4}}{2}}S_{k+\epsilon_1+\epsilon_2})
        \]
        is a K-matrix shift operator with shift $h$ that descends to the
        vector level and thus gives rise to a non-symmetric shift operator
        of shift $h$.
    \end{enumerate}
    These (and appropriate linear combinations of the second case) are the only diagonal matrix shift operators that descend
    to the vector level.
\end{proposition}
\begin{proof}\leavevmode
    \begin{enumerate}
        \item The $z^{-1}$-leading terms of $S_{k+\epsilon_1},S_{k+\epsilon_2}$ equal each other because $\widetilde{\eta}_h(S)$ only
        depends on the product $ab$ and not on $a,b$ individually. Consequently, $\widetilde{\eta}_{\st,h}(\widehat{S})$ is a scalar
        multiple of the identity matrix, say $f(T)\mqty(1 & 0\\0 & 1)$.

        Furthermore, due to the conditions on $h$, we have
        \[
            V_k^{-1}C_{\st,k}(T) = V_{k+h}^{-1} C_{\st, k+h}\qty(q^{-\frac{h\cdot v_1}{2}}T)
            \mqty(q^{-h_1} & 0\\0 & 1),
        \]
        which shows that
        \[
            \qty(V_{k+h}^{-1} C_{\st, k+h}\qty(q^{-\frac{h\cdot v_1}{2}}T))^{-1}\widetilde{\eta}_{\st,h}(\widehat{S})
            V_k^{-1}C_{\st,k}(T)
            = \mqty(f(T)q^{-h_1} & 0\\0 & f(T)),
        \]
        which is indeed diagonal.
    \item Note that $\abs{h\cdot v_3}+\abs{h\cdot v_4}$ is the number of
    factors in $\widetilde{\eta}_h(S)$, so that the $q$-power and
    the conditions on $h$ exactly
    ensure that $\widetilde{\eta}_{\ko,h}(\widetilde{S})$ is a scalar multiple of
    the identity matrix, say $f(T)\mqty(1 & 0\\0 & 1)$.

    Note that $C_{\ko,k}(T)$ only depends on $ab, cd$, so that any shift
    by $h$ leaves these coefficients constant. Moreover, $h\cdot v_1=0$,
    so that
    \[
        C_{\ko,k}(T) = C_{\ko, k+h}\qty(q^{-\frac{h\cdot v_1}{2}}T),
    \]
    whence we obtain that
    \[
        C_{\ko,k+h}\qty(q^{-\frac{h\cdot v_1}{2}}T)^{-1}
        \widetilde{\eta}_{\ko,h}(\widetilde{S})
        C_{\ko,k}(T) = f(T)\mqty(1 & 0\\0 & 1),
    \]
    which is a diagonal matrix.
    \end{enumerate}
    To see that these are the only diagonal matrix shift operators, assume
    that
    \[
        S=\mqty(S_{1,k+\epsilon_1} & 0\\0 & S_{2,k+\epsilon_2})
        \qquad\text{or}\qquad
        \mqty(S_{1,k} & 0\\0 & S_{2,k+v_1+v_2})
    \]
    is an S- or K-matrix shift operator that descends to the vector level, with $S_1,S_2\in\mathcal{S}(h)$, and that
    \[
        \widetilde{\eta}_{s,h}(S) = \mqty(f_1(T) & 0\\0 & f_2(T)).
    \]
    Then by Proposition~\ref{prop-descend-to-vector-necessary} we have $f_1(T)=f_2(T)$.
\end{proof}

\begin{example}
    Applying Proposition~\ref{prop-generators-nonsymmetric}(i) to
    $G^q_+,G^q_-,E^q_{12},E^q_{34}$ defines non-symmetric shift operators
    $\mathcal{G}^q_+,\mathcal{G}^q_-,\mathcal{E}^q_{1,+},\mathcal{E}^q_{1,-}$ with shifts $v_1,-v_1,-v_2,v_2$, respectively. Similarly, using Proposition~\ref{prop-generators-nonsymmetric}(ii) to $E^q_{13},E^q_{24}$, we obtain $\mathcal{E}^q_{2,+},\mathcal{E}^q_{2,-}$ with shifts $-v_3,v_3$, respectively.

    The operators $\mathcal{G}^q_+,\mathcal{G}^q_-$ are called the \emph{forward shift operator} and \emph{backward shift operator},
    respectively. The operators $\mathcal{E}^q_{i,+},\mathcal{E}^q_{i,-}$
    ($i=1,2$) are called the $i$-th \emph{contiguous shift operators}.
\end{example}

\begin{remark}
    Note that our focus on diagonal matrix shift operators is purely 
    for simplicity (and because we do not \emph{need} to consider more
    to compute the norms). It is also possible to construct non-diagonal S-matrix shift operators
    that descend to the vector level:
    \[
        \mqty((a-b)E_{2i,k+\epsilon_1} & 0\\b(q-1)E_{1i,k+\epsilon_1} &
        (aq-b)E_{2i,k+\epsilon_2}),\quad
        \mqty((b-aq^{-1})E_{1i,k+\epsilon_1} & a(1-q^{-1})E_{2i,k+\epsilon_2}\\0 & (a-b)E_{1i,k+\epsilon_2})
    \]
    for $i=3,4$ and similarly for K-matrix shift operators.

    Note furthermore that Proposition~\ref{prop-generators-nonsymmetric}
    applies to all generators $S_{\pm v_i}$ ($i=1,\dots,4$) (half in the
    Steinberg basis, and half in the Koornwinder basis), so that we obtain
    non-zero non-symmetric shift operators for all monoid generators
    $\pm v_i$ ($i=1,\dots,4$) of the lattice of possible shifts. Since
    the algebra of non-symmetric shift operators has no zero divisors, we
    thus obtain the existence of a shift operator for every shift that
    is a $\Z$-linear combination of $v_1,\dots,v_4$, which are all
    that are not explicitly forbidden by Corollary~\ref{cor-allowed-nsso-shifts}.
\end{remark}

\begin{proposition}\label{prop-named-nsso-adjoint}
    The formal adjoints of our named non-symmetric shift operators are as follows:
    \begin{align*}
        \qty(\mathcal{G}^q_+)^* &= -\frac{q^2}{abcd}\mathcal{G}^q_-\\
        \qty(\mathcal{G}^q_-)^* &= -\mathcal{G}^q_+\\
        \qty(\mathcal{E}^q_{1,+})^* &= -\frac{q}{cd}\mathcal{E}^q_{1,-}\\
        \qty(\mathcal{E}^q_{1,-})^* &= -\frac{1}{ab}\mathcal{E}^q_{1,+}\\
        \qty(\mathcal{E}^q_{2,+})^* &= -\frac{q}{bd}\mathcal{E}^q_{2,-}\\
        \qty(\mathcal{E}^q_{2,-})^* &= -\frac{q}{ac}\mathcal{E}^q_{2,+}.
    \end{align*}
    We note that the absence of the $q$ for the formal adjoint of $\mathcal{E}^q_{1,-}$ is not a typo.
\end{proposition}
\begin{proof}
    We begin with the first four equations, which can be shown in the
    Steinberg basis. We adopt the notation from Proposition~\ref{prop-generators-nonsymmetric}(i) and assume that a non-symmetric shift operator $\mathcal{S}$ was obtained from a symmetric shift operator $S$. By Proposition~\ref{prop-nonsymmetric-adjoints} the formal adjoint of $\mathcal{S}$ is obtained from the S-matrix shift operator
    \[
        \mqty((S^*)_{k+\epsilon_1} & 0\\
        0 & (S^*)_{k+\epsilon_2}).
    \]
    From Proposition~\ref{lem-fundamental-adjoints} we thus obtain that the
    operators generated by $(G^q_+)^*$, $(G^q_-)^*$,$(E^q_{12})^*$,$(E^q_{34})^*$ are
    \[
        -\frac{q^2}{abcd}\mathcal{G}^q_-,-\mathcal{G}^q_+,
        -\frac{q}{cd}\mathcal{E}^q_{1,-},-\frac{1}{ab}\mathcal{E}^q_{1,+},
    \]
    respectively. The differences in the first and last operator
    are due to the coefficients being evaluated for the labelling
    $k+\epsilon_1,k+\epsilon_2$.
    
    Analogously, in the Koornwinder basis, if $\mathcal{S}$ is obtained
    from a symmetric shift operator $S$ using Proposition~\ref{prop-generators-nonsymmetric}(ii), its formal adjoint is obtained from the
    K-matrix shift operator
    \[
        \mqty(S^*_{k} & 0\\0 & \qty(q^{-\frac{1}{2}}S)^*_{k+v_1+v_2})
        = \mqty(S^*_k & 0\\0 & q^{\frac{1}{2}} S^*_{k+v_1+v_2}).
    \]
    For the case of $S=E^q_{ij}$ for $\set{i,j,k,l}=\set{1,2,3,4}$ and
    $\#\set{1,2}\cap\set{i,j}=1$ (i.e. the four fundamental shift operators
    that give rise to a diagonal K-matrix shift operator that descends to
    the vector level), both diagonal elements have the same
    $z^{-1}$-leading term $-\frac{q}{u_ku_l} \widetilde{\eta}_h(E^q_{kl})$, so they give rise to another diagonal K-shift operator
    that descends to the vector level. In particular this shows that
    \[
        \qty(\mathcal{E}^q_{2,+})^* = -\frac{q}{bd}\mathcal{E}^q_{2,-}\qquad
        \qty(\mathcal{E}^q_{2,-})^* = -\frac{q}{ac}\mathcal{E}^q_{2,+}.\qedhere
    \]
\end{proof}

\begin{proposition}\label{prop:shift-action-nonsymmetric}
    The named non-symmetric shift operators act as follows:
    \begin{align*}
        \mathcal{G}^q_+ E_{-n,k} &=
        q^{-\frac{1}{2}}\qty(q^{\frac{n}{2}}-q^{-\frac{n}{2}})
        E_{1-n,k+v_1}\\
        \mathcal{G}^q_+ E_{n+1,k} &=
        \qty(q^{\frac{n}{2}}-q^{-\frac{n}{2}})
        E_{n,k+v_1}\\
        \mathcal{G}^q_- E_{-n,k} &=
        \qty(\frac{abcd}{q}q^{\frac{n}{2}}-q^{-\frac{n}{2}})
        E_{-n-1,k-v_1}\\
        \mathcal{G}^q_- E_{n+1,k} &=
        q^{-\frac{1}{2}}
        \qty(\frac{abcd}{q}q^{\frac{n}{2}}-q^{-\frac{n}{2}})
        E_{n+2,k-v_1}\\
        \mathcal{E}^q_{1,+} E_{-n,k} &=
        -q^{-\frac{1}{2}}\qty(abq^{\frac{n}{2}}-q^{-\frac{n}{2}})
        E_{-n,k-v_2}\\
        \mathcal{E}^q_{1,+}E_{n+1,k} &=
        -\qty(abq^{\frac{n}{2}}-q^{-\frac{n}{2}})
        E_{n+1,k-v_2}\\
        \mathcal{E}^q_{1,-}E_{-n,k} &=
        -q^{\frac{1}{2}}\qty(\frac{cd}{q}q^{\frac{n}{2}}
        - q^{-\frac{n}{2}}) E_{-n,k+v_2}\\
        \mathcal{E}^q_{1,-}E_{n+1,k} &=
        -\qty(\frac{cd}{q}q^{\frac{n}{2}}
        - q^{-\frac{n}{2}})E_{n+1,k+v_2}\\
        \mathcal{E}^q_{2,+}E_{-n,k} &=
        -\qty(\frac{ac}{q}q^{\frac{n}{2}}-q^{-\frac{n}{2}}) E_{-n,k-v_3}\\
        \mathcal{E}^q_{2,+}E_{n,k} &=-\qty(\frac{ac}{q}q^{\frac{n}{2}}-q^{-\frac{n}{2}}) E_{n,k-v_3}\\
        \mathcal{E}^q_{2,-}E_{-n,k} &=-\qty(\frac{bd}{q}q^{\frac{n}{2}}-q^{-\frac{n}{2}}) E_{-n,k+v_3}\\
        \mathcal{E}^q_{2,-}E_{n,k} &=-\qty(\frac{bd}{q}q^{\frac{n}{2}}-q^{-\frac{n}{2}}) E_{n,k+v_3}
    \end{align*}
    for all $n\in\N_0$.
\end{proposition}
\begin{proof}
    Let $\mathcal{S}$ be the non-symmetric shift operator generated by a diagonal matrix shift operator $\widehat{S}$ with diagonal element $S$ according to Proposition~\ref{prop-generators-nonsymmetric}. From the proof of that
    proposition we know that
    \begin{align*}
        \widehat{S}_k\mathbb{E}_{\st,m,k} &= \mqty(q^{-h_1} & 0\\0 & 1)
        \widetilde{\eta}_h(S)\qty(aq,b,c,d;q^{\frac{m}{2}})\mathbb{E}_{\st,m - v_1 \cdot h,k + h}\\
        \widehat{S}_k\mathbb{E}_{\ko,m,k} &= \mqty(1 & 0\\0 & 1)
        \widetilde{\eta}_h(S)\qty(a,b,c,d;q^{\frac{m}{2}})\mathbb{E}_{\ko,m - v_1 \cdot h,k + h},
    \end{align*}
    depending on whether we are using the Steinberg or the Koornwinder
    basis. Taking $\widetilde{\eta}_h(S)$ from Proposition~\ref{prop-eta-fundamental-symmetric-so}, we obtain the claimed expressions.
\end{proof}

\begin{proposition}
    Let $S\in\mathcal{S}_{\mathrm{ns}}(h)$. The following are equivalent:
    \begin{enumerate}
        \item $S\mathcal{A}_0\subset\mathcal{A}_0$, and $S$
        restricts to the symmetric shift operator $X\in\mathcal{S}(h)$.
        \item $S$ is generated as in Proposition~\ref{prop-nsso-mso}
        by the $\st$-matrix shift operator
        \[
            \widehat{X} = \mqty(S_{1,k+\epsilon_1} & S_{2,k+\epsilon_2}\\
            S_{3,k+\epsilon_1} & S_{4,k+\epsilon_2})
        \]
        satisfying
        \[
            0=\mqty(1 & 1)\widehat{X}\mqty(1\\-1)
        \]
        (not just the $z^{-1}$-leading term as in Proposition~\ref{prop-descend-to-vector-necessary}), and
        \[
            X = \frac{q^{-h_1}}{a-b}\qty(a(S_{4,k+\epsilon_2} - S_{3,k+\epsilon_1}) - b(S_{1,k+\epsilon_1} - S_{2,k+\epsilon_2})).
        \]
        \item $S$ is generated as in Proposition~\ref{prop-nsso-mso}
        by the $\ko$-matrix shift operator
        \[
            \widetilde{X} = \mqty(S_{1,k} & S_{2,k+v_1+v_2}\\0 & S_{4,k+v_1+v_2})
        \]
        and $X=S_{1,k}$.
    \end{enumerate}
\end{proposition}
\begin{proof}
    For $s\in\set{\st,\ko}$, we have
    \[
        \mathcal{B}_{s,k}^{-1}\mathcal{A}_0 = \mqty(1\\0)\mathcal{A}_0.
    \]
    Consequently, any reflection-difference operator $S$ leaves
    $\mathcal{A}_0$ invariant and acts as the symmetric difference operator $X$ if
    \begin{equation}\label{eq-leaves-invariant-polys-invariant}
        \mathcal{B}_{s,k+h}^{-1} S \mathcal{B}_{s,k}
        = \mqty(X & \ast\\0 & \ast)
    \end{equation}
    for any (equivalently: all) $h\in\Q^4$. In the case $s=\ko$, we
    have
    \[
        \mathcal{B}_{s,k+h}^{-1}S\mathcal{B}_{s,k} = \widetilde{X},
    \]
    which shows the equivalence \enquote{(i)$\Leftrightarrow$(iii)}.

    To show \enquote{(i)$\Leftrightarrow$(ii)}, we note that
    \[
        \widehat{X} = V_{k+h}^{-1}\mathcal{B}_{\st}^{-1}S\mathcal{B}_{\st}V_k
    \]
    and that the conditions of \eqref{eq-leaves-invariant-polys-invariant} are equivalent to (ii) after twisted conjugation with
    $V_{k+h}^{-1}$.
\end{proof}

\begin{corollary}\label{cor:res_koornwinder_basis}
    Let $S\in\mathcal{S}_{\mathrm{ns}}(h)$ be obtained from $X\in\mathcal{S}(h)$ as in Proposition~\ref{prop-generators-nonsymmetric} in the $s$-basis, with $s\in\{\st, \ko\}$.
    \begin{enumerate}
        \item If $s=\ko$, then $S$ restricts to $S$ on $\mathcal{A}_0$. In particular, $\mathcal{E}^q_{2,+}$ and $\mathcal{E}^q_{2,-}$ restrict to $E^q_{13}$ and $E^q_{24}$ on $\mathcal{A}_0$, respectively.
        \item If $s=\st$, then this is not necessarily true, however, a sufficient condition is that $S_{k+\epsilon_1} = S_{k+\epsilon_2} = S_k$. In that case, $S$ restricts to $q^{-h_1}S$ on $\mathcal{A}_0$. In particular, $\mathcal{G}^q_+$ and $\mathcal{E}^q_{1,-}$ restrict to $q^{-\frac{1}{2}}G^q_+$ and $q^{-\frac{1}{2}}E^q_{34}$ on $\mathcal{A}_0$, respectively.
    \end{enumerate}
Notably, $\mathcal{G}^q_-$ and $\mathcal{E}^q_{1,+}$ do not restrict to operators on $\mathcal{A}_0$.
\end{corollary}

\subsection{Rodrigues Relation}
From Proposition~\ref{prop:shift-action-nonsymmetric} and induction, we know
that we can obtain the non-symmetric AW-polynomials as
\begin{align}
    E_{-n,k} &= \frac{(-1)^nq^{\frac{n(n-1)}{4}}}{\qty(abcdq^{n};q)_n} \mathcal{G}^q_{-,k+v_1}\cdots\mathcal{G}^q_{-,k+nv_1}1
    \label{eq:Eminusn-in-terms-of-backward-ops}\\
    E_{n+1,k} &= \frac{(-1)^n q^{\frac{n(n+1)}{4}}}{\qty(abcdq^n;q)_n}
    \mathcal{G}^q_{-,k+v_1}\cdots\mathcal{G}^q_{-,k+nv_1}E_{1,k+nv_1},\label{eq:Enplus1-in-terms-of-backward-ops}
\end{align}
as is also pointed out in the symmetric case in \cite[Equation~2.16]{KaMi89}. However, it turns out that similarly to \cite[\S6.5.14]{Mac03},
these relations can be rewritten so as to remove the dependence on
$k+iv_1$ ($i=1,\dots,n$) from the backward shift operators.

In the following we shall write $\widehat{G}^q_\pm$ for
the S-matrix shift operator
\[
    \mqty(G^q_{\pm,k+\epsilon_1}  & 0\\0 & G^q_{\pm,k+\epsilon_2}).
\]

\begin{lemma}\label{lem:formal-expression-adjoint-forward}
    Defining $\Phi_k:= \frac{V_k^{*T}\mathcal{W}_{\st,k}^*}{z-z^{-1}}$, then
    \[
        (V\circ \widehat{G}^q_+\circ V^{-1})^*_{k+v_1} = \Phi^{-1}_k\widehat{G}^q_{+,k}\Phi_{k+v_1}
    \]
    as formal expressions.
\end{lemma}
\begin{proof}
    If $E,F\in\mathcal{A}_0\otimes K^2$, then
    \begin{align*}
        H_{\st,k+v_1}(V_{k+v_1}\widehat{G}^q_{+,k}V_k^{-1} E, F) &=
        \ct\qty(\qty(V_{k+v_1}\widehat{G}^q_{+,k}V_k^{-1}E)^T\mathcal{W}_{\st,k+v_1}
        F^*)\\
        &= \ct\qty((T-T^{-1})\qty(V_k^{-1}E)^T V_{k+v_1}^T \frac{\mathcal{W}_{\st,k+v_1}}{z-z^{-1}} F^*)\\
        &= -\ct\qty((T-T^{-1})\qty(V_k^{-1}E)^T \Phi_{k+v_1}^*F^*)\\
        &= \ct\qty(E^TV_k^{-T} (T-T^{-1})(\Phi_{k+v_1}^*F^*))\\
        &= \ct\qty(E^T\mathcal{W}_{\st,k} \frac{z-z^{-1}}{\mathcal{W}_kV_k^T}
        \widehat{G}^q_{+,k} (\Phi_{k+v_1}^*F^*))\\
        &= \ct\qty(E^T\mathcal{W}_{\st,k} \mqty(\Phi_k^{-1}\widehat{G}^q_{+,k}\Phi_{k+v_1} F)^*)\\
        &= H_{\st,k}(E, \Phi_k^{-1}\widehat{G}^q_{+,k} \Phi_{k+v_1}F),
    \end{align*}
    where we apply $T-T^{-1}$ to the expression immediately to its right.
    This implies the claim. Here we used that $T,T^{-1}$ commute with $*$, so that $\widehat{G}^q_+$ anti-commutes with $*$.
\end{proof}

\begin{corollary}\label{cor:formal-expression-backward}
    We then have $V_k\widehat{G}^q_{-,k+v_1} V_{k+v_1}^{-1} = -abcd \Phi_k^{-1}\widehat{G}^q_+\Phi_{k+v_1}$.
\end{corollary}
\begin{proof}
    This follows from Lemma~\ref{lem:formal-expression-adjoint-forward}
    and Proposition~\ref{prop-nonsymmetric-adjoints}.
\end{proof}

\begin{lemma}
    On a vector level we then obtain
    \begin{align*}
        \mathcal{B}^{-1}_{\st}(E_{-n,k}) &= q^{\frac{5n(n-1)}{4}}
        \frac{(abcd)^n}{\qty(abcdq^n;q)_n}
        \Phi_k^{-1} \qty(\widehat{G}^q_+)^n\Phi_{k+nv_1}\mqty(1\\0)\\
        \mathcal{B}_{\st}^{-1}(E_{n+1,k}) &=q^{\frac{n(5n-3)}{4}}
    \frac{(abcd)^n}{\qty(abcdq^n;q)_n}
    \Phi_k^{-1} \qty(\widehat{G}^q_+)^n\Phi_{k+nv_1}^{-1}
    \mqty(q^{\frac{n}{2}}\frac{bcdq^n+acdq^n-c-d}{1-abcdq^{2n}}\\1)
    \end{align*}
    for $n\in\N_0$.
\end{lemma}
\begin{proof}
    In vector form, \eqref{eq:Eminusn-in-terms-of-backward-ops}
    reads
    \[
        \mathcal{B}_{\st}^{-1}(E_{-n}) =
        \frac{(-1)^n q^{\frac{n(n-1)}{4}}}{\qty(abcdq^n;q)_n}
        V_k\widehat{G}^q_{-,k+v_1}V_{k+v_1}^{-1}
        \cdots V_{k+(n-1)v_1}\widehat{G}^q_{-,k+nv_1}V_{k+nv_1}^{-1}\mqty(1\\0).
    \]
    Inserting Corollary~\ref{cor:formal-expression-backward}, we get
    \begin{align*}
        &= \frac{q^{\frac{n(n-1)}{4}}\qty(abcd)\cdots\qty(abcdq^{2n-2})}{\qty(abcdq^n;q)_n}
        \Phi_k^{-1}(\widehat{G}^q_+)^n\Phi_{k+nv_1}\mqty(1\\0)\\
        &= q^{\frac{5n(n-1)}{4}}
        \frac{(abcd)^n}{\qty(abcdq^n;q)_n}
        \Phi_k^{-1} \qty(\widehat{G}^q_+)^n\Phi_{k+nv_1}\mqty(1\\0).
    \end{align*}
    Analogously, we obtain
    \[
        \mathcal{B}_{\st}^{-1}(E_{n+1}) = q^{\frac{n(5n-3)}{4}}
    \frac{(abcd)^n}{\qty(abcdq^n;q)_n}
    \Phi_k^{-1} \qty(\widehat{G}^q_+)^n\Phi_{k+nv_1}^{-1}\mathcal{B}_{\st}^{-1}(E_{1,k+nv_1})
    \]
    from \eqref{eq:Enplus1-in-terms-of-backward-ops}.
    Since by definition we have $E_{1,k+nv_1}(z)=z + c_{1,k+nv_1}$, we get
    \[
        \mathcal{B}_{\st}^{-1}(E_{1,k+nv_1}) = \mqty(c_{1,k+nv_1}\\1)
        = q^{\frac{n}{2}}\mqty(-\frac{c+d-acdq^n-bcdq^n}{1-abcdq^{2n}}\\1)
    \]
    using Proposition~\ref{prop:NLO-coefficients}.
\end{proof}

\section{Norms}\label{sec:norms}
We now use the non-symmetric shift operators $\mathcal{G}^q_\pm,\mathcal{E}^q_{1,\pm},\mathcal{E}^q_{2,\pm}$ from the last section to establish relations between and to
ultimately compute the norms $(E_n,E_n)$ for $n\in\Z$.
For this we write
\[
h_{n,k} := h_n(a,b,c,d) := (E_n(a,b,c,d), E_n(a,b,c,d))_{a,b,c,d}
\]
and note that
$h_n$ is symmetric under
$a\leftrightarrow b$ and $c\leftrightarrow d$.

\subsection{Recursion Relations for $h_n$}
Using the non-symmetric shift operators and their formal adjoints, we establish recursion relations between the $h_n$ for different values of $n$ and $a,b,c,d$.

\begin{proposition}\label{prop-recursion-relations}
    For $n\in\N_0$ we have
    \begin{align*}
        h_{-n,k} &= \frac{1-abcdq^{n-1}}{1-q^{n+1}}
        h_{-(n+1), k-v_1}\\
        h_{n+1,k} &= \frac{1-abcdq^{n-1}}{1-q^{n+1}}
        h_{n+2, k-v_1}\\
        h_{-n,k} &= \frac{1-cdq^{n-1}}{1-abq^{n+1}}
        h_{-n,k+v_2}\\
        h_{n+1,k} &= \frac{1-cdq^{n-1}}{1-abq^{n+1}}
        h_{n+1,k+v_2}\\
        h_{\pm n,k} &= \frac{1-bdq^{n-1}}{1-acq^n}
        h_{\pm n,k+v_3}.
    \end{align*}
\end{proposition}
\begin{proof}
    All these relations follow in the same way from
    Propositions~\ref{prop:shift-action-nonsymmetric} and
    \ref{prop-named-nsso-adjoint} using the operators
    \[
        \mathcal{G}^q_+,\mathcal{G}^q_-,\mathcal{E}^q_{1,+}, \mathcal{E}^q_{1,-},\mathcal{E}^q_{2,+},\mathcal{E}^q_{2,-}.
    \]
    We shall only show the first relation explicitly. We have
    \begin{align*}
        q^{-\frac{1}{2}}\qty(q^{\frac{m+1}{2}}-q^{-\frac{m+1}{2}}) h_{-m,k} &= \qty(\mathcal{G}^q_{+,k-v_1}E_{-m-1,k-v_1},
        E_{-m,k})_k \\
        &= \qty(E_{-m-1,k-v_1}, -\frac{q^2}{abcd}\mathcal{G}^q_{-,k} E_{-m,k})_{k-v_1}\\
        &= q^{-1}\qty(\frac{abcd}{q}q^{\frac{m}{2}}
        - q^{-\frac{m}{2}}) h_{-m-1,k-v_1},
    \end{align*}
    which implies the claimed relation
    \[
        h_{-n,k} = \frac{1-abcdq^{n-1}}{1-q^{n+1}}h_{-(n+1), k-v_1}.\qedhere
    \]
\end{proof}

\subsection{Computing the Norms}
First we show a general statement that helps us make sense
of recursion relations of the type we derived in the
previous subsection.
\begin{lemma}\label{lem:power-series}
    Let $f\in \mathcal{R}((q^{1/2}))$ satisfy
    \begin{align*}
        f(a,b,c,d) &= f\qty(aq^{\frac{1}{2}},
        bq^{\frac{1}{2}}, cq^{-\frac{1}{2}},
        dq^{-\frac{1}{2}})\\
        &= f\qty(aq^{\frac{1}{2}},
        bq^{-\frac{1}{2}},cq^{\frac{1}{2}},
        dq^{-\frac{1}{2}})\\
        &= f\qty(b,a,c,d)\\
        &= f\qty(a,b,d,c).
    \end{align*}
    Then $f\in\Q[(abcd)^{\pm1}]((q^{1/2}))$.
\end{lemma}
\begin{proof}
    Write
    \[
        f\qty(a,b,c,d) = \sum_{n\in\Z}\sum_{i,j,k,m\in\Z}C_{i,j,k,m,n} a^i b^j c^k d^m q^{\frac{n}{2}}.
    \]
    The fact that $f$ is a formal Laurent series in $q^{1/2}$ implies that there
    is $N\in\N$ such that $C_{i,j,k,m,n}=0$ for $n<-N$ and
    $i,j,k,m\in\Z$ arbitrary. The last two conditions (symmetry conditions) on $f$ imply that
    $C_{i,j,k,m,n}=C_{j,i,k,m,n}=C_{i,j,m,k,n}$. Furthermore,
    \begin{align*}
        f\qty(aq^{\frac{1}{2}},
        bq^{\frac{1}{2}}, cq^{-\frac{1}{2}},
        dq^{-\frac{1}{2}})
        &= \sum_{n\in\Z}\sum_{i,j,k,m=0}^\infty
        C_{i,j,k,m,n} a^i b^j c^k d^m q^{\frac{n+i+j-k-m}{2}}\\
        &= \sum_{n\in\Z}\sum_{i,j,k,m=0}^\infty
        C_{i,j,k,m,n-i-j+k+m} a^i b^j c^k d^m q^{\frac{n}{2}},
    \end{align*}
    so that $C_{i,j,k,m,n}=C_{i,j,k,m,n-i-j+k+m}$, and hence
    by induction
    \[
        \quad C_{i,j,k,m,n}=C_{i,j,k,m,n+r(-i-j+k+m)}
    \]
    for all $r\in\Z$. Assume that $i+j\ne k+m$, so that $-i-j+k+m\ne0$. Pick $r$ such that $n+r(-i-j+k+m)<-N$, then we obtain
    \[
        C_{i,j,k,m,n}=C_{i,j,k,m,n+r(-i-j+k+m)} = 0.
    \]
    Consequently, $C_{i,j,k,m,n}\ne0$ implies that $i+j=k+m$.
    
    Similarly, we can use the second condition on $f$ to argue
    that $C_{i,j,k,m,n}\ne0$ can only occur when $i-j=-k+m$. This
    further implies that $i=m$ and $j=k$. Moreover, we have
    $C_{i,j,k,m,n}=C_{j,i,k,m,n}\ne0$, so we must also have
    $j=m$ and $i=k$. Consequently, $i=j=k=m$. Hence, we have
    \[
        f = \sum_{n\in\Z}\sum_{k\in\Z} C_{k,k,k,k,n} (abcd)^k q^{\frac{n}{2}},
    \]
    which only depends on the product $abcd$ and on $q^{\frac{1}{2}}$.
\end{proof}

\begin{lemma}\label{lem-h}
    There are power series $(h'_n)_{n\in\Z}\subset \Q[abcd]((q^{1/2}))$ such that
    \begin{align*}
        h_{-n}\qty(a,b,c,d) &= \frac{h'_{-n}\qty(abcd)}{\qty(abq^{n+1},acq^n,adq^n,bcq^n,bdq^n,cdq^n;q)_\infty}\\
        h_{n+1}\qty(a,b,c,d) &= \frac{h'_{n+1}\qty(abcd)}{\qty(abq^{n+1},acq^{n+1},adq^{n+1},bcq^{n+1},bdq^{n+1},cdq^n;q)_\infty}
    \end{align*}
    for all $n\in\N_0$.
\end{lemma}
\begin{proof}
    Let $n\in\N_0$. Define
    \[
        h'_{-n}\qty(a,b,c,d) := h_{-n}\qty(a,b,c,d) \qty(abq^{n+1},acq^n,adq^n,bcq^n,bdq^n,cdq^n;q)_\infty
    \]
    and
    \[
        h'_{n+1}\qty(a,b,c,d) := h_{n+1}\qty(a,b,c,d) \qty(abq^{n+1},acq^{n+1},adq^{n+1},bcq^{n+1},bdq^{n+1},cdq^n;q)_\infty.
    \]
    By Proposition~\ref{prop-recursion-relations}, we see that
    all $h'_n\qty(a,b,c,d)$ ($n\in\Z$) satisfy the conditions of Lemma~\ref{lem:power-series}, so that
    they just depend on $abcd$.
\end{proof}

\begin{lemma}\label{lem:l1-shift}
    We have
    \[
        (1-az)\qty(1-aqz^{-1}) \Delta_k = \Delta_{k+\epsilon_1}.
    \]
\end{lemma}

\begin{lemma}\label{lem:norm-E1}
    We have
    \begin{align*}
        (E_1,E_1)_1 &= \frac{(1-ac)(1-ad)(1-bc)(1-bd)}{(1-abcd)^2}\\
        \qty((1-az)\qty(1-aqz^{-1}),1)_1 &=
        \frac{(1-abq)(1-ac)(1-ad)}{1-abcd}.
    \end{align*}
\end{lemma}
\begin{proof}
    We have $E_1(z) = z + c_1$, which implies
    \begin{align*}
        0 &= (E_1,E_0)_1 = (z,1)_1 + c_1\\
        &= (E_0,E_1)_1 = \qty(z^{-1},1)_1 + c_1^*.
    \end{align*}
    Consequently,
    \begin{align*}
        (E_1,E_1)_1 &= \qty(z,z)_1
        + c_1(1,z)_1 + c_1^*(z,1)_1 + c_1^*c_1\\
        &= 1 + c_1^*c_1 - 2c_1^*c_1
        = 1-c_1^*c_1.
    \end{align*}
    By Proposition~\ref{prop:NLO-coefficients}, this then equals
    \[
        1 - \frac{(a+b-abd-abc)(d+c-acd-bcd)}{(1-abcd)^2}
        =\frac{(1-ac)(1-ad)(1-bc)(1-bd)}{(1-abcd)^2}.
    \]
    For the second equality, note that
    \[
        \qty((1-az)\qty(1-aqz^{-1}),1)_1 =
        1+a^2q + ac_1 + aqc_1^*.
    \]
    By Proposition~\ref{prop:NLO-coefficients}, this equals
    \[
        \frac{(1-abq)(1-ac)(1-ad)}{1-abcd}.
    \]
\end{proof}

\begin{lemma}\label{lem-h-primes}
    There exists a $C\in\Q((q^{1/2}))$ such that
    \begin{align*}
        h'_{-n}(abcd) &= C \frac{(q;q)_n(abcdq^{2n};q)_\infty}{(abcdq^n;q)_n}\\
        h'_{n+1}(abcd) &= C \frac{(q;q)_n(abcdq^{2n+1};q)_\infty}{(abcdq^n;q)_{n+1}}
    \end{align*}
    for all $n\in\N_0$.
\end{lemma}
\begin{proof}
    From Proposition~\ref{prop-recursion-relations} we conclude that
    \[
        h'_{-n}(abcd) = \frac{1-abcdq^{n-1}}{1-q^{n+1}}
        h'_{-n-1}(abcdq^{-2})
    \]
    and
    \[
        h'_{n+1}(abcd) = \frac{1-abcdq^{n-1}}{1-q^{n+1}}
        h_{n+2}(abcdq^{-2})
    \]
    for $n\in\N_0$, so that inductively
    \begin{align*}
        h'_{-n}(abcd) &= \frac{(q;q)_n}{\qty(abcdq^n;q)_n}
        h'_0\qty(abcdq^{2n})\\
        h'_{n+1}(abcd) &= \frac{(q;q)_n}{\qty(abcdq^n;q)_n}
        h'_1\qty(abcdq^{2n}).
    \end{align*}
    We now use the expression for $(E_1,E_1)_1$ from Lemma~\ref{lem:norm-E1} to establish a relation between
    $h'_0$ and $h'_1$. Note that we have
    \begin{align*}
        (E_1,E_1)_1 &= \frac{h_1(a,b,c,d)}{h_0(a,b,c,d)}
        = \frac{h_1'(abcd)}{h_0'(abcd)}\frac{(abq,ac,ad,bc,bd,cd;q)_\infty}{(abq,acq,adq,bcq,bdq,cd;q)_\infty}\\
        &= \frac{h'_1(abcd)}{h'_0(abcd)} (1-ac)(1-ad)(1-bc)(1-bd)\\
        &= \frac{(1-ac)(1-ad)(1-bc)(1-bd)}{(1-abcd)^2},
    \end{align*}
    so that we obtain
    \[
        h'_1(abcd) = \frac{h'_0(abcd)}{(1-abcd)^2}
    \]
    and
    \begin{align*}
        h'_{n+1}(abcd) &= \frac{(q;q)_n}{(abcdq^n;q)_n}
        \frac{h'_0(abcdq^{2n})}{(1-abcdq^{2n})^2}\\
        &= \frac{(q;q)_n}{(abcdq^n;q)_{n+1}(1-abcdq^{2n})} h'_0(abcdq^{2n}).
    \end{align*}
    Having now reduced everything to finding $h'_0(abcd)$, we
    employ Lemma~\ref{lem:l1-shift} and Lemma~\ref{lem:norm-E1} to find that
    \begin{align*}
        \frac{h_0(aq,b,c,d)}{h_0(a,b,c,d)}
        &= \qty((1-az)\qty(1-aqz^{-1}),1)_1\\
        &= \frac{(1-abq)(1-ac)(1-ad)}{1-abcd}\\
        &= (1-abq)(1-ac)(1-ad)\frac{h'_0(abcdq)}{h'_0(abcd)}.
    \end{align*}
    This implies that $(1-abcd)h'_0(abcdq)=h'_0(abcd)$, so that there exists a constant $C\in\Q((q^{1/2}))$ such that
    \[
        h'_0(abcd) = C (abcd;q)_\infty.
    \]
    Consequently,
    \begin{align*}
        h'_{-n}(abcd) &= C\frac{(q;q)_n(abcdq^{2n};q)_\infty}{(abcdq^n;q)_n}\\
        h'_{n+1}(abcd) &= C\frac{(q;q)_n(abcdq^{2n};q)_\infty}{(abcdq^n;q)_{n+1}(1-abcdq^{2n})}
        = C\frac{(q;q)_n(abcdq^{2n+1};q)_\infty}{(abcdq^n;q)_{n+1}}
    \end{align*}
    for all $n\in\N_0$.
\end{proof}

\begin{corollary}\label{cor:norms}
    We have
    \begin{align*}
        h_{-n}(a,b,c,d) &= \frac{2(abcdq^{2n};q)_\infty}{(q^{n+1},abq^{n+1},acq^n,adq^n,bcq^n,bdq^n,cdq^n;q)_\infty (abcdq^n;q)_n}\\
        h_{n+1}(a,b,c,d) &= \frac{2 (abcdq^{2n+1};q)_\infty}{(q^{n+1},abq^{n+1},acq^{n+1},adq^{n+1},bcq^{n+1},bdq^{n+1},cdq^n;q)_\infty (abcdq^n;q)_{n+1}}
    \end{align*}
    for all $n\in\N_0$.
\end{corollary}
\begin{proof}
    From \cite[Equation~2.19]{KaMi89} we know that
    \[
        1 = \langle 1,1\rangle_k = \frac{1}{2}h_{0,k}
    \]
    for $k=(0,0,0,0)$. We have
    \[
        h_{0,k} = h_0\qty(1,-1,q^{\frac{1}{2}},-q^{\frac{1}{2}})
        = C\frac{(q;q)_\infty}{\qty(-q,q^{\frac{1}{2}},-q^{\frac{1}{2}},-q^{\frac{1}{2}},q^{\frac{1}{2}},-q;q)_\infty}.
    \]
    Since $\qty(-q,q^{\frac{1}{2}},-q^{\frac{1}{2}};q)_\infty=1$, we obtain
    \[
        1 = \frac{C(q;q)_\infty}{2},
    \]
    whence we conclude $C = \frac{2}{(q;q)_\infty}$. The rest follows from
    Lemmas~\ref{lem-h} and \ref{lem-h-primes}.
\end{proof}

Dividing by $h_0(a,b,c,d)$, we obtain
\begin{corollary}
    For $n\in\N_0$ we have
    \begin{align*}
        (E_{-n},E_{-n})_1 &= \frac{(q,abq,ac,ad,bc,bd,cd;q)_n}{(abcd;q)_{2n}\qty(abcdq^n;q)_n}\\
        (E_{n+1},E_{n+1})_1 &=(1-abcd)^2
        \frac{(q,abq,acq,adq,bcq,bdq,cd;q)_n}{(abcd;q)_{2n+1}\qty(abcdq^n;q)_n}.
    \end{align*}
\end{corollary}

\section{Specialisation $q\to 1$}\label{sec:specialisation}
We now take $q\to1$ limits or \emph{specialisations} on a purely formal level,
replacing
\begin{align*}
    \lim \frac{a-1}{q^{\frac{1}{2}}-1} &:= 2k_1\\
    \lim \frac{-b-1}{q^{\frac{1}{2}}-1} &:= 2k_2\\
    \lim \frac{c-1}{q^{\frac{1}{2}}-1} &:= 2k_3 + 1\\
    \lim \frac{-d-1}{q^{\frac{1}{2}}-1} &:= 2k_4 + 1\\
    \lim \frac{T-1}{q^{\frac{1}{2}}-1} &:= \frac{z}{2}\partial_z,
\end{align*}
and in particular $\lim \frac{T-1}{q-1}=z\partial_z$, on a formal level, though all of these equations hold on an analytic level if we assume that
\[
    (a,b,c,d) = \qty(q^{k_1},-q^{k_2},q^{k_3 + \frac{1}{2}},-q^{k_4+\frac{1}{2}}).
\]
Note that under these replacements, the (non-)symmetric AW-polynomials
with parameters $a,b,c,d$ are mapped to the (non-)symmetric HO-polynomials
with parameters $k_1+k_3-\frac{1}{2},k_2+k_4-\frac{1}{2}$. This can
be seen from the fact that $\nabla$ specialises to the weight function
of the HO-polynomials, as is shown in \cite[\S\S5.1.12--14]{Mac03}.

Hence, by specialising the non-symmetric shift operators considered in this
work, we can also obtain (non-symmetric) shift operators for the non-symmetric HO-polynomials of type $BC_1$. 

\subsection{Symmetric Shift Operators}
\begin{lemma}\label{lem:symmetric-specialisations}
    Let $f\in\mathcal{A}_0$. Assume that $\lim f(z) = F\in\C[z,z^{-1}]$, then
    \begin{align*}
        \lim\frac{G^q_+}{q-1} f(z,q) &= \left(\frac{1}{z-z^{-1}}z\partial_z\right) F(z)\\
        \lim \frac{G^q_-}{q-1} f(z,q) &= \left((z-z^{-1})z\partial_z + \qty(2k\cdot v_1-1)(z+z^{-1})
        + 4k\cdot v_3\right)F(z)\\
        \lim E^q f(z,q)
        &= 2 F(z)\\
        \lim \frac{E^q_{13}}{q-1}f(z,q)
        &=-\left(\frac{1-z^{-1}}{1+z^{-1}}z\partial_z + k_1+k_3-\frac{1}{2}\right)F(z)\\
        \lim\frac{E^q_{24}}{q-1}f(z,q)
        &=-\left(\frac{1+z^{-1}}{1-z^{-1}}z\partial_z + k_2+k_4-\frac{1}{2}\right)F(z).
    \end{align*}
\end{lemma}
\begin{proof}
    For the forward shift operator, note that $\lim T = 1$ as $T = 1 + \frac{T-1}{q-1}(q-1)$, which is mapped to $1 + 0z\partial_z$ by $\lim$. Consequently, we have
    \[
        \frac{G^q_+}{q-1}
        = \frac{1}{z-z^{-1}}\qty(\frac{T-1}{q-1}
        + T^{-1}\frac{T-1}{q-1}),
    \]
    whose limit is indeed as claimed.
    
    In the case of the backward shift operator, we have
    \[
        \frac{G^q_-}{q-1} = \frac{A_1(z)}{z-z^{-1}}
        \qty(\frac{T-1}{q-1} + T^{-1}\frac{T-1}{q-1})
        +\frac{A_1(z)-A_1(z^{-1})}{q-1}
        \frac{1}{z-z^{-1}}T^{-1},
    \]
    and
    \[
        \lim a =\lim c = 1,\qquad
        \lim b = \lim d = -1,
    \]
    whence
    \[
        \lim A_1(z) = (z-z^{-1})^2.
    \]
    Furthermore,
    \begin{align*}
        A_1(z) - A_1\qty(z^{-1}) &=
        \qty(abcdq^{-\frac{5}{2}} - q^{-\frac{1}{2}})\qty(z^2-z^{-2})
        \\&\quad -\qty(q^{-2}\qty(abc+abd+acd+bcd)
        - q^{-1}\qty(a+b+c+d))(z-z^{-1})\\
        &= q^{-\frac{1}{2}}\qty(abcdq^{-2}-1)\qty(z^2-z^{-2})\\
        &\quad - \qty(abcdq^{-2}\qty(a^{-1}+b^{-1}+c^{-1}+d^{-1})
        - q^{-1}(a+b+c+d)),
    \end{align*}
    so that
    \[
        \lim \frac{A_1(z) - A_1\qty(z^{-1})}{q-1} = \qty(2k\cdot v_1 -1)\qty(z^2-z^{-2})
        +\qty(4k\cdot v_3)\qty(z-z^{-1}),
    \]
    where we used that
    \[
        \frac{b-b^{-1}}{q-1} = -\frac{-b-1}{q-1}
        + b^{-1}\frac{-b-1}{q-1},
    \]
    and its limit is $-2k_2$. This yields the claimed expression for $\lim \frac{G^q_-}{q-1}$.
    
    For $E^q = E^q_{12}$ note that we have
    \[
    \lim A_2(a,b;z) = z-z^{-1},
    \]
    so that
    \[
    \lim E^q_{12} = \frac{z-z^{-1}}{z-z^{-1}}
    -\frac{z^{-1}-z}{z-z^{-1}}
    = 2.
    \]
    
    Similarly, for $E^q_{13}$ we have that
    \[
        \lim A_2(a,c;z) = (1-z)(1-z^{-1}),
    \]
    so that
    \begin{align*}
        \lim \frac{E^q(a,c;z)}{q-1} &= \lim \frac{A_2(a,c;z)}{z-z^{-1}} \lim\frac{T-T^{-1}}{q-1} \\
        &+ \frac{1}{z-z^{-1}}\lim \frac{A_2(a,c;z)-A_2\qty(a,c;z^{-1})}{q-1}T^{-1}.
    \end{align*}
    Note that
    \[
        A_2(a,c;z)-A_2\qty(a,c;z^{-1}) = (-acq^{-1}+1)(z-z^{-1}),
    \]
    so that
    \[
        \lim \frac{A_2(a,c;z)-A_2\qty(a,c;z^{-1})}{q-1}
        = \qty(-k_1-k_3+\frac{1}{2})\qty(z-z^{-1}).
    \]
    This yields the claimed expression for
    $\lim \frac{E^q_{13}}{q-1}$.

    Lastly, for $E^q_{24}$ a similar proof works using
    \[
        \lim \frac{1-\frac{bd}{q}}{q-1} = -k_2-k_4+\frac{1}{2}
    \]
    and
    \[
        \lim A_2(b,d;z) = -(1+z)(1+z^{-1}).\qedhere
    \]
\end{proof}

\begin{corollary}
    If we express $F\in\C[z,z^{-1}]^{W_0}$ as a polynomial in $x:=\frac{z+z^{-1}}{2}$, say $F(z) = G\qty(\frac{z+z^{-1}}{2})$, the limits as $q\to1$ of $\frac{G^q_+}{q-1},\frac{G^q_+}{q-1},\frac{E^q_{13}}{q-1},\frac{E^q_{24}}{q-1}$ act as the operators
    \begin{align*}
        \frac{1}{2}&\partial_x\\
        2 (x^2-1)&\partial_x + 2(2k\cdot v_1-1)x + 4k\cdot v_3\\
        -((x-1)&\partial_x + k_1+k_3-\frac{1}{2})\\
        -((x+1)&\partial_x + k_2+k_4-\frac{1}{2}).
    \end{align*}
\end{corollary}

These calculations show that the operation of
dividing by $\frac{1}{q-1}$ and then specialising only works for
operators derived from $G^q_\pm,E^q_{13},E^q_{24}$, where it yields first-order differential operators. Hence, for the rest of this section,
we are only going to consider the forward and backward shift
operators as well as the \emph{second} type of contiguous shift operator.

\subsection{Matrix Shift Operators}
\begin{proposition}
    The matrix shift operators $\frac{V_{k+l}\widehat{G}^q_+V^{-1}}{q-1}$ and 
    $\frac{V_{k-l}\widehat{G}^q_-V^{-1}}{q-1}$ specialise to
    \[
        \frac{z}{z-z^{-1}}\partial_z,\qquad
        (z-z^{-1})z\partial_z + 2k\cdot v_1(z+z^{-1})
        + 4k\cdot v_3
        - 2\mqty(0 & 1\\1 & 0),
    \]
    respectively.
\end{proposition}
\begin{proof}
    We have $\lim V = \mqty(-1&1\\1&1)$ and
    $\lim V^{-1} = \frac{1}{2}\mqty(-1 & 1\\1 & 1)$. By Lemma~\ref{lem:symmetric-specialisations},
    the specialisation of $\frac{G^q_+}{q-1}$ does not depend on the
    parameters, so that $\frac{\widehat{G}^q_+}{q-1}$ is a diagonal matrix. Since $V^{-1}$ has constant coefficients (in $z$), so that the forward operator specialises to
    \[
        \lim \frac{V_{k+l}\widehat{G}^q_+V^{-1}}{q-1}
        = 
        \frac{1}{2}
        \mqty(-1 & 1\\1 & 1)^2 \frac{z}{z-z^{-1}}\partial_z
        = \frac{z}{z-z^{-1}}\partial_z.
    \]
    The backward operator specialises as
    \begin{align*}
        \lim \frac{V_{k-l}\widehat{G}^q_-V^{-1}}{q-1} &=(z^2-1)\partial_z + 2k\cdot v_1\qty(z+z^{-1})\\
        &+
        \mqty(-1 & 1\\1 & 1)
        \mqty(h+1 & 0\\0 & h-1)
        \mqty(-1 & 1\\1 & 1)\\
        &=(z^2-1)\partial_z + 2k\cdot v_1(z+z^{-1}) + \mqty(4k\cdot v_3 & -2\\-2 & 4k\cdot v_3).\qedhere
    \end{align*}
\end{proof}

\begin{proposition}
    The second type contiguous shift operators specialise as
    \[
        \lim \frac{\widehat{E}^q_{2,\pm}}{q-1} =
        -\qty(\frac{1\mp z^{-1}}{1\pm z^{-1}}z\partial_z
        + k_{1/2} + k_{3/4} -\frac{1}{2})
        -
        \mqty(0 & 0\\0 & 1),
    \]
    where $k_{1/2}=k_1$ or $k_2$ depending on the sign of $\pm$, similarly for $k_{3/4}$.
\end{proposition}
\begin{proof}
    Follows directly from Lemma~\ref{lem:symmetric-specialisations} and the expressions for $\widehat{E}^q_{2,\pm}$.
\end{proof}

\subsection{Non-Symmetric Shift Operators}
To determine the specialisations of the four operators
$\mathcal{G}^q_{\pm},\mathcal{E}^q_{2,\pm}$, we consider the
matrix specialisations from the previous sections and translate them into reflection differential operators using
the expressions from Propositions~\ref{prop-inverse-basis-transforms}.

\begin{lemma}
    In the Steinberg basis, the diagonal matrix $(z-z^{-1})z\partial_z$ corresponds to
    \[
        (z-z^{-1})z\partial_z - z(1-s_1),
    \]
    and the matrix $\mqty(0 & 1\\1& 0)$ corresponds to $zs_1$.
\end{lemma}
\begin{proof}
    We use the matrix expressions from Proposition~\ref{prop-inverse-basis-transforms}. For the derivative operator this yields
    \begin{align*}
        &(z-z^{-1})\mqty(1 & z)z\partial_z
        \frac{1}{z-z^{-1}}\mqty(zs_1-z^{-1}\\1-s_1)\\
        =& -\frac{z+z^{-1}}{z-z^{-1}}
        \mqty(1 & z)\mqty(zs_1 - z^{-1}\\1-s_1)
        + \mqty(1 & z)z\partial_z\mqty(zs_1 - z^{-1}\\1-s_1)\\
        =& -z - z^{-1}
        + \mqty(1 & z)z\partial_z\mqty(zs_1-z^{-1}\\1 - s_1).
    \end{align*}
    Now, note that $z\partial_z s_1 = -s_1z\partial_z$
    (as can be quickly verified by applying to a monomial $z^n$), so that
    \[
    z\partial_z \mqty(z s_1 - z^{-1}\\1-s_1)
    = \mqty(zs_1 - zs_1z\partial_z + z^{-1} - z^{-1}z\partial_z\\
    z\partial_z + s_1z\partial_z)
    = \mqty(zs_1+z^{-1}-(1+s_1)\partial_z\\
    (1+s_1)z\partial_z),
    \]
    so that the diagonal matrix $(z-z^{-1})z\partial_z$
    represents
    \begin{align*}
        \mqty(1 & z)\mqty(zs_1+z^{-1}-(1+s_1)\partial_z\\(1+s_1)z\partial_z)
        - z - z^{-1}
        &= \qty(z-z^{-1})z\partial_z + zs_1 + z^{-1} - z - z^{-1}\\
        &= \qty(z-z^{-1})z\partial_z - z(1-s_1).
    \end{align*}
    For the off-diagonal matrix, we obtain
    \[
        \frac{1}{z-z^{-1}}
        \mqty(1 & z)\mqty(0 & 1\\1 & 0)\mqty(zs_1-z^{-1}\\1-s_1)
        = \frac{z^2s_1 - 1 + 1 - s_1}{z-z^{-1}}
        = zs_1.\qedhere
    \]
\end{proof}

\begin{corollary}\label{cor:specialisation-fw-bw}
    The non-symmetric forward and backward shift operators $\frac{\mathcal{G}^q_{\pm}}{q-1}$ specialise
    to the following differential-reflection operators:
    \begin{align*}
        &\frac{z\partial_z}{z-z^{-1}}
        - \frac{z(1-s_1)}{(z-z^{-1})^2}\\
        & (z-z^{-1})z\partial_z + 2k\cdot v_1(z+z^{-1})
+ 4k\cdot v_3 + z^{-1} - zs_1.
    \end{align*}
\end{corollary}

Next, we consider the two contiguous shift operators that
originate from the Koornwinder basis. As such, we again
begin by determining how some common elements act.

\begin{lemma}
    In the Koornwinder basis, the matrices
    $(z-z^{-1})z\partial_z$ and $\mqty(0 & 0\\0 & 1)$
    represent the following differential-reflection operators:
    \[
        (z-z^{-1})z\partial_z -
        \frac{z+z^{-1}}{2}(1-s_1),\qquad \frac{1}{2}(1 - s_1).
    \]
\end{lemma}
\begin{proof}
    Taking the limit of the expressions in Proposition~\ref{prop-inverse-basis-transforms}, we obtain that
    the matrix differential operator $\widehat{S}$ represents
    the differential-reflection operator
    \[
        \frac{1}{2}\mqty(1 & z^{-1}-z)
        \widehat{S}\frac{1}{z-z^{-1}}\mqty(\qty(z-z^{-1})(1+s_1)\\s_1-1).
    \]
    For the differential operator we get
    \begin{align*}
        &\frac{z-z^{-1}}{2}\mqty(1 & z^{-1}-z)
        z\partial \mqty(1 + s_1\\\frac{s_1-1}{z-z^{-1}})\\
        =& \frac{z-z^{-1}}{2}\mqty(1 & z^{-1}-z)
        \mqty((1-s_1)z\partial_z\\
        -\frac{z+z^{-1}}{\qty(z-z^{-1})^2}(s_1-1)
        -\frac{1}{z-z^{-1}}(1+s_1)z\partial_z)\\
        =& (z-z^{-1})z\partial_z -
        \frac{z+z^{-1}}{2}(1-s_1).
    \end{align*}
    For the diagonal matrix we obtain
    \[
        \frac{1}{2}\mqty(1 & z^{-1}-z)
        \mqty(0 & 0 \\0 & 1)
        \mqty(1 + s_1\\\frac{s_1-1}{z-z^{-1}})
        = \frac{1}{2}(1 - s_1).\qedhere
    \]
\end{proof}

Putting this together, we obtain
\begin{corollary}\label{cor:specialisation-contiguous}
    The second type non-symmetric contiguous shift operators
    $\mathcal{E}^q_{2,\pm}$ specialise to
    \begin{align*}
        &-\qty(\frac{1-z^{-1}}{1+z^{-1}}z\partial_z + k_1+k_3-\frac{1}{2} + \frac{1-z^{-1}}{1+z^{-1}}\frac{1-s_1}{z-z^{-1}})\\
        &-\qty(\frac{1+z^{-1}}{1-z^{-1}}z\partial_z + k_2+k_4-\frac{1}{2} - \frac{1+z^{-1}}{1-z^{-1}}\frac{1-s_1}{z-z^{-1}}).
    \end{align*}
\end{corollary}

\begin{proof}
Recall that
\begin{align*}
    &-\qty(\frac{1-z^{-1}}{1+z^{-1}}z\partial_z + k_1+k_3-\frac{1}{2}) - \mqty(0 & 0\\0 & 1)\\
    &= -\qty(\frac{1-z^{-1}}{1+z^{-1}}z\partial_z + k_1+k_3-\frac{1}{2}) + \frac{1-z^{-1}}{1+z^{-1}}\frac{1}{z-z^{-1}}\frac{z+z^{-1}}{2}(1-s_1) - \frac{1}{2}(1-s_1)\\
    &= -\qty(\frac{1-z^{-1}}{1+z^{-1}}z\partial_z + k_1+k_3-\frac{1}{2} + \frac{1-z^{-1}}{1+z^{-1}}\frac{1-s_1}{z-z^{-1}}).
\end{align*}

Similarly, we have
\begin{align*}
     &-\qty(\frac{1+z^{-1}}{1-z^{-1}}z\partial_z + k_2+k_4-\frac{1}{2}) - \mqty(0 & 0\\0 & 1)\\
     &=-\qty(\frac{1+z^{-1}}{1-z^{-1}}z\partial_z + k_2+k_4-\frac{1}{2}) + \frac{1+z^{-1}}{1-z^{-1}}\frac{1}{z-z^{-1}}\frac{z+z^{-1}}{2}(1-s_1) - \frac{1}{2}(1-s_1)\\
     &=-\qty(\frac{1+z^{-1}}{1-z^{-1}}z\partial_z + k_2+k_4-\frac{1}{2} - \frac{1+z^{-1}}{1-z^{-1}}\frac{1-s_1}{z-z^{-1}}).
\end{align*}
\end{proof}

\begin{remark}
    From \cite[\S\S5.1.12--14]{Mac03} it is known that in the limit $q\to1$ the weight $\nabla$ specialises to the weight function of the HO-polynomials. Consequently, the (non-)symmetric Macdonald polynomials specialise to the (non-)symmetric HO-polynomials. In particular, the (non-)symmetric AW-polynomials (with label $k$) specialise to the (non-)symmetric HO-polynomials of type $BC_1$ and the label $(\alpha,\beta)=(k_1+k_3-\frac{1}{2},k_2+k_4-\frac{1}{2})$.

    The operators from Lemma~\ref{lem:symmetric-specialisations} are therefore shift operators for the symmetric HO-polynomials of type $BC_1$. Similarly, the operators from Corollary~\ref{cor:specialisation-fw-bw} and Corollary~\ref{cor:specialisation-contiguous} are shift operators (in the paradigm of differential-reflection operators of Heckman and Opdam) for the non-symmetric HO-polynomials of type $BC_1$, which agree with the operators found in \cite{vHvP24}.
\end{remark}

\section*{Acknowledgement}
The authors would like to thank Maarten van Pruijssen and Erik Koelink
for valuable feedback and relevant questions. The research of P.S. is funded
by the Dutch Science Council (NWO) grant \texttt{OCENW.M20.108}.

\appendix
\section{Coefficient Calculations}
In this section we shall compute the next-to-leading-term
coefficients of the non-symmetric AW-polynomials,
i.e. the coefficient of $z^{-n}$ of $E_{n+1}(z)$ and the coefficient
of $z^n$ in $E_{-n}(z)$. This will be relevant to prove
Proposition~\ref{prop-relations-E-P} in the next appendix.
We define the coefficients
$c_{n+1,k}$ and $\widetilde{c}_{n,k}$ by
\begin{align}
    E_{-n,k}(z) &= z^{-n} + \widetilde{c}_{n,k} z^n + \lot\\
    E_{n+1,k}(z) &= z^{n+1} + c_{n+1,k}z^{-n}+\lot.
\end{align}

In order to do this, we introduce the double-affine Hecke algebra of type $(C_1^\vee, C_1)$ and note that the non-symmetric AW-polynomials can be characterised as being eigenfunctions of $q$-difference-reflection operators. We will make use of the explicit presentation given in \cite{Koo07}, however we will keep the conventions from \cite{Mac03} and \cite[Definition~2.28]{Sch23}.

\begin{definition}\label{def-daha}
The double affine Hecke algebra $\tilde{\mathfrak{H}}$ of type $(C_1^\vee, C_1)$ is generated $T_0$, $T_1$, $Z$, and $Z^{-1}$, subject to the relations $ZZ^{-1} = 1 = Z^{-1} Z$ and
\begin{align*}
(T_0 - \tau_0)\qty(T_0+\tau_0^{-1}) &= 0\\
(T_1 - \tau_1)\qty(T_1 + \tau_1^{-1}) &= 0\\
\qty(T_0Z^{-1} - \widetilde{\tau}_0^{-1}q^{-\frac{1}{2}})
\qty(T_0Z^{-1} + \widetilde{\tau}_0q^{-\frac{1}{2}}) &= 0\\
\qty(T_1Z - \widetilde{\tau}_1^{-1})\qty(T_1Z+\widetilde{\tau}_1)
&= 0,
\end{align*}
where
\[
(a,b,c,d) = \qty(\tau_1\widetilde{\tau}_1,
-\frac{\tau_1}{\widetilde{\tau}_1}, q^{\frac{1}{2}}\tau_0\widetilde{\tau}_0,
-q^{\frac{1}{2}}\frac{\tau_0}{\widetilde{\tau}_0}).
\]
\end{definition}
In \cite{Mac03}, the parameters $\widetilde{\tau}_i$ are called
$\tau'_i$, which can lead to confusion with the labelling
$k'$ dual to $k$ and the coefficients that are derived from it.

The double affine Hecke algebra  $\tilde{\mathfrak{H}}$ acts faithfully on $\mathcal{A}$ in a representation called the \emph{basic representation}, by taking $(Zf)(z) = zf(z)$ and
\begin{align*}
T_0 &= \tau_0 s_0 + \frac{\tau_0 - \tau_0^{-1} + \qty(\widetilde{\tau}_0 - \widetilde{\tau}_0^{-1})q^{\frac{1}{2}}Z^{-1}}{1 - q Z^{-2}}(1 - s_0)\\
T_1 &= \tau_1 s_1 + \frac{\tau_1 - \tau_1^{-1} + \qty(\widetilde{\tau}_1 - \widetilde{\tau}_1^{-1})Z}{1 - Z^2}(1 - s_1),
\end{align*}
where $s_0z^n := q^nz^{-n}$.

A crucial role is played by the Cherednik operator $Y = T_0 T_1$, especially for the non-symmetric AW-polynomials:

\begin{lemma}[{\cite[Equation~5.2.2]{Mac03}}]\label{lem:action-Y}
    If $n\in\N_0$, then
    \begin{align*}
        YE_{-n,k} &= q^{n+k_1'}E_{-n,k}\\
        YE_{n+1,k} &= q^{-(n+1)-k_1'}E_{n+1,k},
    \end{align*}
    with $k'_1 = k\cdot v_1$.
\end{lemma}

In order to compute these next-to-leading order terms we are going
to make use of creation and annihilation operators as defined in
\cite[\S5.10]{Mac03}.

\begin{definition}
    The two creation operators are
    \begin{align*}
        \bm{\alpha'_0} := 
        T_1^{-1}Z^{-1}
        - \bm{b_0'}\qty(q^{\frac{1}{2}}Y)\\
        \bm{\alpha'_1}
        := T_1 - \bm{b'_1}\qty(Y^{-1}),
    \end{align*}
    where
    \begin{align*}
        \bm{b'_0}(x)
        &= \frac{\widetilde{\tau}_1 - \widetilde{\tau}_1^{-1} + (\widetilde{\tau}_0-\widetilde{\tau}_0^{-1})x}{1-x^2}\\
        \bm{b'_1}(x)
        &= \frac{\tau_1-\tau_1^{-1} + \qty(\tau_0-\tau_0^{-1})x}{1-x^2}.
    \end{align*}
\end{definition}

\begin{lemma}\label{lem:action-alphas}
    For $n\in\N_0$, the creation operators act as follows:
    \begin{align*}
        \bm{\alpha'_0}E_{-n,k} &= \tau_1 E_{n+1,k}\\
        \bm{\alpha'_1}E_{n+1,k} &= \tau_1^{-1} E_{-n-1,k}.
    \end{align*}
\end{lemma}
\begin{proof}
    Follows from \cite[\S\S5.10.7--8]{Mac03}.
\end{proof}

\begin{lemma}[{\cite[\S4.3.21]{Mac03}}]\label{lem:action-T1}
    For $n\in\Z$ we have
    \[
        T_1 z^n =
        \begin{cases}
            \tau_1 z^{-n} & n=0\\
            \tau_1^{-1} z^{-n} - \qty(\widetilde{\tau}_1 - \widetilde{\tau}_1^{-1})\qty(z^{1-n} + z^{n-1}) & n>0\\
            \tau_1 z^{-n}
            + \qty(\tau_1-\tau_1^{-1})z^n
            + \qty(\widetilde{\tau}_1-\widetilde{\tau}_1^{-1})(z^{n+1} + z^{-n-1}) & n<0
        \end{cases}.
    \]
\end{lemma}

\begin{proposition}\label{prop:NLO-coefficients}
    For $n\in\N_0$ we have
    \begin{align*}
        c_{n+1,k} &=-\frac{cq^n + dq^n - acdq^{2n}-bcdq^{2n}}{1-abcdq^{2n}} \\
        \widetilde{c}_{n,k} &=\frac{1 + ab - abcdq^{n-1} - abq^n}{1-abcdq^{2n-1}}.
    \end{align*}
\end{proposition}
\begin{proof}
    Since we will only be dealing with the labelling $k$ in
    this proof, we shall drop its mention.
    
    We are first proving the relations concerning $\widetilde{c}_n$. For $n=0$, we obtain $\widetilde{c}_n=1$, which
    is the coefficient of $z^n$ in $E_{-n}$, even though the shorthand
    \[
        E_{-n}(z) = z^{-n} + \widetilde{c}_nz^n + \lot
    \]
    is clearly wrong in this case.
    Assume $n>0$, then
    \[
        E_{-n}(z) = \tau_1 \bm{\alpha'_1} E_{n}(z)
    \]
    by Lemma~\ref{lem:action-alphas}, where $E_{n}(z) = z^n + \lot$ with no other terms in the $W_0$-orbit of $z^n$. Then
    \[
        E_{-n}(z) = \tau_1\qty(T_1 - \bm{b'_1}\qty(q^{n+k_1'}))(z^n + \lot)
    \]
    by definition and Lemma~\ref{lem:action-Y}. By Lemma~\ref{lem:action-T1} and the fact that $E_{n,k}(z)$ has only one
    monomial in its outermost $W_0$-orbit, we obtain
    \[
        E_{-n}(z) = z^{-n} - \tau_1\bm{b'_1}\qty(q^{n+k'_1})z^n + \lot,
    \]
    hence
    \begin{align*}
        \widetilde{c}_n &= -\tau_1\bm{b'_1}\qty(q^{n+k'_1})\\
        &= - \frac{\tau_1^2 - 1 + \qty(\tau_0\tau_1 - \tau_0^{-1}\tau_1)q^{n+k'_1}}{1-q^{2n+2k'_1}}
        = \frac{1 + ab - abcdq^{n-1} - abq^n}{1-abcdq^{2n-1}}.
    \end{align*}
    Now for $c_{n+1}$. Note that
    \begin{align*}
        E_{n+1}(z) &= \tau_1^{-1}\bm{\alpha'_0} E_{-n}(z)\\
        &=\tau_1^{-1} \qty(T_1^{-1}Z^{-1} - \bm{b'_0}\qty(q^{\frac{1}{2}+n+k'_1})) \qty(z^{-n} + \widetilde{c}_nz^n + \widetilde{d}_nz^{1-n} + \lot)
    \end{align*}
    by Lemmas~\ref{lem:action-alphas} and \ref{lem:action-Y}, where
    $\widetilde{d}_n$ is another unknown coefficient that will turn out to be irrelevant to our computation. Using
    Lemma~\ref{lem:action-T1}, this then equals
    \begin{align*}
        E_{n+1}(z) &= \tau_1^{-1}T_1^{-1}\qty(z^{-n-1} + \widetilde{d}_n z^{-n}+0z^n+\lot)\\
        &\qquad- \tau_1^{-1}\bm{b'_0}\qty(q^{\frac{1}{2}+n+k'_1})\qty(z^{-n} + \lot)\\
        &= \tau_1^{-1}\qty(\tau_1 z^{n+1} + \qty(\widetilde{\tau}_1-\widetilde{\tau}_1^{-1})\qty(z^{-n} + z^n)
        +\tau_1\widetilde{d}_nz^n)\\
        &\qquad- \tau_1^{-1}\bm{b'_0}\qty(q^{\frac{1}{2}+n+k'_1})z^{-n} + \lot\\
        &= z^{n+1}
        + \tau_1^{-1}\qty(\widetilde{\tau}_1-\widetilde{\tau}_1^{-1} 
        - \bm{b'_0}\qty(q^{\frac{1}{2}+n+k'_1}))z^{-n} + \lot,
    \end{align*}
    showing that
    \begin{align*}
        c_{n+1} &= \tau_1^{-1}\qty(\widetilde{\tau}_1-\widetilde{\tau}_1^{-1}-\bm{b'_0}\qty(q^{\frac{1}{2}+n+k'_1}))\\
        &= \tau_1^{-1}\frac{-\qty(\widetilde{\tau}_1-\widetilde{\tau}_1^{-1})q^{1+2n+2k'_1} - \qty(\widetilde{\tau}_0-\widetilde{\tau}_0^{-1})q^{\frac{1}{2}+n+k'_1}}{1-q^{1+2n+2k'_1}}\\
        &= \frac{\qty(b^{-1}+a^{-1})abcdq^{2n} - (c+d)q^n}{1-abcdq^{2n}}\\
        &= -\frac{cq^n + dq^n - acdq^{2n}-bcdq^{2n}}{1-abcdq^{2n}}.\qedhere
    \end{align*}
\end{proof}

\section{Relating $\mathbb{E}_{s,m,k}$ and $\mathbb{P}_{s,m,k}$}\label{sec-proof-E-P}

\begin{lemma}
    We can expand $\mathbb{E}_{\st,m,k}$ as
    \[
        \mathbb{E}_{\st,m,k}(z) =  C_{\st,k}\qty(q^{\frac{m}{2}})\qty(z+z^{-1})^m + \lot,
    \]
    where
    \[
        C_{\st,k}\qty(q^{\frac{m}{2}}) = \mqty(1 & c_{m+1}\\0 & 1).
    \]
\end{lemma}
\begin{proof}
    In this proof we suppress the dependence on the labelling $k$. We write
    \begin{align*}
        E_{-m}(z) &= z^{-m} + \alpha z^m + \lot\\
        E_{m+1}(z) &= z^{m+1} + \beta z^{-m} + \gamma z^m + \lot,
    \end{align*}
    where $\beta = c_{m+1}$. Using Proposition~\ref{prop-inverse-basis-transforms}, we can then compute
    \begin{align*}
        \mathcal{B}_{\st}^{-1}(E_{-m})_1(z) &= \frac{z(z^m + \alpha z^{-m}) - z^{-1}(z^{-m} + \alpha z^m) + \lot}{z-z^{-1}}\\
        &= \qty(z + z^{-1})^m + \lot\\
        \mathcal{B}_{\st}^{-1}(E_{-m})_2(z) &= \frac{z^m + \alpha z^{-m} - z^{-m} - \alpha z^m + \lot}{z-z^{-1}}\\
        &= 0\qty(z + z^{-1})^m + \lot,
    \end{align*}
    and
    \begin{align*}
        \mathcal{B}_{\st}^{-1}(E_{m + 1})_1(z) &= \frac{z(z^{-m-1} + \beta z^m + \gamma z^{-m}) - z^{-1}(z^{m+1} + \beta z^{-m} + \gamma z^m)+\lot}{z-z^{-1}} \\
        &= \beta \qty(z + z^{-1})^m + \lot\\
        \mathcal{B}_{\st}^{-1}(E_{m + 1})_2(z) &= \frac{z^{m + 1} + \beta z^{-m} + \gamma z^m - z^{-m-1} - \beta z^m - \gamma z^{-m} + \lot}{z - z^{-1}} \\
        &= \qty(z + z^{-1})^m + \lot.
    \end{align*}
    Lastly, we recall from Proposition~\ref{prop:NLO-coefficients} that
    $c_{m+1}$ only depends on $m$ through $q^{\frac{m}{2}}$, whence it makes
    sense to define the matrix $C_{\st,k}(T)$ the way we defined it.
\end{proof}

This allows us to prove the Steinberg half.

\begin{proof}[Proof of Proposition~\ref{prop-relations-E-P}, Steinberg Basis]
    As a consequence of the properties of matrix-valued inner
    products, families of MVOP can be multiplied on the
    right by arbitrary (constant) invertible matrices to produce
    new families of MVOP. Thus in particular,
    \[
        \qty(\mathbb{P}_{\st,k,m} V_k^{-1}C_{\st,k}\qty(q^{\frac{m}{2}}))_{m\in\N_0}
    \]
    is a family of MVOP (with respect to the weight $V_k^T \mathcal{W}_{\st,k} V_k^*$),
    whose leading coefficients are $V_k^{-1}C_{\st,k}\qty(q^{\frac{m}{2}})$. Since the leading coefficients of $(V_k^{-1}\mathbb{E}_{\st,k,m})_{m\in\N_0}$ are
    also $V_k^{-1}C_{\st,k}\qty(q^{\frac{m}{2}})$, they are the same family of MVOP. Rearranging the terms yields the claim.
\end{proof}

The exact same proof is not applicable to the Koornwinder basis as neither
$(\mathbb{E}_{\ko,m,k})_{m\in\N_0}$ nor $(\mathbb{P}_{\ko,m,k})_{m\in\N_0}$
is a family of MVOP since their leading terms (as polynomials in $z$)
are singular matrices. This, however, can be remedied by using a different
monomial basis than $\mqty(z^m + z^{-m} & 0\\0 & z^m + z^{-m})_{m\in\N_0}$.

\begin{definition}
    For $m\in\N_0$ define the $m$-th \emph{degree-distorted monomial}
    \[
        M_m := \mqty(z^m + z^{-m} & 0\\0 & z^{m-1} + z^{1-m})
    \]
    for $m>0$ and $M_0:= \mqty(1 & 0\\0 & 0)$. This is a generating system
    of $\mathcal{A}_0^{2\times 2}$ as a right $K^{2\times2}$-module, whose
    only relation is
    \[
        M_0\mqty(0 & 0\\\ast & \ast) = 0.
    \]
    We equip this generating system with the following total order:
    \[
        M_0 < M_1<\cdots,
    \]
    which allows us to define notions such as $M$-leading terms (up to the
    relation for constant polynomials), and two expressions being equal
    up to $M$-lower-order terms.
\end{definition}

\begin{lemma}\label{lem:koornwinder-initial-coefficient}
    With respect to the monomial generating system $M$, the family
    $\qty(\widetilde{\mathbb{E}}_m)_{m\in\N_0}$ is a family of
    MVOP. In particular, if we expand $E_{-m}(z) = z^{-n} + \widetilde{c}_n z^n + \lot$, we have
    \[
        \mathbb{E}_{\ko,m,k} = M_n C_{\ko,k}\qty(q^{\frac{m}{2}}) + M-\lot,
    \]
    with
    \[
        C_{\ko,k}\qty(q^{\frac{m}{2}}) = \frac{1}{ab-1}\mqty(ab - \widetilde{c}_n & -1\\\widetilde{c}_n - 1 & 1),
    \]
    for all $m\in\N$, and
    \[
        \mathbb{E}_{\ko,0,k} = M_0 \mqty(1 & 1\\0 & 0).
    \]
\end{lemma}
\begin{proof}
    By Lemma~\ref{lem-es-orthogonal}, they 
    are orthogonal. So it remains to show that the
    $M$-leading coefficients are regular. We first consider $m>0$, write
    $\alpha = \widetilde{c}_m$, and substitute
    $E_m(z) = z^m + \lot$ and $E_{-m}(z) = z^{-m} + \alpha z^m + \lot$ into the formulas in Proposition~\ref{prop-inverse-basis-transforms} to obtain
    \begin{align*}
        \mathcal{B}_{\ko,k}^{-1}(E_{-m})_1(z) &=
        \frac{1}{ab-1}\frac{(1-s_1)\qty(abz + z^{-1}-a-b)\qty(z^m + \alpha z^{-m}) + \lot}{z-z^{-1}}\\
        &= \frac{1}{ab-1}\frac{(ab-\alpha)\qty(z^{m+1}-z^{-m-1}) + \lot}{z-z^{-1}}\\
        &= \frac{ab-\alpha}{ab-1} \qty(z^m + z^{-m}) + \lot\\
        \mathcal{B}_{\ko,k}^{-1}(E_{-m})_2(z) &=
        \frac{1}{ab-1}\frac{z^{-m} + \alpha z^m - z^m - \alpha z^{-m} + \lot}{z-z^{-1}}\\
        &= \frac{\alpha - 1}{ab-1}\qty(z^{m-1} - z^{1-m}) + \lot\\
        \mathcal{B}_{\ko,k}^{-1}(E_m)_1(z) &=
        \frac{1}{ab-1}\frac{(1-s_1)\qty(abz+z^{-1}-a-b)z^{-m}
        + \lot}{z-z^{-1}}\\
        &= \frac{1}{ab-1}\frac{-z^{m+1} + z^{-m-1}+\lot}{z-z^{-1}}\\
        &= -\frac{1}{ab-1}\qty(z^m+z^{-m})\\
        \mathcal{B}_{\ko,k}^{-1}(E_m)_2(z) &=
        \frac{1}{ab-1}\frac{z^m - z^{-m} + \lot}{z-z^{-1}}\\
        &= \frac{1}{ab-1}\qty(z^{m-1} + z^{1-m}) + \lot.
    \end{align*}
    This shows that
    \[
        \mathbb{E}_{\ko,m,k} =
        \frac{1}{ab-1}M \mqty(ab-\alpha & -1\\\alpha-1 & 1) + M-\lot.
    \]
    Since the leading term has determinant
    $\frac{1}{ab-1}$, it is regular. From Proposition~\ref{prop:NLO-coefficients} we see that $\widetilde{c}_m=\alpha$ only depends on
    $m$ through a rational dependence on $q^{\frac{m}{2}}$. Consequently,
    we can indeed define $C_{\ko,k}(T)$ in the manner claimed. For $m=0$, we have
    \[
        \mathbb{E}_{\ko,m,k} = \mqty(1 & 1\\0 & 0)
        = \mathbb{P}_{\ko,m,k}\mqty(1 & 1\\0 & 0).\qedhere
    \]
\end{proof}

Note moreover that the $(\mathbb{P}_{\ko,m,k})_{m\in\N_0}$ are also
a family of MVOP whose leading coefficients with respect to $M$ are
the identity matrix. This allows us to show the Koornwinder part of
Proposition~\ref{prop-relations-E-P}.

\begin{proof}[Proof of Proposition~\ref{prop-relations-E-P}, Koornwinder Basis]
    Throughout, we are suppressing the index $k$ for the labelling. Note
    that for $m=0$, the claim is clear as
    \[
        \mathbb{E}_{\ko,m}=\mqty(1 & 1\\0 & 0)
        = \mqty(1 & 0\\0 & 0)\mqty(1 & 1\\x & y)
        = \mathbb{P}_{\ko,m}\mqty(1 & 1\\x & y)
    \]
    for all $x,y\in K$.
    If $m\in\N$, then there are matrices $A_0,\dots,A_m\in K^{2\times2}$ such that
    \[
        \mathbb{E}_{\ko,m} C_{\ko}\qty(q^{\frac{m}{2}})^{-1} = \mathbb{P}_{\ko,m} A_m + \cdots + \mathbb{P}_{\ko,0} A_0.
    \]
    Since the $M$-leading term of both sides is $M_m$, we have $A_m=1$. Then we have
    \begin{equation*}
        0 = \widehat{H}_{\ko}\qty(\mathbb{E}_{\ko,m} C_{\ko}\qty(q^{\frac{m}{2}})^{-1} - \mathbb{P}_{\ko,m}, \mathbb{P}_{\ko,i}) = \widehat{H}_{\ko}(\mathbb{P}_{\ko,i},\mathbb{P}_{\ko,i}) A_i
    \end{equation*}
    for all $0 \leq i \leq m - 1$. 
    
    For $i>0$ we know that $\widehat{H}_{\ko}(\mathbb{P}_{\ko,i}, \mathbb{P}_{\ko,i})$ is regular, so we find that $A_m = 0$ for all $1 \leq m \leq n - 1$.

    For $i=0$ we have
    \[
        \widehat{H}_{\ko}(\mathbb{P}_{\ko,0},\mathbb{P}_{\ko,0})
        = \mqty(\mu & 0\\0 & 0)
    \]
    for some $\mu\ne0$. Thus
    \begin{equation*}
    \mqty(\mu & 0 \\ 0 & 0) A_0 = 0.
    \end{equation*}
    Without loss of generality, since any potential other non-zero entry will be annihilated by the zero entries of $\widetilde{\mathbb{P}}_0$, we may take
    \begin{equation*}
    A_0 = \mqty(x & y \\ 0 & 0),
    \end{equation*}
    so that
    \[
        0 = \mqty(\mu & 0\\0 & 0)A_0 
        = \mqty(\mu & 0\\0 & 0)\mqty(x & y\\0 & 0)
        = \mqty(\mu x & \mu y\\0 & 0),
    \]
    which shows that $A_0=0$. The claim follows.
\end{proof}

\printbibliography

\end{document}